\documentclass[11pt]{amsart}

\usepackage[T1]{fontenc}
\usepackage[latin1]{inputenc}
\usepackage{amsmath}
\usepackage{amssymb}
\usepackage{amsthm}
\usepackage{dsfont}
\usepackage{color}
\usepackage{geometry}
\theoremstyle{plain}\newtheorem{theo}{Theorem}[section]
\theoremstyle{plain}\newtheorem{cor}[theo]{Corollary}
\theoremstyle{defn}\newtheorem{rem}[theo]{Remark}
\theoremstyle{plain}
\theoremstyle{plain}\newtheorem{lem}[theo]{Lemma}
\theoremstyle{plain}
\theoremstyle{plain}\newtheorem{prop}[theo]{Proposition}
\theoremstyle{defn}
\theoremstyle{defn}\newtheorem{hypo}[theo]{Assumption}
\theoremstyle{defn}\newtheorem{nota}[theo]{Notation}

\def\cA{\mathcal{A}}
\def\cB{\mathcal{B}}
\def\cC{\mathcal{C}}

\def\cG{\mathcal{G}}

\def\bO{\boldsymbol{0}}
\def\uomega{\underline{\omega}}

\def\cK{\mathcal{K}}
\def\cL{\mathcal{L}}

\def\cP{\mathcal{P}}
\def\cQ{\mathcal{Q}}

\def\cS{\mathcal{S}}

\def\cV{\mathcal{V}}
\def\cW{\mathcal{W}}

\def\cZ{\mathcal{Z}}

\newcommand{\vf}{\varphi}

\newcommand{\ep}{\epsilon}
\newcommand{\ve}{\varepsilon}

\DeclareMathOperator{\cov}{Cov}

\begin{document}

\title[Local limit theorem for randomly deforming billiards]{Local limit theorem for randomly deforming billiards}

\author{Mark F.~Demers}
\address{Department of Mathematics, Fairfield University, 1073 North Benson Road, Fairfield CT 06824, USA}
\email{mdemers@fairfield.edu}

\author{Fran\c{c}oise P\`ene}
\address{Institut Universitaire de France and Universit\'e de Brest,
UMR CNRS 6205, Laboratoire de Math\'ematique de Bretagne Atlantique,
6 avenue Le Gorgeu, 29238 Brest cedex, France}
\email{francoise.pene@univ-brest.fr}

\author{Hong-Kun Zhang}
\address{Department of Math \& Stat. University of Massachusetts Amherst, MA, 01003}
\email{hongkun@math.umass.edu}

\maketitle
\keywords{}

\begin{abstract}
We study limit theorems in the context of random perturbations of dispersing billiards in finite and infinite measure.
In the context of a planar periodic Lorentz gas with finite horizon, we consider random perturbations in the form of
movements and deformations of scatterers.  We prove a Central Limit Theorem for the cell index of planar motion, as well
as a mixing Local Limit Theorem for the cell index   with piecewise H\"older continuous observables.  In the context of the infinite measure random
system, we prove limit theorems regarding visits to new obstacles and self-intersections, as well as decorrelation
estimates.  The main tool we use is the adaptation of anisotropic Banach spaces to the random setting.
\end{abstract}
\centerline{AMS classification numbers: 37D50, 37A25}

\tableofcontents
\printindex
\section*{Introduction}
The Lorentz process is a physically  interesting
mechanical system modeled  by mathematical
billiards with chaotic behavior.  Introduced by Sinai in \cite{Sin}, it has been 
studied extensively by many authors, see \cite{BSC90, BSC91, chernovbook} and other related references.
 It is the deterministic motion
of a point particle starting from a random phase point and undergoing specular
reflections on the boundaries of strictly convex scatterers. Throughout this paper
we will  consider a $\mathbb{Z}^2$-periodic random configuration of scatterers, with finite horizon.  The diffusion limit of the planar Lorentz process can be described by a
Wiener process \cite{BSC91}, and is thus closely related to the  Central Limit Theorem (CLT) and  Local Limit Theorem (LLT).

The history of the LLT goes back to the historic De Moivre Laplace theorem for independent identically distributed (iid) Bernoulli random variables.
It has then been generalized in many contexts. 
The CLT appears as a consequence of the LLT.
In the context of dynamical systems, the first LLT was established
by Guivarc'h and Hardy for subshifts of finite type \cite{GH}.
The method they used, also used by Nagaev in \cite{Nag}, was based on perturbations of 
an associated transfer operator and 
has since been used for many expanding and hyperbolic dynamical systems. 
This method is now often called the Nagaev-Guivarc'h method.
For the Sinai billiard (with fixed scatterers), the LLT was proved by Sz\'asz and Varj\'u in \cite{SV}
using Young towers and the Nagaev-Guivarc'h method.  Also using Young towers, P\`ene established and used in \cite{soazRange,soazPAPA,soazSelfInter} some precise versions of the LLT
to prove further limit theorems for the Sinai billiard (see also her works with Saussol \cite{FPBS09} and with Thomine \cite{DamienSoaz} for other applications of the LLT). 

The goal of this article is to prove the LLT, as well as several of its applications, in the context of randomly deforming scatterers in a dispersing Lorentz gas with finite horizon. 
In this context the use of Young towers does not appear very adequate, since a different tower is associated to every
different $\mathbb Z^2$-periodic configuration of scatterers. 
It is therefore much more natural to work directly with the billiard transformations since these transformations act on the same space $\bar M_0$ and preserve the same measure. To this end, we will work with the spaces
considered  in \cite{MarkHongKun2011,MarkHongKun2013,MarkHongKun2014}, which are spaces 
$\mathcal B,\mathcal B_w$ made of distributions instead of being spaces of functions
contained in $L^p$ for some $p>1$ as in \cite{GH,SV}. This will complicate our study. 
One advantage of the approach used by
Demers and Zhang is that the Banach spaces they construct in 
\cite{MarkHongKun2013} are the same
for natural families of billiard transformations.

Since we are interested in random iterations of billiard transformations,
we will consider the full random billiard system corresponding to the skew product transformation which takes in account both the billiard configuration (position and speed) and the randomness of the configuration of scatterers.
Let us mention that Aimino, Nicol and Vaienti established in \cite{ANV} 
an LLT (together with other limit theorems) for random iterations of expanding dynamical systems. 
Their approach was based on the Nagaev-Guivarc'h method applied to the restriction of the transfer operator of the full random system
to functions depending only on the phase space coordinate (and not on the random coordinate).
The advantage of their method is that they worked on a simple Banach space (in which the randomness of the transformations is not taken into account).
But the disadvantage is that they had to reprove for this restricted operator theorems that were already known for transfer operators.
In the present paper, we apply directly the Nagaev-Guivarc'h method to the transfer operator of the full random system
acting on suitable Banach spaces $\widetilde{\mathcal B},\widetilde{\mathcal B}_w$ which are easily defined using
${\mathcal B},{\mathcal B}_w$. As a consequence, our results apply to observables that may depend on
both the position and speed of the billiard, as well as the random coordinate. \\

This article is organized as follows.
In Section \ref{nota}, we specify our assumptions and notation.
In Section \ref{mainresults}, we state our main limit theorems: LLT, asymptotic estimate of the return time to the initial scatterer, 
asymptotic behavior of the number of self-intersections, annealed and
quenched limit theorem for a random billiard in random scenery,
limit theorems for some ergodic sums of the planar random billiard
(in infinite measure), mixing and decorrelation for the planar random billiard (in infinite measure).
In Section \ref{sec:ops}, we study the spectral properties of the 
transfer operator of the full random system. Section \ref{proofs}
is devoted to the proof of our main results under general spectral assumptions.

\section{Notation and assumptions}\label{nota}

\subsection{Deterministic billiard systems.}\label{sec:bar T}
Let $I\ge 1$ and let $O_1,...,O_I$ be $I$ convex open subsets of $\mathbb R^2$, having $\cC^3$
boundary with strictly positive curvature,
and such that the closure of the sets $(U_{i,\ell}:=\ell+O_i)_{i=1,...,I;\ \ell\in\mathbb Z^2}$ are pairwise disjoint. 
We consider the $\mathbb Z^2$-periodic
billiard table $Q:=\mathbb R^2\setminus \bigcup_{\ell\in\mathbb Z^2}\bigcup_{i=1}^I(U_{i,\ell})$. 
We assume moreover that every line meets $\partial Q$ (i.e. that the horizon is finite).
We are interested in the behavior of a point particle moving in $Q$
at unit speed, going straight inside $Q$, and reflecting elastically off $\partial Q$
(the reflected direction being the symmetric of the incident one with respect to the normal line to $Q$ at the reflection point).

We consider the {\bf planar billiard system} $(M_0,\mu_0,T_0)$ modeling the behavior of the point particle at reflection times. A configuration is given by a pair
$(q,\vec v)\in M_0$ representing position and velocity, and corresponding to a reflected vector off
$\partial Q$, with
$$M_0:=\{(q,\vec v)\in\mathbb R^2\times\mathbb R^2 :\
           q\in\partial Q,\
         \Vert\vec v\Vert=1,\ \langle \vec n(q),\vec v\rangle>0  \},$$
where $\vec n(q)$ is the unit vector, normal to $\partial Q$ at $q$ and directed into $Q$.
The transformation $T_0$ maps a reflected vector to
the reflected vector at the next reflection time.
This transformation preserves the measure $\mu_0$ given by $d\mu_0 = \tilde c \cos\varphi \, dr\, d\varphi$ (where $r$ is the
parametrized arclength coordinate on $\partial Q$ corresponding to $q$ and
$\varphi$ is the algebraic measure of the angle
$\widehat{(\vec n(q),\vec v)}$ and where $\tilde c=1/(2\sum_{i=1}^I|\partial O_i|)$, the reason for the choice of $\tilde c$ will be clear in a few lines).

For every $i\in\{1,...,I\}$ and every $\ell\in\mathbb Z^2$,
we define $M_{i,\ell}:=\{(q,\vec v)\in M_0\ : \ q\in \partial U_{i,\ell}\}$
for the set of reflected vectors based on the obstacle $U_{i,\ell}$.
For every $\ell\in\mathbb Z^2$, we will call an {\bf $\boldsymbol \ell$-cell} the set $\bar M_\ell:=\bigcup_{i=1}^IM_{i,\ell}$ .

Identifying the boundary of each scatterer $\partial O_i$ with a circle $\mathbb{S}_i$ of length $|\partial O_i|$, we define
$\bar M_0 := \cup_{i=1}^I \mathbb{S}_i \times [-\pi/2, \pi/2]$.  Thus $\bar M_0$ is a parametrization of $\bar M_{(0,0)}$ in the coordinates
$(r,\vf)$ introduced above. 
Note that may configurations of obstacles $O_i$ result in the same parametrized space $\bar M_0$.
We shall exploit this fact when defining the classes of random perturbations that we shall consider.  

Because of its $\mathbb Z^2$ periodicity, the planar billiard system can
be identified with a $\mathbb Z^2$-cylindrical extension over a dynamical system $(\bar M_0,\bar\mu_0,\bar T_0)$.
Indeed, using the notation $x+\ell=(q+\ell,\vec v)$ for every $x=(q,\vec v)\in M_0$ and every $\ell\in\mathbb Z^2$, we observe that
there exists a transformation $\bar T_0:\bar M_0\rightarrow\bar M_0$
(corresponding to the billiard map modulo $\mathbb Z^2$)
and a function $\Phi_0:\bar M_0\rightarrow \mathbb Z^2$ called a {\bf cell-change}) such that
$$T_0(x+\ell)=\bar T_0 (x)+\ell+\Phi_0(x)\, .$$
This transformation $\bar T_0$
preserves the probability measure $\bar\mu_0(\cdot):=\mu_0(\cdot\cap\bar M_0)$ (the fact that $\bar\mu_0$ is a probability comes from our choice for the normalizing constant $\tilde c$).

In the following, identifying a couple $(x,\ell)\in\bar M_0\times\mathbb Z^2$ with $x+\ell\in M_0$, we identify $(M_0,\mu_0,T_0)$ with the $\mathbb Z^2$-cylindrical extension of  $(\bar M_0,\bar\mu_0,\bar T_0)$ by $\Phi_0$, i.e. we identify $M_0$ with $\bar M_0\times\mathbb Z^2$, $\mu_0$ with
$\bar\mu_0\otimes \mathfrak m$, where $\mathfrak m:=\sum_{k\in\mathbb Z^2}\delta_k$ is the counting measure on $\mathbb Z^2$.

\subsection{Random perturbations of the initial billiard system}
\label{sec:rand pert}

Before describing the random perturbations we shall consider, we describe a class of maps
$\bar {\mathcal{F}}$ on $\bar M_0$
with uniform properties from which we will draw random sequences of maps.
The class $\bar {\mathcal{F}}$ we will use is a slightly simplified version of the one introduced in
\cite{MarkHongKun2013}.  The perturbations in \cite{MarkHongKun2013} allowed billiards
with infinite horizon, while for the present work we will assume a finite horizon condition and that the invariant measure 
is absolutely continuous with respect to the Lebesgue measure,
which
simplifies several of our assumptions.

We consider a probability space $(E,\mathfrak T,\eta)$ containing 0 and
 a family $(T_\omega)_{\omega\in E}$ of $\mathbb Z^2$-periodic planar Sinai billiard systems (with finite horizon) defined on $M$, the quotient billiard maps (modulo $\mathbb Z^2$ for the position) $\bar T_\omega$ of which are in $ \bar{\mathcal{F}}$, and below 
 we will choose $\bar{\mathcal{F}}_{\vartheta_0}(\bar T_0)$ as a small
$\vartheta_0$-neighbourhood of our original map $\bar T_0$, see (\ref{defnFtheta}).

For any $\uomega\in E^{\mathbb{N}}$, we will consider random iterations $T^k_{\uomega}:=T_{\omega_{k-1}}\circ...\circ T_{\omega_0}$, where $\omega_k$
are chosen independently with respect to $\eta$. Here $\uomega=(\omega_k)_{k\geq 0}$, and 
$T_{\omega_k}\in {\mathcal{F}}$, for any $k\geq 0$, where ${\mathcal{F}}$ is a collection of $\mathbb{Z}^2$ extensions of $\bar {\mathcal{F}}$. 
This will be formalized below.
In our model, the modification of environment is applied during the reflection time of the particle; the particle stays on the obstacle and moves with it during the modification of the billiard system.
At its $k$-th reflection time, the particle arrives on an obstacle in an environment parametrized by $\omega_{k-1}$, but when it leaves it sees the environment $\omega_{k}$.

We identify $(M_0,\mu_0,T_\omega)$ with the $\mathbb Z^2$-extension of $(\bar M_0,\bar\mu_0,\bar T_\omega)$ by some function
$\Phi_\omega : \bar M_0\rightarrow\mathbb Z^2$ which is constant on each
connected component of continuity of $\bar T_\omega$. We define {\bf the random billiard system} $(\bar M,
\bar\mu,\bar T)$, corresponding to random iterations of maps in $\bar {\mathcal{F}}$, 
by setting:
$$\bar M:=\bar M_0\times E^{\mathbb N},\,\,\,\,\,\,\,\,\bar\mu:=\bar\mu_0\otimes \eta^{\otimes\mathbb N},\,\,\,\,\,\,\,\,\bar T(x,(\omega_k)_{k\ge 0}):=(\bar T_{\omega_0}x,(\omega_{k+1})_{k\ge 0})\, .$$
We also define {\bf the planar random billiard system} $(M,\mu, T)$ with:
$$ M:= M_0\times E^{\mathbb N},\,\,\,\,\,\,\,\mu:=\mu_0\otimes \eta^{\otimes\mathbb N},\,\,\,\,\,\,\,
 T\left((x,\ell,(\omega_k)_k)_{k\ge 0}\right):=( T_{\omega_0}(x,\ell),(\omega_{k+1})_{k\ge 0})\, .$$
This dynamical system is a $\mathbb Z^2$-extension of
$\bar M$ by $\Phi:\bar M\rightarrow\mathbb Z^2$ given by:
$$\Phi(x,(\omega_k)_{k\ge 0})=\Phi_{{\omega_0}}(x))\, .$$
Observe that
\begin{eqnarray*}
T^n\left((x,\ell,(\omega_k)_k)_{k\ge 0}\right)&=&( T_{\omega_{n-1}}\circ ...\circ  T_{\omega_0}(x),(\omega_{n+k})_k)\\
&=&(\bar T_{\omega_{n-1}}\circ ...\circ \bar T_{\omega_0}(x),\ell+S_n(x,(\omega_k)_k),(\omega_{n+k})_k)\, ,
\end{eqnarray*}
with
$$S_n(x,(\omega_k)_k)
 :=\sum_{k=0}^{n-1}\Phi\circ \bar T^k(x,(\omega_k)_k)
     =\sum_{k=0}^{n-1}\Phi_{\omega_{k}}\circ
     \bar T_{\omega_{k-1}}\circ\cdots\circ \bar T_{\omega_0}(x),$$
corresponding to the cell change, starting from $x$, after $n$ iterations of maps labeled successively by $\omega_0,\dots,\omega_{n-1}$.

\begin{nota}
As exemplified by the definitions above, we will use overlines such as 
$\bar \mu$, $\bar M$, $\bar T$ to denote objects associated with the quotient random system,
defined in finite measure.
When we introduce a subscript such as $\bar \mu_0$, $\bar M_0$, $\bar T_\omega$, these
denote objects which are not functions of the random coordinate, but are still defined on the quotient space.
\end{nota}
\subsection{A uniform family of maps}
\label{sec:family}
We fix the phase space $\bar{M}_0 = \cup_{i=1}^I \mathbb{S}_i \times [-\pi/2, \pi/2]$ as described above,
and denote by $m$ the normalized Lebesgue measure on $\bar M_0$.  Define
$\cS_0 = \{ \vf = \pm \frac{\pi}{2} \}$ and for a fixed $k_0 \in \mathbb{N}$ with value to be chosen in (\ref{eq:one step}),
for $k \ge k_0$ we define the homogeneity strips,
\begin{equation}
\label{eq:strips}
\mathbb{H}_k = \big\{ (r, \vf) \in \bar M_0 : \tfrac{\pi}{2} - \tfrac{1}{k^2} < \vf < \tfrac{\pi}{2} - \tfrac{1}{(k+1)^2} \big\},
\end{equation}
and the strips $\mathbb{H}_{-k}$ are defined similarly in a neighborhood of $\vf = -\pi/2$.
For the class of maps defined below, we will work with the
extended singularity set $\cS_{0,H} = \cS_0 \cup (\cup_{k \ge k_0} \partial \mathbb{H}_{\pm k})$.
Thus for any $F \in \bar {\mathcal{F}}$, the set $\cS_{\pm n}^{F} := \cup_{i=0}^n F^{\mp i} \cS_{0,H}$
represents the singularity set for $F^{\pm n}$.

We suppose there exists a class $\bar {\mathcal{F}}$ of maps $F : \bar M_0 \circlearrowleft$ such that each
$F \in \bar{\mathcal{F}}$ is a $C^2$ diffeomorphism of
$\bar M_0 \setminus \cS_1^{F}$ onto $\bar M_0 \setminus \cS_{-1}^{F}$ and satisfies the following properties.\\

\smallskip
\noindent
{\bf (H1)  Hyperbolicity and Singularities.}
There exist continuous families of stable and unstable cones, $C^s(x)$ and $C^u(x)$,
which are strictly invariant in the following sense: $DF(x) C^u(x) \subset C^u(Fx)$
and $DF^{-1}(x) C^s(x) \subset C^s(F^{-1}x)$ for all $F \in \bar {\mathcal{F}}$ wherever $DF$ and 
$DF^{-1}$
are defined.

The sets $\cS_{\pm n}^{F}$ comprise finitely many smooth curves for each $n \in \mathbb{N}$.
$\cS_n^{F}$ is uniformly transverse\footnote{The uniformity is assumed to be a lower bound
on the angle between these curves and the relevant cone, which is indepedent of
$x \in \bar M_0$, $n \in \mathbb{N}$ and $F \in \bar {\mathcal{F}}$.} to $C^u(x)$ and $\cS_{-n}^{F}$ is uniformly transverse to
$C^s(x)$ for each $n \ge 0$.  Moreover, $C^s(x)$ and $C^u(x)$ are uniformly transverse on $\bar M_0$
and $C^s(x)$ is uniformly transverse to the horizontal and vertical directions on all of $\bar M_0$.\footnote{This is not
a restrictive assumption for perturbations of the Lorentz gas since the standard cones for
the associated billiard map satisfy this property \cite[Section~4.5]{chernovbook}.}

We assume there exist constants $C_e>0$ and $\Lambda>1$ such that for all $F \in \bar {\mathcal{F}}$
and $n \ge 0$,
\begin{equation}
\label{eq:hyp}
\| DF^n(x) v \| \ge C_e^{-1} \Lambda^n \| v \|, \forall v \in C^u(x), \mbox{ and }
\| DF^{-n}(x) v \| \ge C_e^{-1} \Lambda^n \| v \|, \forall v \in C^s(x),
\end{equation}
where $\| \cdot \|$ is the Euclidean norm on the tangent space to $\bar M_0$.

Finally, near singularities, we assume the maps in $\bar {\mathcal{F}}$ behave like billiards:
there exists $C_a > 0$ such that
\[
C_a \| v \| \le \| DF^{-1}(x) v \|  \cos \vf(F^{-1}x) \le C_a^{-1} \| v\|, \quad \forall v \in C^s(x),
\]
where $\vf(z)$ denotes the angle $\vf$ at the point $z = (r, \vf) \in \bar M_0$.
We also require that the second derivative is bounded by,
\[
C_a \le \| D^2F^{-1}(x) \| \cos^3 \vf(F^{-1}x) \le C_a^{-1} .
\]

\smallskip
\noindent
{\bf (H2) Families of stable and unstable curves.}
We call a $\cC^2$ curve $W \subset \bar M_0$ a {\em stable curve} with respect to the class $\bar {\mathcal{F}}$ if the
unit tangent to $W$ lies in $C^s(x)$ for all $x \in W$.  We say $W$ is {\em homogeneous}
if it lies in a single homogeneity strip $\mathbb{H}_k$.  We define homogeneous unstable
curves analogously.

Let $\widehat \cW^s$ denote the set of $\cC^2$ homogeneous stable curves in $\bar M_0$ whose curvature
is bounded above by a constant $B>0$.  We assume there exists $B$ large enough that
$F^{-1}W$ is a union of elements of $\widehat \cW^s$ for all $W \in \widehat \cW^s$ and $F \in \bar {\mathcal{F}}$.
A family $\widehat \cW^u$ of unstable curves is defined analogously.

\smallskip
\noindent
{\bf (H3) One-step Expansion.}
Assume there exists an adapted norm $\| \cdot \|_*$ on the tangent space to $\bar M_0$,
equivalent to $\| \cdot \|$, in which the constant $C_e$ in \eqref{eq:hyp} can be taken to be 1.
This yields a uniform expansion and contraction in one step for maps in the class $\bar {\mathcal{F}}$.

Let $W \in \widehat \cW^s$.  For $F \in \bar {\mathcal{F}}$, we subdivide $F^{-1}W$ into maximal homogeneous
curves $V_i = V_i(F) \in \widehat W^s$.  We denote by $|J_{V_i}F|_*$ the minimum contraction
on $V_i$ under $F$ in the metric induced by the adapted norm $\| \cdot \|_*$.  We assume
that $k_0$ in \eqref{eq:strips} can be chosen sufficiently large that,
\begin{equation}
\label{eq:one step}
\limsup_{\delta \to 0} \sup_{F \in \bar {\mathcal{F}}} \sup_{\substack{W \in \widehat \cW^s \\ |W| < \delta}} \sum_i |J_{V_i}F|_* < 1 ,
\end{equation}
where $|W|$ denotes the arclength of $W$.

In addition, if we weaken the power of the Jacobian slightly, we assume that the sum above
still converges (although it need not be a contraction).  There exists $\zeta_0 \in (0, 1)$
and $C_1>0$ such that
for each $\delta > 0$ and $\zeta \in [\zeta_0, 1]$,
\[
\sup_{F \in \bar {\mathcal{F}}} \sup_{\substack{W \in \widehat \cW^s \\ |W| < \delta}} \sum_i |J_{V_i}F|_{\cC^0(V_i)}^\zeta \le C_1.
\]

\smallskip
\noindent
{\bf (H4) Bounded distortion.}
There exists a constant $C_d > 0$ with the following properties.
Let $W' \in \widehat \cW^s$ and for $F \in \bar {\mathcal{F}}$, $n \in \mathbb{N}$, let $x,y \in W \subset F^{-n}W'$
such that $F^iW$ is a homogeneous stable curve for each $0 \le i \le n$.  Then,
\begin{equation}
\label{eq:dist}
\left| \frac{J_WF^n(x)}{J_WF^n(y)} - 1 \right| \le C_d d_W(x,y)^{1/3}\, ,
\end{equation}
where $J_WF^n$ denotes the (stable) Jacobian of $F^n$ along $W$ with respect to arclength.

\smallskip
\noindent
{\bf (H5) Invariant measure.}
All the maps $F \in \bar {\mathcal{F}}$ have the same  invariant measure $\bar\mu_0$.

\medskip
\begin{rem}
Assumption (H5) can be replaced more generally with the requirement that all $F \in \bar {\mathcal{F}}$ preserve the same
measure $\bar\mu$ which is absolutely continuous with respect to $m$ and mixing.  In addition,
$\bar\mu$ should satisfy the following technical assumptions:
For $k \ge k_0$, $\bar\mu
(\mathbb H_k)=O(k^{-5q})$ for some $q\in(\frac 45,1)$; also, $\bar\mu$ can be 
disintegrated into measures $\mu_{\alpha}$ along any measurable foliation of $\bar M_0$ into stable 
manifolds $\{W_{\alpha}, \alpha\in \cA\}$, with a factor measure $\lambda$, such that 
$$\bar\mu_0(A)=\int_{\alpha\in \cA}\int_{x\in W_{\alpha}} 1_A(x)\, d\mu_{\alpha} d\lambda(\alpha),$$ 
where $d\mu_{\alpha}=\rho_{\alpha} dm_\alpha$ satisfies a regularity condition: 
$|\ln \rho_{\alpha}(x)-\ln \rho_{\alpha}(y)|\leq C_F d_{W_{\alpha}}(x,y)^{1/3}$, for some constant 
$C_F \ge C_d$,  $d_{W}(x,y)$ is the distance of $x$ and $y$ measured along the curve $W$, and $m_\alpha$ is arclength
measure on $W_\alpha$.

This generalization to other smooth invariant measures is of interest, for example, when considering perturbations in the form 
of certain soft potentials rather than hard scatterers, or the case of external forces due to gradient fields.  See for instance \cite{balint toth,Chernov2001} and their inclusion in a similar perturbative
framework \cite{MarkHongKun2013}.
\end{rem}

A crucial lemma, which will allow us to draw random sequences from the class $\bar {\mathcal{F}}$, is the following.

\begin{lem}
\label{lem:sequence}
Fix a class $\bar {\mathcal{F}}$ satisfying {\bf (H1)}-{\bf (H5)} with uniform constants.
Let $\uomega \in E^{\mathbb{N}}$, and suppose $\bar T_{\omega_k} \in \bar {\mathcal{F}}$ for all $k \ge 0$.

Then for all $n \in \mathbb{N}$, the composition 
$\bar T^n_{\uomega} := \bar T_{\omega_{n-1}} \circ \cdots \circ \bar T_{\omega_0}$ satisfies assumptions {\bf (H1)}-{\bf (H5)},
with possibly larger constants (that are nonetheless independent of $n$ and $\uomega$), and with respect to the singularity sets
$\cS_n^{\bar T_{\uomega}} = \cup_{k=0}^{n-1} \bar T_{\omega_0}^{-1} \circ \cdots \circ \bar T_{\omega_k}^{-1} \cS_0$.
\end{lem}

Lemma~\ref{lem:sequence} is proved in \cite[Section~5.3]{MarkHongKun2013}.

\subsection{Distance in the class $\bar {\mathcal{F}}$}
\label{sec:distance}

To define a notion of distance $d_{\bar {\mathcal{F}}}(\cdot, \cdot)$ in the class of maps $\bar {\mathcal{F}}$, let 
$F_1, F_2 \in \bar {\mathcal{F}}$ and for $\ep >0$,
let $N_\ep(\cS_{-1}^{F_i})$ denote the $\ep$-neighborhood of the singularity set $\cS_{-1}^{F_i}$.
We say $d_{\bar {\mathcal{F}}}(F_1, F_2) \le \ep$ if for all $x \notin N_\ep(\cS_{-1}^{F_1} \cup \cS_{-1}^{F_2})$:
\begin{itemize}
  \item[{\bf (C1)}]  $d((F_1)^{-1}(x), (F_2)^{-1}(x)) \le \ep$;
  \item[{\bf (C2)}]  $\displaystyle \left| \frac{J_{\bar\mu_0}F_i(x)}{J_{\bar\mu_0}F_j(x)} - 1 \right| \le \ep$, $i,j = 1,2$;
  \item[{\bf (C3)}]  $\displaystyle \left| \frac{J_WF_i(x)}{J_WF_j(x)} - 1 \right| \le \ep$,
  for all $W \in \widehat \cW^s$ and $x \in W$, $i,j = 1,2$;
  \item[{\bf (C4)}]  $\| D(F_1)^{-1}(x) v - D(F_2)^{-1}(x) v \| \le \sqrt{\ep}$, for any unit vector $v$ tangent
  to $W \in \widehat \cW^s$ at $x$.
\end{itemize}

For $F_0 \in \bar {\mathcal{F}}$ and $\vartheta_0 > 0$, define
\begin{equation}\label{defnFtheta}
\bar{\mathcal{F}}_{\vartheta_0}(F_0)=\{F \in \bar {\mathcal{F}} : d_{\bar {\mathcal{F}}}(F, F_0)<\vartheta_0\},
\end{equation}
to be the $\theta_0$ neighborhood of $F_0$ in $\bar {\mathcal{F}}$.\\

We remark that is definition of distance does not require the sets $\cS_{-1}^1$ and $\cS_{-1}^2$ to
be close in any sense, only that the maps are $\cC^1$-close outside an $\ep$-neighborhood
of the union of the two singularity sets. 
Next, we describe a perturbation family of billiards that satisfying assumptions 
(\textbf{H1})-(\textbf{H5}), to illustrate that these  assumptions are reasonable.


\subsection{Applications -- Deterministic perturbations}
Given $I$ intervals $J_1, \ldots J_I$, we fix the phase space 
$\bar M_0 = \bigcup_{i=1}^I J_i\times [-\pi/2,\pi/2]$ on which the maps in class $\bar {\mathcal{F}}$ are defined.
We use the notation $\bar Q = \bar Q(\{ O_i \}_{i=1}^I ; \{ J_i \}_{i=1}^I)$ to denote the configuration of
scatterers $O_1, \ldots, O_l$ placed on the billiard table
such that $|\partial O_i| = |J_i|$, $i = 1, \ldots, I$.
We identify the endpoints of $J_i$ so that each $J_i$ 
can be identified with a circle and each component
of $\bar M_0$ is a cylinder.
Since we have fixed $J_1, \ldots, J_I$, $\bar M_0$ remains the same
for all configurations $\bar Q$ that we consider.
For each such configuration, we define
\[
\tau_{\min}(\bar Q) =
\inf \{ \tau(x) : \tau(x) \mbox{ is defined for the configuration } \bar Q \} .
\]
Similarly, we define $\tau_{\max}$, as well as $\cK_{\min}(\bar Q)$ and $\cK_{\max}(\bar Q)$, which denote the minimum and maximum curvatures
respectively of the $\partial O_i$ in the configuration $\bar Q$.  The constant
$E_{\max}(\bar Q)$ denotes the maximum $C^3$ norm of the $\partial O_i$ in $\bar Q$.

For each fixed $\tau_*, \cK_*, E_* >0$, define $\cQ_1(\tau_*, \cK_*, E_*)$ to be the collection
of all configurations $\bar Q$ such that:
$$\tau_* \le \tau_{\min}(\bar Q) \le \tau_{\max}(\bar Q) \le \tau_*^{-1},\,\,\,\,\,\,
\cK_* \le \cK_{\min}(\bar Q) \le K_{\max}(\bar Q) \le \cK_*^{-1},\,\,\,\,\,\,\,\,E_{\max}(\bar Q) \le E_*.$$
Let $\bar {\mathcal{F}}_1(\tau_*, \cK_*, E_*)$ be the corresponding set of billiard maps induced by the
configurations in $\cQ_1$.  
The following lemma is proved in \cite{MarkHongKun2013}.

\begin{lem} {\em (\cite[Theorem~2.7]{MarkHongKun2013})}
\label{thm:F1}
Fix intervals $J_1, \ldots, J_I$ and let $\tau_*, \cK_*, E_* >0$.  The family
$\bar {\mathcal{F}}_1(\tau_*, \cK_*, E_*)$ satisfies {\bf (H1)}-{\bf (H5)} with uniform constants
depending only on $\tau_*$, $\cK_*$ and $E_*$.  
\end{lem}

We fix an initial configuration of scatterers $\bar Q_0 \in \cQ_1(\tau_*, \cK_*, E_*)$ and consider
configurations $\bar Q$ which alter each $\partial O_i$ in $\bar Q_0$ to a curve
$\partial \tilde{O}_i$ having the same arclength as $\partial O_i$.
We consider each $\partial O_i$ as a parametrized curve
$u_i:J_i \to \mathbb{R}^2$ and each $\partial \tilde{O}_i$ as parametrized by $\tilde{u}_i$.
Define $$\Delta(\bar Q, \bar Q_0) = \sum_{i=1}^l |u_i - \tilde{u}_i|_{C^2(J_i,\mathbb{R}^2)}.$$The following is proved in \cite{MarkHongKun2013} .

\begin{lem} {\em (\cite[Theorem~2.8]{MarkHongKun2013})}
\label{thm:deform}
Choose $\vartheta_0 \le  \min \{ \tau_*/2, \cK_*/2 \}$ and let $\bar {\mathcal{F}}_A(\bar Q_0, E_*; \vartheta_0)$
be the set of all billiard maps corresponding to
configurations $\bar Q$ such that $\Delta(\bar Q, \bar Q_0) \le \vartheta_0$ and 
$E_{\max}(\bar Q) \le E_*$.
Then $\bar {\mathcal{F}}_A(\bar Q_0, E_*; \vartheta_0) \subset \bar {\mathcal{F}}_1(\tau_*/2, \cK_*/2, E_*)$
and $d_{\bar {\mathcal{F}}}(\bar F_1, \bar F_2) \le C|\vartheta_0|^{1/3}$ for any $\bar F_1, \bar F_2 \in \bar {\mathcal{F}}_A(Q_0, E_*; \vartheta_0)$.
\end{lem}

The importance of these results is that together, they will imply that the transfer operators associated to maps in the
neighborhood $\bar {\mathcal{F}}_{\vartheta_0}(\bar T_0)$ have a uniform spectral gap if the transfer operator associated with
$\bar T_0$ has a spectral gap.  Moreover, small changes in the configuration of scatterers are seen to generate small
differences in the distance $d_{\bar {\mathcal{F}}}( \cdot , \cdot)$.

%
%
%

%
%
%
\section{Main results}\label{mainresults}
In this section, we consider all $\bar T_{\omega}\in \bar{\mathcal{F}}_{\vartheta_0}(\bar T_0)$, for some $\vartheta_0>0$ small enough and a fixed map $\bar T_0 : \bar M_0 \circlearrowleft$.
\subsection{Local Limit Theorem}
Adapting the proof of \cite[Corollary 2.4]{MarkHongKun2013} (with the slight difference that, here, the observable $\Phi(x,\underline{\omega})$ we are interested in depends also on $\underline{\omega}$), we will prove the following central limit theorem.
\begin{theo}[Central Limit Theorem for the cell index]\label{CLT}
With respect to $\bar\mu$, the
covariance matrix of $(S_n/\sqrt{n})_n$ converges to a
non-negative symmetric function
\begin{equation}\label{Sigma2}
\Sigma^2:=\left(\mathbb E_{\bar\mu}\left[\Phi^{(i)}.\Phi^{(j)}\right]+\sum_{k\ge 1}\mathbb E_{\bar\mu}\left[\Phi^{(i)}.\Phi^{(j)}\circ \bar T^k+\Phi^{(j)}.\Phi^{(i)}\circ \bar T^k\right]\right)_{i,j=1,2}\, ,
\end{equation}
where, for every $j=1,2$,  $\Phi^{(j)}$ is the $j$-th coordinate of $\Phi$, and using $.$ to denote multiplication.\\
Moreover $(S_n/\sqrt{n})_n$ converges in distribution to
a centered Gaussian distribution with covariance matrix $\Sigma^2$.
\end{theo}
The fact that $\Sigma^2$ is positive if $\vartheta_0$ is small enough
will be proved in Lemma \ref{Sigma>0} (using a continuity argument).
In Section~\ref{sec:tilde}, we will define a Banach space $\widetilde \cB$, containing a class
of distributions on $\bar M$, and its dual $\widetilde \cB'$.
For a function $g: \bar M \to \mathbb{R}$, define the functional $H_g$, by
\begin{equation}
\label{eq:H_g}
H_g(\cdot) := \mathbb{E}_{\bar \mu}[g . \, \cdot \, ] \, .
\end{equation}
Remark~\ref{Rke0} and Lemma~\ref{lem:piecewise} will give conditions on $g$ that guarantee
that $H_g \in \widetilde \cB'$.

\begin{theo}[Local limit theorem]\label{theo:LLT}
For every $f,g:\bar M\rightarrow \mathbb R$ such that $H_g\in\widetilde\cB'$ and such that $f\in\widetilde\cB$,
\begin{equation}
\mathbb E_{\bar\mu}\left[f.\mathbf 1_{\{{S}_n=\ell\}}.g\circ \bar T^n\right]=\frac {\exp\left(-\frac{\Sigma^{-2}\ell\cdot \ell}{2n}\right)}{2\pi n\sqrt{\det \Sigma^2}}\mathbb E_{\bar\mu}[f]\mathbb E_{\bar\mu}[g]+ O\left(n^{-\frac 32}\Vert f\Vert_{\widetilde \cB}\, \Vert H_g\Vert_{\widetilde \cB'}\right)\, .
\end{equation}
\end{theo}

\begin{rem}
Due to Lemma~\ref{lem:piecewise} and Remark~\ref{rem:dual}, it suffices for the conclusion of Theorem~\ref{theo:LLT} that $f( \cdot, \underline\omega)$ and $g(\cdot, \underline\omega)$
be piecewise H\"older continuous on
$\bar M_0$ (with H\"older bounds that are uniform in $\underline\omega$).  For instance, the coordinates $\Phi^{(i)}$ of the
displacement function $\Phi$ satisfy these conditions, as well as the free flight function for the billiard map $\bar T_\omega$, 
$\tau(\cdot, \omega)$.
\end{rem}


\subsection{Return time, visit to new obstacles and self intersections}
\label{self}
We define ${\mathcal I}_0(x,\underline{\omega}):=i$ if $x\in \bigcup_{\ell\in\mathbb Z^2}M_{i,\ell}$ as the index in $\{1,...,I\}$ of the obstacle on which
the particle is at time $0$ and ${\mathcal I}_k:={\mathcal I}_0\circ  T^k$.
Since the quantity ${\mathcal I}_0(x,\underline{\omega})$ does not depend on $\underline\omega$, we will also write $\mathcal I_0(x)$ for this quantity.
Note that $\mathcal I_k(x,\underline{\omega})$ does not depend on the index $\ell$ of the cell containing $x$, this allows us to define also $\mathcal I_k$ on $\bar M$ (by projection).

Observe that the fact that the point particle is on the obstacle $(i,\ell)$ at the $k$-th reflection time (i.e.
 $T^k(x,\underline{\omega})\in M_{i,\ell}$) can be rewritten:
$$(\ell_0+S_k(\bar x,\underline{\omega}),\mathcal I_k(\bar x,\underline{\omega}))=(\ell,i),$$
if $x=(\bar x,\ell_0)\in\bar M_0\times\mathbb Z^2$.
We are interested here in the study of the probability that a point particle starting\footnote{Throughout the
paper, we shall use the notation $\bO = (0,0)$ as an element of $\mathbb{Z}^2$.} 
from $\bar M\times\{ \bO \}$ does not come back to
its original obstacle until time $n$, that is in 
$\bar\mu(B_n)$
with
$$B_n:=\{\forall k=1,...,n\, :\, 
       ({\mathcal I}_k,S_k)\ne (\mathcal I_0,(0,0))\} \subset \bar M \,  .$$
We also study the probability that the obstacle visited at time $n$ has not been visited before, that is
$\bar\mu(B'_n)$
with
$$B'_n:=\{\forall k=0,...,n-1\, :\, ({\mathcal I}_k,S_k)\ne (\mathcal I_{n},S_n)\} \subset \bar M \, .$$
Observe that, because of the reversibility of our model, $\mu(B_n)=\mu(B'_n)$.
\begin{theo}\label{theo:range}
We have the following asymptotics
$$
\bar\mu(B_n)= \bar\mu(B'_n) 
= \frac{2I\pi\sqrt{\det\Sigma^2}}{\log n}+O\left((\log n)^{-\frac 43}\right)\, ,\quad\mbox{as}\quad
n\rightarrow + \infty\, .
$$
\end{theo}
In Section \ref{proofreturn}, we give a proof of the above asymptotic estimates of $\mu(B_n)$ and $\mu(B'_n)$ for in a more general context. This result will appear as an easy and direct consequence
of the local limit theorem, Theorem~\ref{theo:LLT}.
We now consider the number of couples of times
at which the point particle hits the same obstacle:
\[
\mathcal V_n:=\sum_{i,j=1}^{n}\mathbf 1_{\{S_j=S_i,\ \mathcal I_j=\mathcal I_i\}}.
\]
\begin{theo}\label{theo:selfintersection} $\bar\mu$-almost surely, we have:
$$\lim_{n\to\infty} \frac{\mathcal V_n}{n\log n}= \frac 1{\pi\sqrt{\det \Sigma^2}} \frac{\sum_{a=1}^I |\partial O_a|^2}{\left(\sum_{b=1}^I|\partial O_a|\right)^2}.$$ 
\end{theo}
The proof of the previous result is delicate as it uses a precise estimate of the variance of $\mathcal V_n$. 
As can be seen from the works by Bolthausen \cite{Bolthausen} and by Deligiannidis and Utev \cite{DU}, going from
a rough to a precise estimate of the variance of the number of self intersections requires important additional work.
In section \ref{proofselfinter}, we give a proof of this result
under general spectral assumptions. Our argument provides, in the case of random walks, an alternative argument to the one given by Deligiannidis and Utev in \cite{DU}. Let us indicate that even if we use 
the general sheme of the previous unpublished paper \cite{soazSelfInter} (in which an analogous result is proved for a single billiard map), this general scheme being just the natural decomposition already used by Bolthausen in \cite{Bolthausen} to get a non-optimal estimate of the variance, the method we use in the present paper to establish our crucial estimates
is different from \cite{soazSelfInter}. In particular our method enables us to get rid of some
assumptions (bounded cell change function, Banach spaces continuously injected in some $L^p$) that were satisfied and used in \cite{soazSelfInter}.

The two previous results (probability to visit a new site, precise asymptotics for the number self-intersections), in addition to being
 interesting in their own right, will greatly help us to prove the result of the next section.

\subsection{Billiard in random scenery}
To each obstacle $(i,\ell)$, we associate a random variable
$\xi_{(i,\ell)}$. We assume that these random variables
are i.i.d., centered and square integrable and independent of
the dynamic of the billiard.
We assume that, each time the point particle hits the obstacle $(i,\ell)$, it wins the value $\xi_{(i,\ell)}$. Let $\mathcal Z_n$
be the total amount won by the particle up to the $n$-th reflection.
For every $n$, we consider the linearized process $(\widetilde {\mathcal Z}_n(t))_{t\ge 0}$
defined by
$$\widetilde {\mathcal Z}_n(t)=\mathcal Z_{\lfloor nt\rfloor}+(nt-\lfloor nt\rfloor)(\mathcal Z_{\lfloor nt\rfloor+1}-\mathcal Z_{\lfloor nt\rfloor})\,  .$$
\begin{theo}\label{theo:randomscenery}
The sequence of processes $((\widetilde{\mathcal Z}_n(t)/\sqrt{n\log n})_{t\ge 0})_n$
converges in distribution,
with respect to the uniform norm on $[0,T]$ for every $T>0$, to a Brownian motion $B = (B_t)_{t \ge 0}$ such that 
$$\mathbb E[B_1^2]=\frac {\sigma^2_\xi}{\pi\sqrt{\det \Sigma^2}} \frac{\sum_{a=1}^I |\partial O_a|^2}{\left(\sum_{b=1}^I|\partial O_b|\right)^2}.$$

If, moreover, there exists $\chi>0$ such that $\mathbb E[|\xi_{(1,0)}|^2(\log^+|\xi_{(1,0)}|)^\chi|)]<\infty$, then, for 
almost every
realization of $(\xi_{i,\ell})_{i,\ell}$, 
$(\widetilde {\mathcal Z}_n)_n$ converges in distribution to the same Brownian motion $B$.
\end{theo}

Let us say a few words about the historical background of this result.
Limit distributional theorems of analogous processes when $S_n$ 
is replaced by a random  walk on $\mathbb Z^d$ were first established at the end of the 70's by Borodin in \cite{Bor1,Bor2} and by Kesten and Spitzer in \cite{KS}, by Boltausen \cite{Bolthausen} in dimension 2 ten years later, and more recently by
Deligiannidis and Utev in \cite{DU} and by Castell, Guillotin-Plantard
and the second author in \cite{FFN}. Let us also remark that when the random walk is the one dimensional simple symmetric random walk on $\mathbb Z$, the random walk in random scenery corresponds to an ergodic sum of a dynamical system, the so-called $T,T^{-1}$-transformation. This dynamical system has been introduced in a list of open problems by Weiss \cite[problem 2, p. 682]{Weiss} in the early 1970's. This dynamical system is a famous natural example of a $K$-transformation which is not Bernoulli and even not loosely Bernoulli as has been shown by Kalikow in \cite{Kalikow}.\\

We prove Theorem \ref{theo:randomscenery} in a more general context in Section \ref{proofRS}.
As noticed by Deligiannidis and Utev in \cite{DU} in the context of random walks, the estimate provided by Theorem \ref{theo:selfintersection} simplifies greatly the proof of Theorem~\ref{theo:randomscenery} compared to \cite{Bolthausen,soazPAPA}
(\cite{soazPAPA} contained a proof of this result for a single billiard map, with the use of the properties of Young towers). Furthermore, we
simplify also the tightness argument used by Bolthausen in \cite{Bolthausen}.
\subsection{Limit theorems in infinite measure}
The following results are consequences of our perturbation result (Proposition \ref{prop:perturb}), combined with
the general results of \cite{DamienSoaz} and of \cite{SoazDecorr}.

Our next result deals with the asymptotic behavior of additive functionals of $S_n$, that is of quantities of the form
$\sum_{k=0}^{n-1}g(S_k)$, for summable functions $g:\mathbb Z^2\rightarrow \mathbb R$. This can be seen as the ergodic sum
$\sum_{k=0}^{n-1}G\circ T^k$ with $G(x,\ell,\uomega):=g(\ell)$.

\begin{theo}[Additive functionals of $S_n$]\label{TCL}
If $g$ is a summable (i.e. $\sum_{\ell\in\mathbb Z^2}|g(\ell)|<\infty$), then
$$\lim_{n\to\infty} \frac{\sum_{k=0}^{n-1}g(S_k)}{\log n}  = \frac 1{2\pi\sqrt{\det\Sigma^2}}\, \sum_{\ell\in\mathbb Z^2} g(\ell)\ \mathcal E\, ,$$ convergence in distribution,
where $\mathcal E$ is an exponential random variable with expectation 1 and where $\Longrightarrow$ means the convergence in distribution with respect to 
any probability measure absolutely continuous with respect to $\bar\mu$.

If moreover $\sum_{\ell\in\mathbb Z^2}g(\ell)=0$ and $\sum_{\ell\in\mathbb Z^2}|\ell|^{\varepsilon}|g(\ell)|<\infty$, for some $\varepsilon>0$, then
$$\lim_{n\to\infty} \frac{\sum_{k=0}^{n-1}g(S_k)}{\sqrt{\log n}}  = \frac 1{\sqrt{2\pi}(\det\Sigma^2)^{\frac 14}}\, \sigma_g\ \mathcal E\, \mathcal N,$$ convergence in distribution, 
where $\mathcal E$ is as above and where $\mathcal N$ is a standard
Gaussian random variable independent of $\mathcal E$ and where
$$\sigma_g^2:=\sum_{\ell\in\mathbb Z^2}(g(\ell))^2+2\sum_{k\ge 1}
  \left(\sum_{\ell,\ell'\in\mathbb Z^2}g(\ell)g(\ell')\bar\mu_0(S_k=\ell-\ell')\right)\, .$$ \\
\end{theo} 

For $g: M_0 \to \mathbb{R}$, define $H_{g, \ell} : \cB \to \mathbb{R}$ by 
$H_{g, \ell}(h) = \mathbb{E}_{\bar \mu_0} [ g(\cdot , \ell) h ]$.  We also obtain the decay rates of correlations for the process generated  by our random systems in infinite measure:\\

\begin{theo}[Mixing and decorrelation in infinite measure]\label{mixing}
Let $K\ge 1$.
Let $f,g:M_0\rightarrow\mathbb R$ be two functions such that
$$\sum_{\ell\in\mathbb Z^2}|\ell|^{2K}\left(\Vert f(\cdot,\ell)\Vert_{\mathcal B}+\Vert H_{g , \ell} \Vert_{\mathcal B'}\right) <\infty\, .$$
Then, there exist real numbers $C_0(f,g),...,C_K(f,g)$ such that
$$\int_{M_0\times E^{\mathbb N}}f.g\circ T_{\omega_n}\circ...\circ T_{\omega_1}\, d\mu_0\,  d\eta^{\otimes\mathbb N}((\omega_n)_n)=
\sum_{m=0}^K\frac{C_m(f,g)}{n^{m+1}}+o(n^{-K+1})\, ,$$
with $C_0(f,g)=\frac 1{2\pi\sqrt{\det \Sigma^2}}\int_{M_0}f\, d\mu_0\, \int_{M_0}g\, d\mu_0$.
\end{theo}
\section{Transfer Operators}
\label{sec:ops}
In order to prove our main limit theorems, we will study the transfer operators associated
with the random maps $T$ and $\bar T$ as perturbations of the
transfer operator associated with a fixed quotient billiard map $\bar T_0$.

In this section, we fix a class of maps $\bar {\mathcal{F}}$ satisfying {\bf (H1)}-{\bf (H5)} with uniform constants.
$\bar T$ denotes the quotient of the full random map $T$, while $\bar T_\omega$, $\omega \in E$ 
denotes a quotient billiard map
belonging to $\bar {\mathcal{F}}$, following the notation defined in Section~\ref{sec:rand pert}.

Using {\bf (H3)}, choose $\delta_0 > 0$ for which there exists $\theta<1$ so that \eqref{eq:one step}
satisfies,
\begin{equation}
\label{eq:one step contract}
\sup_{\bar T_\omega \in \bar{\mathcal{F}}_{\vartheta_0}} \sup_{\substack{W \in \widehat\cW^s \\ |W| < \delta_0}}
\sum_i |J_{V_i}\bar T_\omega|_* \le \theta .
\end{equation}
We then define $\cW^s \subset \widehat\cW^s$ to be those stable curves in $\widehat\cW^s$ whose
length is at most $\delta_0$.

Following \cite{MarkHongKun2011}, for any $\bar T_\omega \in \bar{\mathcal{F}}$ and $n \ge 0$,
define $\bar T_\omega^{-n}\cW^s \subset \cW^s$ to be the set of homogeneous stable curves
$W \in \cW^s$ whose images $\bar T_\omega^iW \in \cW^s$ for $0 \le i \le n$.
For $p \in [0,1]$ and letting $\cC^p(\bar T_\omega^{-n}\cW^s)$ denote those functions $\psi$ which
are H\"older continuous on elements of $\bar T_\omega^{-n}\cW^s$, it follows from {\bf (H1)}
that $\psi \circ \bar T_\omega \in \cC^p(\bar T^{-n-1}\cW^s)$.  Thus if
$f \in (\cC^p(\bar T_\omega^{-n-1}\cW^s))'$ is an element of the dual of
$\cC^p(\bar T_\omega^{-n-1}\cW^s)$, then
$\cL_{\bar T_\omega} : (\cC^p(\bar T_\omega^{-n-1}\cW^s))' \to (\cC^p(\bar T_\omega^{-n}\cW^s))'$
is defined by
\[
\cL_{\bar T_\omega} f (\psi) = f(\psi \circ \bar T_\omega), \qquad \forall \; \psi \in \cC^p(\bar T_\omega^{-n} \cW^s) .
\]
If in addition, $f$ is a measure absolutely continuous with respect to $\bar \mu_0$, then we identify
$f$ with its density in $L^1(\bar \mu_0)$, which we shall also denote $f$, i.e. $f(\psi) = \int_{\bar M_0} \psi \, f \, d\bar \mu_0$.  With this identification, we write $L^1(\bar\mu_0) \subset (\cC^p(\bar T_\omega^{-n}\cW^s))'$ for each $n \in \mathbb{N}$.  Then acting on $L^1(\bar\mu_0)$, $\cL_{\bar T_\omega}$ has the following
familiar expression,
\[
\cL_{\bar T_\omega}^n f = f \circ \bar T_\omega^{-n} , \qquad \mbox{for any $n \ge 0$.}
\]
For brevity, sometimes we will denote $\cL_{\bar T_\omega}$ by $\cL_\omega$.

Let $P$ be the transfer operator
of $\bar  T$ with respect to $\bar\mu:=\bar\mu_0\otimes\eta^{\otimes \mathbb N}$. This operator is given by
$$P f(y,(\omega_k)_{k\ge 0})
    =\int_{E}\mathcal L_{\omega_{-1}} f(\cdot)(y,(\omega_{k-1})_{k\ge 0})\,  d\eta(\omega_{-1}).$$
Let us write $\cdot$ for the usual scalar product on $\mathbb R^2$.
We consider the family of operators
$(P_u)_{u\in\mathbb R^2}$ given by
$$ P_uf(y,(\omega_k)_k) :=P\left(e^{iu \cdot \Phi}f\right)(y,(\omega_k)_k)=\int_{E}\mathcal L_{u,\omega_{-1}}
        f(\cdot,(\omega_{k-1})_{k\ge 0})(y)\, d\eta(\omega_{-1}),$$
    where
    \[
    \cL_{u,\omega_{-1}} f = \cL_{\omega_{-1}} (e^{iu \cdot \Phi_{\omega_{-1}}} f) \, .
    \]
Note that
$$P_u^nf=P^n(e^{iu\cdot S_n}f)\, .$$

Using results of \cite{MarkHongKun2013}, we will see that
if we restrict $\bar T_\omega$ to a neighborhood $\bar {\mathcal{F}}_{\vartheta_0}(\bar T_0)$ according to 
\eqref{defnFtheta}, then 
$P$ is a small (depending on $\vartheta_0$) perturbation of
the transfer operator $\mathcal P_{0}$ of the product system $(\bar M, \bar\mu:=\bar\mu_0\times\eta^{\otimes \mathbb N},\bar T_0\times \sigma)$, with $\sigma$
is the shift over $E^{\mathbb N}$ (i.e. $\sigma((\omega_k)_{k\ge 0})=(\omega_{k+1})_{k\ge 0}$) and where $(\bar T_0\times\sigma)(x,\uomega):=(\bar T_0(x),\sigma(\uomega))$.
\subsection{Banach spaces $\mathcal B$ and $\mathcal B_w$}
\label{sec:norm def}

We start by defining Banach spaces $\mathcal B\subset \mathcal B_w$ of distributions on $\bar M_0$, on which the
transfer operators $\mathcal L_\omega$ associated to $\bar T_\omega \in \bar {\mathcal{F}}$ are well-behaved.

In order to define our norms, we first require
a notion of distance $d_{\cW^s}(\cdot, \cdot)$ between stable curves as well as
a distance $d(\cdot, \cdot)$ defined among functions supported on these curves.

Due to the transversality condition on the stable cones
$C^s(x)$ given by {\bf (H1)}, each $W \in \cW^s$ can be viewed
as the graph of a function $\vf_W(r)$ of the arc length parameter $r$.
For each $W \in \cW^s$,
let $J_W$ denote the interval on which
$\vf_W$ is defined and set $G_W(r) = (r, \vf_W(r))$ to be its graph so that
$W = \{ G_W(r) : r \in J_W \}$.
We let $m_W$ denote the unnormalized arclength measure on $W$,
defined using the Euclidean metric.

Let $W_1, W_2 \in \cW^s$ and let $\vf_{W_i}$, $G_{W_i}$ denote the corresponding
functions defined above, for $i = 1, 2$.
Denote by $\ell(J_{W_1} \triangle J_{W_2})$ the length of the symmetric difference
between $J_{W_1}$ and $J_{W_2}$.
If $W_1$ and $W_2$ belong to the same homogeneity strip, we define the distance between
them to be,
\[
d_{\cW^s} (W_1,W_2) =
\ell( J_{W_1} \triangle J_{W_2}) + |\vf_{W_1} -\vf_{W_2}|_{\cC^1(J_{W_1} \cap J_{W_2})};
\]
otherwise, we set $d_{\cW^s}(W_1, W_2) = \infty$.

For $0 \leq p \leq 1$, let
$\tilde{\cC}^p(W)$ denote the set of continuous complex-valued functions on
$W$ with H\"{o}lder exponent $p$, measured in the Euclidean
metric.
Denote by $\cC^p(W)$ the closure of $\cC^\infty(W)$
in the $\tilde{\cC}^p$-norm\footnote{While $\cC^p(W)$ is smaller than
$\tilde{\cC}^p(W)$, it does contain $\cC^{p'}\!(W)$ for all $p'>p$.}:
$$| \psi |_{\cC^p(W)} = |\psi|_{\cC^0(W)} + C^{(p)}_W(\psi),$$ where
$C^{(p)}_W(\psi)$ is the H\"older constant of $\psi$ along $W$.
It is remarkable to note that that with this definition,
$$|\psi_1 \psi_2 |_{\cC^p(W)} \le |\psi_1|_{\cC^p(W)} |\psi_2|_{\cC^p(W)}.$$
 $\tilde{\cC}^p(\bar M_0)$ and $\cC^p(\bar M_0)$ can be defined similarly.

Given two curves $W_1, W_2 \in \cW^s$ with $d_{\cW^s}(W_1, W_2) < \infty$, and two
test functions
$\psi_i\in\cC^p(W_i,\mathbb{C})$, {\it{ the distance between}}
$\psi_1$, $\psi_2$ is defined as:
\[
d(\psi_1,\psi_2) =|\psi_1\circ G_{W_1}-\psi_2\circ G_{W_2}|_{\cC^0(I_{W_1} \cap I_{W_2})}.
\]
We will define the relevant Banach spaces by closing $\cC^1(\bar M_0)$ with respect to
the following set of norms.
Fix $0 < p \le \frac 13 $.
Given a function $f \in \cC^1(\bar M_0)$, define the \emph{weak norm}
of $f$ by
\begin{equation}
\label{eq:weak}
|f|_w:=\sup_{W\in\cW^s}\sup_{\substack{\psi \in\cC^p(W)\\
|\psi|_{\cC^p(W)} \leq 1}}\int_W f \psi \; dm_W .
\end{equation}
Choose\footnote{The restrictions on the constants are placed according to the
dynamical properties summarized in {\bf (H1)}-{\bf (H5)}.  For example, $p \le 1/3$ due to the distortion bounds
in {\bf (H4)}, while $\varsigma \le 1 - \zeta_0$ due to {\bf (H3)}, which is relevant for the
uniform Lasota-Yorke inequalities (Lemma~\ref{lem:uniform ly}).}
$q$, $\gamma$, $\varsigma >0$ such that $\varsigma \le 1 - \zeta_0$, $q < p$ and
$\gamma \leq \min \{ \varsigma, p-q \}$.
We define the \emph{strong stable norm} of $f$ as
\begin{equation}
\label{eq:s-stable}
\|f\|_s:=\sup_{W\in\cW^s}\sup_{\substack{\psi\in\cC^q(W)\\
|\psi|_{\cC^q(W)} \leq |W|^{-\varsigma}}} \int_W f \psi \; dm_W
\end{equation}
and the \emph{strong unstable norm} as
\begin{equation}
\label{eq:s-unstable}
\|f\|_u:=\sup_{\ve \leq \ve_0} \; \sup_{\substack{W_1,
W_2 \in \cW^s \\
d_{\cW^s} (W_1,W_2)\leq \varepsilon}}\;
\sup_{\substack{\psi_i \in \cC^p(W_i) \\ |\psi_i|_{\cC^p(W)} \leq 1\\ d(\psi_1,\psi_2)
=0}} \;
\frac{1}{\varepsilon^\gamma} \left| \int_{W_1} f
\psi_1 \; dm_W - \int_{W_2} f \psi_2 \; dm_W \right|
\end{equation}
where $\omega_0 > 0$ is chosen less than $\delta_0$, the maximum length of $W \in \cW^s$ which
is determined by \eqref{eq:one step contract}.
The \emph{strong norm} of $f$ is defined by
\[
\|f\|_\cB = \|f\|_s + c_0 \|f\|_u,
\]
where $c_0$ is a small constant chosen so that the uniform Lasota-Yorke inequalities in
\cite[Theorem~2.2]{MarkHongKun2013} hold.

We define $\cB$ to be the completion of $\cC^1(\bar M_0)$ in the strong norm\footnote{As a measure,
$f \in \cC^1(\bar M_0)$ is identified with $fd\bar\mu_0$ according to our earlier convention.
As a consequence, Lebesgue measure $dm = (\cos \vf)^{-1} d\bar\mu_0$ is not automatically
included in $\cB$ since $(\cos \vf)^{-1} \notin \cC^1(\bar M_0)$.  It follows from
\cite[Lemma 3.5]{MarkHongKun2013} that in fact, $m \in \cB$ (and $\cB_w$).}
and $\cB_w$ to be the completion of $\cC^1(\bar M_0)$ in the weak norm.
\begin{rem}\label{Rke0}
Due to \cite[Lemma 3.4]{MarkHongKun2013}, we have for $f \in \cB_w$,
\[
| f(\psi) | \le |f|_w \Big( |\psi|_\infty + \sup_{W \in \cW^s} C^{(p)}_W(\psi) \Big), \qquad \mbox{for all $\psi \in \cC^p(\cW^s)$.}
\]
This permits us to
extend $\mathbb E_{\bar\mu_0}[\cdot]$ to a linear continuous form on $\mathcal B_w$
(and so on $\mathcal B$) since
$$\forall f\in C^1(\bar M_0),\quad
      \mathbb E_{\bar\mu_0}[f]=\int_{\bar M_0}f\, d\bar\mu_0=f(\mathbf 1_{\bar M_0})\, . $$ 
\end{rem}
We begin by recalling some properties of $\cB$ and $\cB_w$ proved in \cite{MarkHongKun2011,
MarkHongKun2013, MarkHongKun2014}.

\begin{lem}
\label{lem:recall}
\begin{itemize}
  \item[a)] \cite[Lemma 3.7]{MarkHongKun2011} $\cB$ contains piecewise H\"older continuous
  functions $f$ with exponent $\zeta > \gamma/(1-\gamma)$ as described in Lemma~\ref{lem:piecewise} below.
  \item[b)] \cite[Lemma 3.5]{MarkHongKun2013} $(\cos \vf)^{-1} \in \cB$.  Thus, Lebesgue
  measure $m = (\cos \vf)^{-1} \bar\mu_0 \in \cB$ and so is $fm$ for any $f$ as in item (a) above.
  \item[c)] \cite[Lemma 2.1]{MarkHongKun2011} $\cL_\omega$ is well-defined as a continuous
  linear operator on both $\cB$ and $\cB_w$ for any $\bar T_\omega \in \bar {\mathcal{F}}$.
  Moreover, there exists a sequence of continuous\footnote{The first three of these are
  also injective.  The fourth can be made injective by introducing a weight $|W|^{-\eta}$
  for test functions $\psi$ in the weak norm (as appears in the definition of $\| \cdot \|_s$)
  and requiring $\eta > p$ (see, for example, \cite[Lemma~3.8]{MarkHongKun2014}).} inclusions
  $\cC^\zeta(\bar M_0) \hookrightarrow \cB \hookrightarrow \cB_w \hookrightarrow (\cC^p(\bar M_0))'$,
  for all $\zeta > \gamma/(1-\gamma)$.
  \item[d)] \cite[Lemma 3.10]{MarkHongKun2011} The unit ball of $(\cB, \| \cdot \|_{\cB})$
  is compactly embedded in $(\cB_w, | \cdot |_w)$.
\end{itemize}
\end{lem}

The following lemma is crucial for describing the types of discontinuities allowed in elements of
$\cB$ and for proving that the operator $\cL_{u, \omega}$ is analytic in $u$.

\begin{lem}
\label{lem:piecewise}
Let $\mathfrak P$ be a (mod 0) countable partition of $\bar M_0$ into open, simply connected sets such that:
(1) for each $k \in \mathbb{N}$, there is an $N_k < \infty$ such that at most $N_k$ elements
$Z \in \mathfrak P$ intersect $\mathbb{H}_k$; (2) there are constants $K, C_0>0$ such that for each
$Z \in \mathfrak P$ and $W \in \cW^s$, $Z \cap W$ comprises at most $K$ connected components
and for any\footnote{In fact, Lemma 3.5 of \cite{MarkHongKun2014} allows a nondegenerate
tangency between $\partial \mathfrak P$ and the stable cone: $m_W(N_\ve(\partial Z) \cap W) \le C_0 \ve^{t_0}$, for some $t_0 >0$.  But we will not need this weaker condition here so we assume
$t_0=1$ in order to simplify the proofs and also the statement of
the norms (which otherwise would depend on $t_0$).} $\ve>0$, $m_W(N_{\ve}(\partial Z) \cap W) \le C_0 \ve$.

\begin{itemize}
  \item[a)] \cite[Lemma~3.5]{MarkHongKun2014}
 Let $\zeta >\gamma/(1-\gamma)$.  If  $f \in \cC^\zeta(Z)$ for each $Z \in \cZ$ and $\sup_{Z \in \mathfrak P} |f|_{\cC^\zeta(Z)} < \infty$, then $f \in \cB$ and $\| f \|_{\cB} \le C \sup_{Z \in \mathfrak P} |f|_{\cC^\zeta(Z)}$, for some $C>0$ independent of
 $f$.
 In particular, $\cC^\zeta(\bar M_0) \subset \cB$ for each $\zeta>\gamma/(1-\gamma)$.
  \item[b)] \cite[Lemma~5.3]{MarkHongKun2014}
  Suppose in addition that $\zeta > \max \{ p, \gamma/(1-\gamma) \}$ and there is a uniform bound
  on the $N_k$ above.  If $g$ satisfies
  $\sup_{Z \in \mathfrak P} |g|_{\cC^\zeta(Z)} < \infty$ and $f \in \cB$, then $fg \in \cB$ and
  $\| f g \|_{\cB} \le C \| f \|_{\cB} \sup_{Z \in \mathfrak P} |g|_{\cC^\zeta(Z)}$ for some $C>0$ independent
  of $f$ and $g$.
\end{itemize}
\end{lem}
\subsection{Banach spaces $\widetilde {\mathcal B}$ and $\widetilde {\mathcal B}_w$}
\label{sec:tilde}

In this section, we introduce the associated Banach spaces $\widetilde {\mathcal B}_w$ and  $\widetilde{\mathcal B}$
on $\bar M$ on which $P$ acts suitably.
$\widetilde{\mathcal B}$ will correspond to a set of Lipschitz functions from $E^{\mathbb N}$ to $\mathcal B$ and 
$\widetilde{\mathcal B}_w$ will correspond to the set of uniformly bounded functions from $E^{\mathbb N}$ to $\mathcal B_w$.
For convenience, we will identify elements of 
$\cB ^{E^{\mathbb N}}$
with 
distributions $f$ on $\bar M_0 \times E^{\mathbb N}$ such that $f( \cdot, \uomega) \in \cB$ for all $\uomega \in E^{\mathbb N}$.

Let $\varkappa>\sup_{\omega\in E}\Vert \mathcal L_\omega\Vert_{L(\mathcal B,\mathcal B)}\ge 1$. Let us define
$$\widetilde {\mathcal B}:=\{ f \in \cB ^{E^{\mathbb N}} \ : \ 
   \Vert f\Vert_{\widetilde{\mathcal B}}<\infty\}\, ,$$
with 
$$\Vert f\Vert_{\widetilde{\mathcal B}}:=\sup_{\uomega
     \in E^\mathbb N}
   \Vert f(\cdot,\uomega)\Vert_{\mathcal B}+\sup_{\uomega\ne \uomega'}
    \frac{\Vert f(\cdot,\uomega)- f(\cdot,\uomega')\Vert_{\mathcal B}}
      {d(\uomega,\uomega')}\, ,$$
and with $$d((\omega_k)_k,(\omega'_k)_k)
      =\varkappa^{-\min\{k\ge 0:\omega_k\ne\omega_{k'}\}}.$$
It is immediate from this definition and the definition of $\cB$, that $\widetilde \cB$ is the completion 
in the $\| \cdot \|_{\widetilde \cB}$ norm of the set
of functions 
$$
{(C^1(\bar M_0))^{E^{\mathbb N}}} = \{ f: \bar M_0 \times E^{\mathbb N} \to \mathbb{C} \ : \
f(\cdot, \underline \omega) \in C^1(\bar M_0) \ \forall \uomega \in E^\mathbb{N} \} \, .
$$  
In particular, $\widetilde \cB$ is a Banach space.

\begin{rem}\label{PhiinB}
It will be worthwhile to notice that, due Lemma~\ref{lem:piecewise}(a), for every $\omega\in E$, the coordinates of $\Phi_\omega$ belong to $\mathcal B$, so that the coordinates of  $\Phi$ are in $\widetilde{\mathcal B}$.
\end{rem}

We also define
$$\widetilde {\mathcal B}_w:=\{ f  \in (\cB_w )^{E^{\mathbb N}}
\ :\ \vert f\vert_{\widetilde{\mathcal B}_w} <\infty\}\, ,$$
with
$\vert f\vert_{\widetilde{\mathcal B}_w} :=\sup_{{\underline{\omega}}
     \in E^\mathbb N}
   \vert f(\cdot,{\underline{\omega}})\vert_w\, .$  As with $\widetilde \cB$, the space $\widetilde \cB_w$ can also be realized as the
   completion of ${(C^1(\bar M_0))^{E^{\mathbb N}}}$ in the $| \cdot |_{\widetilde \cB_w}$ norm.

\begin{rem}\label{Rke1}
Using Remark \ref{Rke0}, we extend $\mathbb E_{\bar\mu}[\cdot]$
to a continuous linear form on $\widetilde{\mathcal B}_w$ (and so on 
 $\widetilde{\mathcal B}$) by setting
$$\forall f\in \widetilde{\mathcal B}_w,\quad
      \mathbb E_{\bar\mu}[f]=\int_{E^{\mathbb N}}\mathbb E_{\bar\mu_0}[f(\cdot,\underline{\omega})] \, d\eta^{\otimes\mathbb N}(\underline{\omega})\, . $$ 
      
It follows from Lemma~\ref{lem:piecewise}(a) that for any obstacle $O_a$, $1_{O_a} \in \cB$, and from Lemma~\ref{lem:piecewise}(b)
that $f(\cdot, \uomega) \mapsto 1_{O_a} f(\cdot, \uomega)$ is a bounded linear operator on $\cB$ for each 
$\uomega \in E^\mathbb{N}$ and $f \in \widetilde \cB$.
  Thus
$f \mapsto 1_{O_a} f$ is a bounded linear operator on $\widetilde \cB$ as well.
\end{rem}
We introduce the following notation for convenience.
\begin{nota}
For any positive integer $m$, any $\tilde{\uomega}_m \in E^m$ and any $\uomega\in E^{\mathbb N}$, we will write $(\tilde{\uomega}_m,\uomega)$ as the element of $E^{\mathbb N}$ obtained by concatenation; i.e. such that the first $m$ terms correspond to those of $\tilde{\uomega}_m$ and that the term of order $m+k$ corresponds to the term of order $k$ of $\uomega$.
\end{nota}

\begin{lem}
\label{tildeBnorm}
\begin{itemize}
\item[(a)]
Let $n$ be a positive integer. Denote the norm $\Vert\cdot\Vert_{\sigma}$, for $\sigma\in \{ w, s, u\}$.
If $(f(\cdot,\tilde\uomega_n))_{\tilde\uomega_n\in E^n}$ is a measurable (in $\tilde\uomega_n$) family of elements of $\mathcal B_w$ such that $$\sup_{\tilde\uomega_n\in E^n}\Vert f(\cdot,\tilde\uomega_n)\Vert<\infty,$$ then
$$\left\|\int_{E^n}  f(\cdot ,\tilde\uomega_n)\, d\eta^{\otimes n}(\tilde\uomega_n)\right\|_{\sigma}\le \int_{E^n}  \|f(\cdot ,\tilde\uomega_n)\|_{\sigma}\, d\eta^{\otimes n}(\tilde\uomega_n)\le\sup_{\tilde\uomega_n\in E^n}\|f(\cdot ,\tilde\uomega_n)\|_{\sigma}\, .$$
\item[(b)]If $(H_\omega)_{\omega\in E}$ is a measurable (in $\omega$) family of uniformly bounded operators on $\mathcal B$ (resp. $\mathcal B_w$), then
$H:f(x,\uomega) \mapsto \int_{E}H_{\tilde\omega} (f(x,(\tilde\omega,\uomega)))\, d\eta(\tilde\omega) $ 
defines a continuous linear operator on $\widetilde\cB$ (resp. $\widetilde\cB_w$) with operator norm dominated by
$\sup_{\omega\in E}\Vert H_\omega\Vert_{L(\cB,\cB)}$.
\end{itemize}
\end{lem}
\begin{proof}
(a) is just the triangle inequality.
Let us prove Item (b). Let $f\in \widetilde \cB$ or in $\widetilde\cB_w$ and writing $\Vert\cdot\Vert_{\sigma}$ for the associated norm, due to (a), for every $\uomega,\uomega'\in E^{\mathbb N}$, we have
\begin{eqnarray*}
\Vert Hf(\cdot,\uomega)\Vert_{\sigma}&\le& \sup_{\tilde\omega\in E} \left\Vert H_{\tilde\omega} f(\cdot,(\tilde\omega,\uomega))\right\Vert_{\sigma}\\
   &\le&
     \sup_{\tilde\omega}\Vert H_{\tilde\omega}\Vert_{\sigma}\, \Vert f(\cdot,(\tilde\omega,\uomega))\Vert_{\sigma}\le      \sup_{\tilde\omega}\Vert H_{\tilde\omega}\Vert_{\sigma}\,   \sup_{\tilde\uomega'}\Vert f(\cdot,\uomega')\Vert _{\sigma}\, ,
\end{eqnarray*}
which proves (b) if $f \in \widetilde{\cB}_w$.  If, in addition, $f \in \widetilde{\cB}$, then
\begin{eqnarray*}
{\Vert Hf(\cdot,\uomega)-Hf(\cdot,\uomega')}\Vert_{\mathcal B}&\le&
     \sup_{\tilde\omega}\Vert H_{\tilde\omega}\Vert_{\sigma}\, \sup_{\omega\in E}\Vert f(\cdot,(\omega,\uomega))-f(\cdot,(\omega,\uomega'))\Vert_{\sigma}\, \\
&\le&    \sup_{\tilde\omega}\Vert H_{\tilde\omega}\Vert_{\sigma}\, \sup_{\uomega,\uomega'\in E^{\mathbb N}}\frac{\Vert f(\cdot,(\omega,\uomega))-f(\cdot,(\omega,\uomega'))\Vert_{\sigma}}{d(\tilde\uomega,\tilde\uomega')}\varkappa^{-1}d(\uomega,\uomega')\, ,
\end{eqnarray*}
where $\tilde\uomega = (\omega, \uomega)$ and $\tilde\uomega' = (\omega, \uomega')$.
\end{proof}

\begin{rem}
The previous lemma ensures in particular that $P$ acts continuously
$\widetilde{\mathcal B}$ and on $\widetilde\cB_w$ since $\mathcal L_\omega$ acts uniformly continuously on $\cB$ and on $\cB_w$.
\end{rem}
A key step in our proof is the study the spectral properties on $\widetilde{\mathcal B}$
of $P$ and of the family of operators $P_u$ defined by
 $$P_u:=P(e^{iu\cdot\Phi}\cdot)\, .$$
The next lemma ensures, in particular, that $P_u$ is a linear operator
on $\widetilde{\mathcal B}$. Denote by $\Phi^{(1)}$ and $\Phi^{(2)}$ the components of the vector $\Phi$.
\begin{lem}\label{lem:PuOp}
For every $u\in\mathbb R^2$, any positive integer $m$ and any
$i_1,...,i_m\in\{1,2\}$, $P(\Phi^{(i_1)}...\Phi^{(i_m)} e^{iu\cdot\Phi}\cdot)$ is a linear operator on $\widetilde\cB$ and on $\widetilde\cB_w$, with operator norms uniformly in
$O(\sup_{\omega\in E}\Vert \Phi \Vert^m_\infty)$.
\end{lem}
\begin{proof}
This proof is a variation of the argument used in
\cite[Section 5.2]{MarkHongKun2013}.
Recall that $Pg(\cdot,\uomega)=\int_{E}f(\bar T_\omega^{-1}(\cdot),(\omega,\uomega))\, d\eta(\omega)$, so that
$$ P(\Phi^{(i_1)}...\Phi^{(i_m)} e^{iu\cdot\Phi}f)(\cdot,\uomega)
       =\int_E \left(\Phi_{\omega}^{(i_1)}...\Phi_{\omega}^{(i_m)} e^{iu\cdot \Phi_\omega}\right)\circ \bar T_\omega^{-1}
      \mathcal L_\omega f(\cdot,(\omega,\uomega)\, d\eta(\omega)\, .$$

Let $\omega\in E$.
The singularity set for $\Phi_\omega$ is $\cS_1^{\bar T_\omega}$, which does not satisfy the hypotheses of
Lemma~\ref{lem:piecewise}; however, the
singularity set for $\Phi_\omega \circ \bar T_\omega^{-1}$ is $\cS_{-1}^{\bar T_\omega}$, which satisfies the
hypotheses of Lemma~\ref{lem:piecewise}(b) since there are only finitely many such curves and
they lie in the unstable cone.
Let $\cZ$ denote the (finite) partition of $\bar M_0 \setminus \cS_{-1}^{\bar T_\omega}$ into its maximal
connected components.  Note that $\Phi_\omega \circ \bar T_\omega^{-1}$ is constant on each element of $\cZ$.  
We use Lemma~\ref{lem:piecewise}(b) to estimate, for every $f\in \widetilde \cB$,
\begin{equation}
\label{eq:analytic}
\begin{split}
\| \mathcal L_\omega(\Phi_\omega^{(i_1)}...\Phi_\omega^{(i_m)} e^{iu\cdot\Phi}f) (\cdot, \uomega) \|_{\cB} & = \| (i \cdot \Phi_\omega)^n \circ \bar T_\omega^{-1} \, (\cL_\omega f) (\cdot, \uomega) \|_{\cB}\\
&\le C \sup_{Z \in \cZ} | (\Phi_\omega^{(i_1)}...\Phi_\omega^{(i_m)}e^{iu\cdot\Phi_\omega}) \circ \bar T_\omega^{-1}|_{\cC^1(Z)} \| \cL_\omega f(\cdot,(\omega,\uomega)) \|_{\cB} \\
&\le C'  \Vert \Phi_\omega\Vert_\infty^m \|  f (\cdot,(\omega,\uomega))\|_{\cB} \, .
\end{split}
\end{equation}
Analogously, for every  $ f\in \widetilde \cB_w$,
\[
| \mathcal L_\omega(\Phi_\omega^{(i_1)}...\Phi_\omega^{(i_m)} e^{iu\cdot\Phi}f )   (\cdot, \uomega)  |_{w}
\le C' \Vert \Phi\Vert_\infty^m |  f (\cdot,(\omega,\uomega)) |_w \, ,
\]
and we conclude by Item (b) of Lemma \ref{tildeBnorm}.
\end{proof}

\subsection{View $P_u$ as a perturbation of a quasicompact operator}
\label{sec:props}

For the remainder of Section~\ref{sec:ops}, we fix a billiard map $\bar T_0$, and for $\vartheta_0>0$, define
$\bar {\mathcal{F}}_{\vartheta_0}(\bar T_0)$ as in \eqref{defnFtheta}.  Our main results in this setting will be that for $\vartheta_0$
sufficiently small, both $P$ and $P_u$ are quasi-compact and have a spectral gap in $\widetilde \cB$.  These statements
are contained in Proposition~\ref{Quasicompact} and Theorem~\ref{prop:perturb}.

Recall $P_u:=P(e^{iu\cdot\Phi}\cdot)$.
Our next result states that $P_u$ is a small perturbation (as $\vartheta_0\rightarrow 0$) of  $\mathcal P_u:=\mathcal P(e^{iu\cdot\Phi_0}\cdot)$, where $\mathcal P$ is the transfer operator $\mathcal {P}_0$ of the direct product $(\bar M,\bar\mu,\bar T'_0:=\bar T_0\times \sigma)$, i.e.
$$\mathcal{ P}(f)(y,(\omega_k)_{\ge 0})=\int_{E}\mathcal L_{0} f(\cdot,(\omega_{k-1})_{k})(y)\,  d\eta(\omega_{-1})\,  ,$$
and
$$\mathcal{ P}_u(f)(y,(\omega_k)_{\ge 0})=\int_{E}\mathcal L_{u,0} f(\cdot,(\omega_{k-1})_{k})(y)\,  d\eta(\omega_{-1})\,  .$$
Here, $\cL_0 = \cL_{\bar T_0}$ and $\cL_{u,0} = \cL_{\bar T_0}( e^{iu \cdot \Phi_0} \cdot)$.

\begin{prop}\label{lem:close2}
There exists $C>0$ such that for every $u\in\mathbb R^2$ and every $f\in\widetilde{\mathcal B}$,
$$\left\vert  P_u f-\mathcal P_u f \right\vert_{\widetilde{\mathcal B}_w}\le C\Vert f\Vert_{\widetilde{\mathcal B}}\, \vartheta_0^{\frac\gamma 2}\, .$$
\end{prop}
Before proving this proposition, we state the following lemma.

\begin{lem}
\label{lem:close}
There exists $C > 0$ such that for all $\omega\in E$ and $u \in \mathbb{R}^2$,
\[
| \cL_{u, \omega} f - \cL_{u, 0} f |_w \le C \| f \|_{\cB} d_{\bar {\mathcal{F}}}(\bar T_\omega, \bar T_0)^{\gamma/2}, \quad \forall \, f \in \cB .
\]
\end{lem}
\begin{proof}
This lemma for $u=0$ is proved in \cite[Theorem~2.3]{MarkHongKun2013}.  As in the proof
of Lemma~\ref{lem:uniform ly}, we must show that the relevant estimates are independent of $u$.
The relevant estimate is eq. (5.1) of \cite{MarkHongKun2013}:  Let
$W \in \cW^s$, $f \in \cC^1(\bar M_0)$ and $\psi \in \cC^p(W)$ with $|\psi|_{\cC^p(W)} \le 1$.  Then,
\[
\begin{split}
\int_W (\cL_{u, \omega} f & - \cL_{u, 0} f) \psi \, dm_W
= \int_W \cL_{u, \omega} f \, \psi \, dm_W - \int_W \cL_{u,0} f \, \psi \, dm_W \\
& = \int_{\bar T_\omega^{-1} W} f \, e^{i u \cdot \Phi_\omega} \, \psi \circ \bar T_\omega
 J_{\bar T_\omega^{-1} W} \bar T_\omega \, dm_W 
 -  \int_{\bar T_0^{-1} W} f \, e^{i u \cdot \Phi_0} \, \psi \circ \bar T_0
 J_{\bar T_0^{-1} W} \bar T_0 \, dm_W .
\end{split}
\]
Following \cite{MarkHongKun2013}, we split $\bar T_\omega^{-1}W$ and $\bar T_0^{-1}W$ into
matched pieces (which can be connected by a transverse foliation of unstable curves) and
unmatched pieces which cannot.  On the unmatched pieces, we use \eqref{eq:same}
to note that the estimate \cite[eq.~(5.3)]{MarkHongKun2013} remains identical.  Similarly,
since matched pieces lie in the same connected component of
$\bar M_0 \setminus (\cS_1^{\bar T_\omega} \cup \cS_1^{\bar T_0})$, we have
$\Phi_\omega = \Phi_0$ on such components.  Thus factoring out the common constant
$e^{i u \cdot \Phi_0}$ from the difference on matched pieces, we see that
the estimate \cite[eq.~(5.4)]{MarkHongKun2013} holds without any changes.  Thus the final
estimate \cite[eq.~(5.9)]{MarkHongKun2013} holds with the same constants, independently of
$u \in \mathbb{R}^2$.
\end{proof}
\begin{proof}[Proof of Proposition \ref{lem:close2}]
This comes directly from Lemmas \ref{lem:close} and \ref{tildeBnorm}. Indeed, for every $f\in\widetilde{\mathcal B}$, we have
\begin{eqnarray*}
\sup_{\underline{\omega}\in E^{\mathbb N}}\left|
    (P_u-\mathcal P_u)f (\cdot, \uomega) \right|_w &=&
    \sup_{\underline{\omega}\in E^{\mathbb N}}\left|
        \int_E\left(\mathcal L_{u,\omega_{-1}}-\mathcal L_{u,0}\right)f(\cdot,(\omega_{-1},\underline{\omega}))\, d\eta(\omega_{-1})\right|_w\\
&\le&     \sup_{\underline{\omega}\in E^{\mathbb N}}\int _E\left|
        \left(\mathcal L_{u,\omega_{-1}}-\mathcal L_{u,0}\right)f(\cdot,(\omega_{-1},\underline{\omega}))\right|\, d\eta(\omega_{-1})\\
&\le&C\sup_{\underline{\omega'}\in E^{\mathbb N}}
\Vert f(\cdot, \underline{\omega'})\Vert_{\mathcal B}\, \vartheta_0^{\frac\gamma 2}= C
\Vert f\Vert_{\widetilde{\mathcal B}}  \,   \vartheta_0^{\frac\gamma 2} \, ,
\end{eqnarray*}
since $\bar T_\omega \in \bar {\mathcal{F}}_{\theta_0}(\bar T_0)$.
\end{proof}

\begin{lem}\label{quasicompact1}
${\mathcal P}$ is quasicompact, 1 its is
only dominating eigenvalue and it is a simple eigenvalue
(with eigenspace $\mathbb C.\bar \mu$). In particular, there exists $\tilde C>0$ and $\tilde\alpha\in(0,1)$ such that
$$\forall f\in\widetilde{\mathcal B},\quad
          \Vert { \mathcal P}^n f-\mathbb E_{\bar\mu}[f]\mathbf 1_{\bar M}\Vert_{\widetilde {\mathcal B}}\le
      \tilde C\tilde\alpha^n\Vert f\Vert_{\widetilde{\mathcal B}}.$$
\end{lem}
\begin{proof}
Due to \cite[Theorem 2.2 and Corollary 2.4]{MarkHongKun2013}) $\cL_0$ is quasicompact, 1
is its only dominating eigenvalue and it is a simple eigenvalue
(with eigenspace $\mathbb C.\mathbf 1_{\bar M_0}$).
In particular, there exists $\tilde C>0$, $\tilde \alpha_0 \in(0,1)$ such that
$$\forall h\in\mathcal B,\quad
     \Vert\mathcal L_0^n h-\mathbb E_{\bar\mu_0}[h]\mathbf 1_{\bar M_0}\Vert_{\mathcal B}\le \tilde C\tilde\alpha_0^n  \Vert h\Vert_{\mathcal B}\, .$$
Let $f \in\widetilde{\mathcal B}$.
Observe that
$${ \mathcal P}^n(f)(y,(\omega_k)_{\ge 0})=\int_{E^n}\mathcal L^n_{0} f(\cdot,(\omega_{k-n})_{k})(y)\,  d\eta^{\otimes n}(\omega_{-n},...,\omega_{-1}) \, $$
and that
$$ \mathbb E_{\bar\mu}[f]=\int_{\mathbb E^{\mathbb N}}\mathbb E_{\bar\mu_0}[f(\cdot,\underline{\omega'})]\, d\eta^{\otimes\mathbb N}(\underline{\omega'})\, .$$
First, setting $\tilde \omega_n = (\omega_{-n}, \ldots, \omega_{-1})$, we have, using Lemma~\ref{lem:piecewise}(a), 
 \begin{eqnarray*}
&\ &\sup_{\underline{\omega}}\Vert\mathcal P^n(f)(\cdot,{\underline{\omega}})-\mathbb E_{\bar\mu} [f]\mathbf 1_{\bar M}\Vert_{\mathcal B}= \sup_{\underline\omega}\left\Vert
\int_{E^n}(\mathcal L_0^n f(\cdot,(\omega_{k-n})_{k})- \mathbb{E}_{\bar \mu}[f])
\,  d\eta(\omega_{-1})...d\eta(\omega_{-n})\right\Vert_{\mathcal B}\\
&\le &\sup_{\underline{\omega}}\left\Vert
\int_{E^n}\mathcal L_0^n f(\cdot,(\omega_{k-n})_{k})-\mathbb E_{\bar\mu_0}\left[f(\cdot,(\omega_{k-n})_{k})\right]\,  d\eta(\omega_{-1})...d\eta(\omega_{-n})\right\Vert_{{\mathcal B}}\\
&\ &+ \, \Vert \mathbf 1_{\bar M_0}\Vert_{\mathcal B}\sup_{\underline{\omega}}\left\vert\int_{E^n}\left(\mathbb E_{\bar\mu_0}\left[f(\cdot,(\tilde\omega_n,\underline{\omega}))\right]-\int_{E^{\mathbb{N}}} \mathbb E_{\bar\mu_0}(f(\cdot,(\tilde\omega_n,\underline{\omega}')))d\eta^{\otimes\mathbb N}(\underline{\omega}')\right)\,  d\eta^{\otimes n}(\tilde\omega_n)\right\vert\\
&\le &\sup_{\underline{\omega}}
\int_{E^n}\left\Vert\mathcal L_0^n f(\cdot,(\omega_{k-n})_{k})-\mathbb E_{\bar\mu_0}[f(\cdot,(\omega_{k-n})_{k})]\right\Vert_{\mathcal B}\,  d\eta(\omega_{-1})...d\eta(\omega_{-n})\\
&\ &+\Vert \mathbf 1_{\bar M_0}\Vert_{\mathcal B}\sup_{\underline{\omega},\underline{\omega}'\, :\, d(\underline{\omega},\underline{\omega}')<\varkappa^{-n}}\left\vert\mathbb E_{\bar\mu_0}\left[f(\cdot,\underline{\omega})-f(\cdot,\underline{\omega}')\right]\right\vert\\
&\le&\tilde C\tilde \alpha_0^n\int_{E^n}\left\Vert f(\cdot,(\omega_{k-n})_{k})\right\Vert_{\mathcal B}\,  d\eta(\omega_{-1})...d\eta(\omega_{-n})\\
&\ &+\Vert \mathbf 1_{\bar M_0}\Vert_{\mathcal B}\Vert\mathbb E_{\bar\mu_0}[\cdot]\Vert_{\mathcal B'}\sup_{\underline{\omega},\underline{\omega}'\, :\, d(\underline{\omega},\underline{\omega}')<\varkappa^{-n}}\left\Vert f(\cdot,\underline{\omega})-f(\cdot,\underline{\omega}')\right\Vert_{\mathcal B}\\
&\le&(\tilde C\tilde \alpha_0^n+C_1\varkappa^{-n})\Vert f\Vert_{\widetilde{\mathcal B}},
\end{eqnarray*}
since $\mathbf 1_{\bar M_0}$ is in $\mathcal B$ and $\mathbb E_{\bar\mu_0}[\cdot]$ is in the dual of $\mathcal B$
by Remark~\ref{Rke0}.

Second, for every $\underline{\omega}$ and
$\underline{\omega}'$ in $E^{\mathbb N}$, we have
 \begin{eqnarray*}
&\ &\Vert{\mathcal P}^n(f)(\cdot,{\underline{\omega}})-{\mathcal P}^n(f)(\cdot,{\underline{\omega}'})\Vert_{\mathcal B}= \left\Vert
\int_{E^n}\left(\mathcal L_0^n f(\cdot,(\tilde\omega_n,\underline{\omega}))-\mathcal L_0^n f(\cdot,(\tilde\omega_n,\underline{\omega}))\right)
\,  d\eta^{\otimes n}(\tilde\omega_n)\right\Vert_{\mathcal B}\\
&\le&
\int_{E^n}\left\Vert \mathcal L_0^n\left( f(\cdot,(\tilde\omega_n,\underline{\omega}))- f(\cdot,(\tilde\omega_n,\underline{\omega}))\right)\right\Vert_{\mathcal B}
\,  d\eta^{\otimes n}(\tilde\omega_n)\\
&\le&\sup_{\underline{\omega}^{(1)},\underline{\omega}^{(2)}\, :\, d(\underline{\omega}^{(1)},\underline{\omega}^{(2)})<d(\underline{\omega},\underline{\omega}')\varkappa^{-n}}\left\vert \mathbb E_{\bar\mu_0}\left[f(\cdot,\underline{\omega}^{(1)})-f(\cdot,\underline{\omega}^{(2)})\right]\right\vert
+\tilde C\tilde\alpha_0^n \Vert f\Vert_{\widetilde{\mathcal B}}d(\underline{\omega},\underline{\omega}')\varkappa^{-n}\\
&\le&\left\Vert \mathbb E_{\bar\mu_0}\left[\cdot\right]\right\Vert_{\mathcal B'}\sup_{\underline{\omega}^{(1)},\underline{\omega}^{(2)}\, :\, d(\underline{\omega}^{(1)},\underline{\omega}^{(2)})<d(\underline{\omega},\underline{\omega}')\varkappa^{-n}}\left\Vert f(\cdot,\underline{\omega}^{(1)})-f(\cdot,\underline{\omega}^{(2)})\right\Vert_{\mathcal B}
+\tilde C\tilde \alpha_0^n \Vert f\Vert_{\widetilde{\mathcal B}}d(\underline{\omega},\underline{\omega}')\varkappa^{-n}\\
&\le&\left\Vert \mathbb E_{\bar\mu_0}\left[\cdot\right]\right\Vert_{\mathcal B'} \Vert f\Vert_{\widetilde  B}
d(\underline{\omega},\underline{\omega}')\varkappa^{-n}
+\tilde C\tilde \alpha_0^n \Vert f\Vert_{\widetilde{\mathcal B}}d(\underline{\omega},\underline{\omega}')\varkappa^{-n}\, .
\end{eqnarray*}
This proves the lemma with $\tilde\alpha = \max \{ \tilde \alpha_0, \varkappa^{-1} \}$.
\end{proof}
\subsection{Doeblin-Fortet-Lasota-Yorke type inequality for $P_u$}

We next establish the spectral properties of $P$ and $P_u$ on $\widetilde \cB$.

\begin{prop}\label{DFLY}
If $\vartheta_0$ is small enough, there exist $\tilde C >0$ and $\tilde \tau \in(0,1)$, such that for every $n\ge 1$, 
$f \in \widetilde \cB$, 
$u \in \mathbb{R}^2$ and $n \ge 0$,
\begin{eqnarray}
| P_u^n f |_{\widetilde{\mathcal B}_w} & \le & \tilde C  |f|_{\widetilde{\mathcal B}_w} , \nonumber \\
\|P_u^nf \|_{\widetilde{\mathcal B}} & \le & \tilde C \left(\tilde \tau^n\Vert f\Vert_{\widetilde{\mathcal B}}+|f|_{\widetilde{\mathcal B}_w}\right)\, .
\end{eqnarray}
\end{prop}

This result will follow directly from the next lemma.
\begin{lem}
\label{lem:uniform ly}
If $\vartheta_0$ is small enough, there exist $C >0$ and $\tau \in(0,1)$, such for every $n\ge 1$, $ \varepsilon_1,...,\varepsilon_{n}\in E$, $f \in \cB$, $u \in \mathbb{R}^2$ and $n \ge 0$,
\begin{eqnarray}
| \cL_{u, \omega_1}\cdots \cL_{u, \omega_n} f |_w & \le & C |f|_w , \nonumber \\
\|\cL_{u, \omega_1}\cdots \cL_{u, \omega_n}f \|_{\mathcal B} & \le & C \left(\tau^n\Vert f\Vert_{\mathcal B}+|f|_w\right)\, .
\end{eqnarray}
\end{lem}

\begin{proof}
Here we denote  $\cL_{u, \underline{\omega}}^n:=\cL_{u, \omega_1}\cdots \cL_{u, \omega_n} $, and $\bar T^n_{\underline{\omega}} =\bar T_{\omega_n}\circ \cdots \circ\bar T_{\omega_1}$.
The above Lasota-Yorke inequalities are proved\footnote{The estimates
in \cite[Proposition~5.6]{MarkHongKun2013} include a factor $\eta \ge 1$, which comes from the
Jacobian of $\bar T_\omega$ with respect to $\bar \mu_0$.  Since we have assumed that 
$J_{\bar \mu_0}\bar T_\omega = 1$ in our simplified version of {\bf (H5)}, we have $\eta =1$ in the present
setting.  Also note that the density function $g$ for the random perturbation in \cite{MarkHongKun2013}
is identically 1 in our setting as well.} for $\cL^n_{\underline{\omega}}$ as long as each $\bar T_{\omega_k} \in \bar {\mathcal{F}}$ by \cite[Proposition~5.6]{MarkHongKun2013}, with $\underline{\omega}=(\omega_k)_{k\geq 1}$.   We must show that the
constants appearing in the inequalities are independent of $u \in \mathbb{R}^2$, and all $\underline{\omega}\in E^{\mathbb{N}}$.
We will use the fact that $S_n \Phi_{\underline{\omega}}$ is constant on elements of $\bar M_0 \setminus \cS_n^{\bar T_{\underline{\omega}}}$.

For $f \in \cB$, $W \in \cW^s$ and appropriate test functions $\psi$, we must estimate expressions of the form,
\[
\int_W \cL_{u, \underline{\omega}}^n f \, \psi \, dm_W = \sum_{W_i \in \cG_n(W)} \int_{W_i} f e^{i u \cdot S_n \Phi_{\underline{\omega}}}  J_{W_i}\bar T_{\underline{\omega}}^n \, \psi \circ \bar T_{\underline{\omega}}^n \, dm_{W_i},
\]
where $\cG_n(W)$ are the components of $\bar T_{\underline{\omega}}^{-n}W$, subdivided so that they each
belong to $\cW^s$.

For example, for the weak norm estimate,
\[
\int_W \cL_{u, {\underline{\omega}}}^n f \, \psi \, dm_W
\le \sum_{W_i \in \cG_n(W)}
| f |_w |e^{i u \cdot S_n \Phi_{\underline{\omega}}} \psi \circ \bar T_{\underline{\omega}}^n|_{\cC^p(W_i)} |J_{W_i}\bar T_{\underline{\omega}}^n|_{\cC^p(W_i)} ,
\]
which is the same as \cite[eq. (4.5)]{MarkHongKun2013}, except that the test function
is $e^{i u \cdot S_n \Phi_{\underline{\omega}}} \psi \circ \bar T_{\underline{\omega}}^n$ rather than simply $\psi \circ \bar T_{\underline{\omega}}^n$.  (See also the generalization of
this inequality 
for the random sequence of maps in \cite[eq. (5.26)]{MarkHongKun2013}.)
But since $e^{i u \cdot S_n \Phi_{\underline{\omega}}}$ is constant on each $W_i$, we have
\begin{equation}
\label{eq:same}
|e^{i u \cdot S_n \Phi_{\underline{\omega}}} \psi \circ \bar T_{\underline{\omega}}^n|_{\cC^p(W_i)}
= |e^{i u \cdot S_n \Phi_{\underline{\omega}}} |_\infty | \psi \circ \bar T_{\underline{\omega}}^n|_{\cC^p(W_i)}
= | \psi \circ \bar T_{\underline{\omega}}^n|_{\cC^p(W_i)},
\end{equation}
so that the estimate is precisely the same and independent of $u$, as (\textbf{H1}) and (\textbf{H5}) are used for  $|J_{W_i}\bar T_{\underline{\omega}}^n|_{\cC^p(W_i)}$ to get the uniform bound independent of
${\underline{\omega}}$.

The same observation is true for the strong norm estimates for precisely the same reason.
For the unstable norm estimate, we must compare values of test functions on two stable
curves $W^1, W^2$ that lie close together.
But the `matched' pieces\footnote{These are curves $U^1_j \in \cG_n(W^1)$ and
$U^2_j \in \cG_n(W^2)$ that can be connected
by a transverse foliation of unstable curves.  See \cite[Section 4.3]{MarkHongKun2013} for
a precise definition.} of $\bar T_{\underline{\omega}}^{-n} W^1$ and $\bar T_{\underline{\omega}}^{-n} W^2$ lie in the same
component of $\bar M_0 \setminus \cS_n^{\bar T_{\underline{\omega}}}$ so that $S_n \Phi_{\underline{\omega}}$ is the same constant
on both curves and does not affect any of the constants appearing in the Lasota-Yorke inequalities.
For the `unmatched' pieces of  $\bar T_{\underline{\omega}}^{-n} W^1$ and $\bar T_{\underline{\omega}}^{-n} W^2$,
the estimate is precisely the same as in \cite[eq. (4.14)]{MarkHongKun2013}
due to \eqref{eq:same}.
\end{proof}
\begin{proof}[Proof of Proposition \ref{DFLY}]
Observe that
$$ (P_u^n f)(\cdot,\uomega)
        = \int_{E^n}    \mathcal L_{u,\omega_{-n}}\cdots\mathcal L_{u,\omega_{-1}} f(\cdot,(\omega_{k-n})_{k\ge 0})(y)\, d\eta^{\otimes n}(\omega_{-n},...,\omega_{-1})\, .$$
Due to Lemma \ref{tildeBnorm} and to the first inequality of Lemma \ref{lem:uniform ly}, for any $f\in\widetilde{\mathcal B}_w$ and $n\ge 1$,
\begin{eqnarray*}
\left\vert P_u^nf \right\vert_{\widetilde{\mathcal B}_w}&=&\sup_{\uomega\in E^{\mathbb N}}
       \left\vert (P_u^n f)(\cdot,\uomega)\right\vert_{w}\\
&\le&
\sup_{\uomega\in E^{\mathbb N}}
       \int_{E^n}   \left\vert \mathcal L_{u,\omega_{-n}}\cdots\mathcal L_{u,\omega_{-1}} f(\cdot,(\omega_{k-n})_{k\ge 0})\right\vert_{w}\, d\eta^{\otimes n}(\omega_{-n},...,\omega_{-1})\\
&\le& C
\sup_{\uomega'\in E^{\mathbb N}} \left\vert f(\cdot,\uomega')\right\vert_{w}=C\vert f\vert_{\widetilde{\mathcal B}_w}\, .
\end{eqnarray*}
Analogously, using again Lemma \ref{tildeBnorm} and, this time, the second inequality of 
Lemma~\ref{lem:uniform ly}, we obtain, for any $f\in\widetilde{\mathcal B}$ and $n\ge 1$,
\begin{eqnarray*}
\sup_{\uomega\in E^{\mathbb N}}
       \left\Vert (P_u^n f)(\cdot,\uomega)\right\Vert_{\mathcal B}
\le  C\left(\tau^n
\sup_{\uomega\in E^{\mathbb N}} \left\Vert f(\cdot,\uomega)\right\Vert_{\mathcal B}+
\sup_{\uomega\in E^{\mathbb N}} \left\vert f(\cdot,\uomega)\right\vert_{w}\right) .
\end{eqnarray*}
Finally, using Lemma \ref{tildeBnorm},
\begin{eqnarray*}
&\ &      \sup_{\uomega\ne\uomega'}\frac{\Vert P_u^nf(\cdot,\omega)- P_u^nf(\cdot,\uomega')\Vert_{\mathcal B}}
     {d(\uomega,\uomega')}\\
    &=& \sup_{\uomega\ne\uomega'}\frac{\Vert    \int_{E^n} \mathcal L_{u,\omega_{-n}}\cdots\mathcal L_{u,\omega_{-1}}\left(f(\cdot,(\tilde\omega,\uomega))-f(\cdot,(\tilde\omega,\uomega'))\right)\, d\eta^{\otimes n}(\tilde\omega)\Vert_{\mathcal B}}
     {d(\uomega,\uomega')}  \\
&\le & \sup_{\uomega\ne\uomega'}\frac{    \int_{E^n} \Vert\mathcal L_{u,\omega_{-n}}\cdots\mathcal L_{u,\omega_{-1}}\left(f(\cdot,(\tilde\omega,\uomega))-f(\cdot,(\tilde\omega,\uomega'))\right)\Vert_{\mathcal B}\, d\eta^{\otimes n}(\tilde\omega)}
     {d(\uomega,\uomega')}  \\
&\le &     \int_{E^n} \sup_{\uomega\ne\uomega'}\frac{\Vert\mathcal L_{u,\omega_{-n}}\cdots\mathcal L_{u,\omega_{-1}}\left(f(\cdot,(\tilde\omega,\uomega))-f(\cdot,(\tilde\omega,\uomega'))\right)\Vert_{\mathcal B}}{d(\uomega,\uomega')}\, d\eta^{\otimes n}(\tilde\omega)
       \\
&\le &     \int_{E^n} \sup_{\uomega_0\ne\uomega'_0}\frac{\Vert\mathcal L_{u,\omega_{-n}}\cdots\mathcal L_{u,\omega_{-1}}\left(f(\cdot,\uomega_0)-f(\cdot,\uomega'_0)\right)\Vert_{\mathcal B}
}
{\varkappa^{n}d(\uomega_0,\uomega'_0)}\, d\eta^{\otimes n}(\tilde\omega)
       \\
&\le& \varkappa^{-n}
     C\sup_{\uomega\ne\uomega'}
       \frac{\Vert \mathcal L_{\omega_n}\cdots \mathcal L_{\omega_1}(h(\cdot,\uomega)-h(\cdot,\uomega'))\Vert_{\mathcal B}}
         {d(\uomega,\uomega')}\\
&\le&\varkappa^{-n}\sup_{\omega\in E}\Vert \mathcal L_\omega\Vert_{L(\mathcal B,\mathcal B)}^n
   \sup_{\uomega\ne\uomega'}
       \frac{\Vert h(\cdot,\uomega)-h(\cdot,\uomega')\Vert_{\mathcal B}}
         {d(\uomega,\uomega')}.
\end{eqnarray*}
since $\varkappa>\sup_{\omega\in E}\Vert \mathcal L_\omega\Vert_{L(\mathcal B,\mathcal B)}$, we obtain that
$P_u$ satisfies Doeblin-Fortet-Lasota-Yorke conditions
for $(\widetilde{\mathcal B}$ and $\widetilde{\mathcal B}_w)$.
\end{proof}
\subsection{Quasicompactness of $P$ and  of $P_u$}
\begin{prop}\label{Quasicompact}
If $\vartheta_0$ and $\varkappa^{-1}$ are small enough,
${ P}$ is quasicompact on $\widetilde{\mathcal B}$, 1 its is
only dominating eigenvalue and it is a simple eigenvalue
(with eigenspace $\mathbb C.\mathbf 1_{\bar M}$). In particular, there exists $\tilde C>0$ and $\tilde \alpha\in(0,1)$,
 such that
$$\forall f\in\widetilde{\mathcal B},\quad
          \Vert{P}^n f-\mathbb E_{\bar\mu}[f]\mathbf 1_{\bar M}\Vert_{\widetilde{\mathcal B}}\le
      \tilde C\tilde\alpha^n\Vert f\Vert_{\widetilde{\mathcal B}}\, .$$
\end{prop}

\begin{proof}
This comes from the Keller-Liverani perturbation theorem \cite[Corollary 1]{kellerliverani99} thanks to 
Lemma~\ref{quasicompact1}  and Propositions \ref{lem:close2} and \ref{DFLY}.
Observe moreover that, since $P$
is the dual operator of $f\mapsto f\circ \bar T$, the spectral radius of 
$P$ is 1 and 1 is an eigenvalue of $P$.
We conclude that 1 is the dominating eigenvalue and that it is simple.
\end{proof}

\begin{prop}\label{analytic}
$P_u$, as an operator acting on $\widetilde\cB$, is an analytic perturbation of $P$.
\end{prop}
\begin{proof}
Observe that the $n$-th derivative of $u\mapsto P_u$ is the operator defined by
$$ f\mapsto i^nP\left(\Phi^{(i_1)}...\Phi^{(i_n)}e^{iu\cdot\Phi}f\right)$$
Due to  Lemma \ref{lem:PuOp} and to classical results on analytic functions, we conclude that,  in $L(\widetilde\cB,\widetilde \cB)$,
$u\mapsto P_u$ is analytic on $\mathbb R^2$ and that
\[
P_u=\sum_{n=0}^\infty\frac{1}{n!} \cA_{n,\omega}, \quad\mbox{with}\
\cA_{n} f(u)= P ((i u \cdot \Phi)^n f) \, ,
\]
where $\cA_{n} f(u)$ is $n$-linear in $u$.
\end{proof}
Our main results will follow from the following technical result.
\begin{theo}\label{prop:perturb}
The function $\mathbf 1_{\bar M}$ is in $\widetilde\cB$ and $\mathbb E_{\bar\mu}[\cdot]$ is a continuous linear form on $\widetilde\cB$ and $\widetilde\cB_w$.\\
If $\vartheta_0$ and $\varkappa^{-1}$ are small enough, there exist $\beta\in(0,\pi)$, $C>0$ and $\alpha\in(0,1)$, three analytic maps $u\mapsto \lambda_u$ from $[-\beta,\beta]^2$ to $\mathbb C$,
$u\mapsto N_u$ and $u\mapsto \Pi_u$ from $[-\beta,\beta]^2$ to $L(\widetilde\cB, \widetilde\cB)$ such that
\begin{itemize}
\item[a)] $\lambda_0=1$, $\Pi_0:=\mathbb E_{\bar\mu}[\cdot]\mathbf 1_{\bar M}$,
\item[b)] for every $u\in[-\beta,\beta]^2$ and every integer $n\ge 1$, $P_u^n=\lambda_u^n\Pi_u+N_u^n$, $\Pi_uN_u=N_u\Pi_u=0$,
$\Pi_u^2=\Pi_u$, and $\Vert N_u^n\Vert_{L(\widetilde \cB,\widetilde\cB)}\le C\alpha^n$.\\
Moreover, for every integer $k\ge 0$, $\Vert (N_u^n)^{(k)}\Vert_{L(\widetilde \cB,\widetilde\cB)}=O(\alpha^n)$, where $(N_u^n)^{(k)}$ means the $k$-th derivative of $N_u^n$.
\item[c)] for every $u\in[-\pi,\pi]^2\setminus[-\beta,\beta]^2$ and every integer
$n\ge 1$, we have $\Vert P_u^n\Vert_{L(\widetilde \cB,\widetilde\cB)}\le C\alpha^n$.
\item[d)] There exists a positive symmetric matrix $\Sigma^2$ such that
$\lambda_u=1-\frac 12 (\Sigma^2 u\cdot u)+
O(|u|^3)$.
\end{itemize}
\end{theo}
\begin{proof}[Proof of Theorem \ref{prop:perturb}]
The fact that $\mathbf 1_{\bar M}$ is in $\widetilde\cB$ comes from
the fact that $\mathbf 1_{\bar M_0}$ is in $\cB$.\\
As seen in Remark \ref{Rke1},  $\mathbb E_{\bar\mu}[\cdot]$ is a continuous form on $\widetilde\cB$.
The proof of the remaining part of the theorem relies on Propositions  \ref{lem:close2}, \ref{DFLY} and \ref{Quasicompact} and \ref{analytic}.

Propositions \ref{DFLY}, \ref{Quasicompact} and \ref{analytic} immediately imply the existence of a spectral
gap for $P_u$ for $|u|$ sufficiently small, using standard perturbation theory \cite{kato}.  This yields the analyticity and items (a) and (b) of the
proposition with $\beta$ depending on $\vartheta_0$ and the uniform constants depending on the
family $\bar {\mathcal{F}}_{\vartheta_0}$, but not on the probability measure $\eta$.

For item (c), due to \cite[Lemma 4.3]{AaronsonDenker}, it is enough to prove that, if $\vartheta_0$ is small enough, then
for every $u\in [-\pi,\pi]^2\setminus[-\beta,\beta]^2$, $P_u$
admits no eigenvalue of modulus 1.
Assume the contrary. There would exist a sequence
of operators $(P_{u_k}^{(k)})_k$ corresponding to a sequence of vanishing neighbourhoods $(E_k)_k$ of $\bar T_0$ in $\mathcal F$ and with $\beta\le |u_k|\le\pi$ and $\rho(P_{u_k}^{(k)})=1$,
where $\rho(\cdot)$ denotes the spectral radius. Up to extracting a subsequence, we also have $\lim_{k\rightarrow +\infty}u_k=u_\infty$. But, due to Proposition \ref{lem:close2} and since
$u\mapsto\mathcal L_{u,0}$ is continuous from $\mathbb R^2$ to $ L(\mathcal B,\mathcal B)$, we would deduce that $$\lim_{k\rightarrow +\infty}\Vert P_{u_k}^{(k)}-\mathcal{P}_{u_\infty}\Vert_{L(\mathcal B,\mathcal B_w)}=0.$$ Combining this with Proposition
\ref{DFLY} and with the perturbation theorem of \cite{kellerliverani99}, this would imply that $\rho(\mathcal P_{u_\infty})=1$, which would contradict Proposition \ref{lem:nonarithm2}.
We conclude that, as soon as $\vartheta_0$ is sufficiently small,
$\sup_{\beta\le |u|\le\pi}\rho( P_u)<1$ as claimed.

It remains to prove item (d).
Due to \cite[Corollary 2.4]{MarkHongKun2013}, for any initial probability measure
$\nu\in\mathcal B$, $(S_n/\sqrt {n})_n$
converges in distribution to a (possibly generalized) centered Gaussian random variable
with variance $\Sigma^2$. Moreover, due to item (b) of the
present theorem,
$$\sup_{t\in[-\beta,\beta]^2}|\mathbb E_{\bar \mu}[e^{it S_n}]
-\lambda_t^n\mathbb E_{\bar \mu}[\Pi_t(\mathbf 1)]|=O(\alpha^n)$$ and so $$\lim_{n\rightarrow +\infty}\lambda_{t/\sqrt{n}}^n= e^{-\frac 12(\Sigma^2 t\cdot t)}$$
with uniform convergence on any compact set of $\mathbb R^2$.
This implies that 
$$\lim_{n\rightarrow +\infty}n\log \big( \lambda_{t/\sqrt{n}} \big) = -\frac 12(\Sigma^2 t\cdot t) \, . $$
On the other hand, $\log (\lambda_{t/\sqrt{n}}) \sim
    (\lambda_{t/\sqrt{n}}-1)$ as $n\rightarrow +\infty$.
Hence  $$\lim_{n\rightarrow +\infty}n(\lambda_{t/\sqrt{n}}-1)= -\frac 12(\Sigma^2 t\cdot t).$$ Setting $u=t/\sqrt{n}$, we can then deduce the stated Taylor expansion since $u \mapsto \lambda_u$ is
analytic.
The positivity of $\Sigma^2$ follows from the next lemma.
\end{proof}

\begin{lem}\label{Sigma>0}
If $\vartheta_0$ is small enough, $\Sigma^2$ is positive.
\end{lem}
\begin{proof}
Recall that $\Sigma^2$ has been defined in \eqref{Sigma2}. We consider $\Sigma_0^2$ being defined by
\begin{equation}\label{Sigma02}
\Sigma_0^2:=\left(\mathbb E_{\bar\mu}\left[\Phi_0^{(i)}.\Phi_0^{(j)}\right]+\sum_{k\ge 1}\mathbb E_{\bar\mu_0}\left[\Phi_0^{(i)}.\Phi_0^{(j)}\circ \bar T_0^k+\Phi_0^{(j)}.\Phi_0^{(i)}\circ \bar T_0^k\right]\right)_{i,j=1,2}\, .
\end{equation}
It is enough to prove that $\Sigma^2$ converges to $\Sigma_0^2$ as
$\vartheta_0$ goes to $0$.
We use \eqref{Sigma2} together with the fact that $\Sigma^2_0$
satisfies an analogous formula (with $\Phi(x,\underline\omega)$ replaced by $\Phi_0(x)$ and with $\bar T(x,\underline\omega)$ replaced by $\bar T_0(x)$.
Therefore
$$\Sigma^2-\Sigma_0^2= A_0+2\sum_{k\ge 1}A_k\, ,$$
with
$A_k:=\mathbb E_{\bar\mu}\left[
 \Phi.\Phi\circ \bar T^k]-\mathbb E_{\bar\mu_0}[\Phi_0.\Phi_0\circ\bar T_0^k\right].$
Extending the definition of $\Phi_0$ on $\bar M$ by setting $\Phi_0(x,\uomega):=\Phi_0(x)$, we obtain
\begin{eqnarray*}
A_k&:=&\mathbb E_{\bar\mu}\left[
 \Phi.\Phi\circ \bar T^k-\Phi_0.\Phi_0\circ(\bar T'_0)^k\right] \\
&=&\mathbb E_{\bar\mu}\left[
 P^k\Phi.\Phi-\mathcal P^k\Phi_0.\Phi_0\right] \\
&=&\mathbb E_{\bar\mu}\left[
 (\Phi-\Phi_0). P^k\Phi\right] +\mathbb E_{\bar\mu}\left[
 \Phi_0.P^k(\Phi-\Phi_0)\right]+\mathbb E_{\bar\mu}\left[
 \Phi_0.(P^k\Phi_{0}-\mathcal P^k\Phi_0)\right]\, .
\end{eqnarray*}
The two first terms of the right hand side of this formula are less than
$$ 4 \Vert\Phi\Vert_\infty^2\sup_{\omega\in E}\bar\mu_0(\Phi_\omega - \Phi_0 \ne \mathbf{0} ) \, , $$
which goes to 0 as $\vartheta_0\rightarrow 0$.
The third term is dominated by
$$k\, \max\left(\Vert P\Vert_{L(\widetilde\cB,\widetilde\cB)},\Vert \mathcal P\Vert_{L(\widetilde\cB_w,\widetilde\cB_w)}\right)^{k-1} \Vert P-\mathcal P\Vert_{\mathcal L(\widetilde\cB,\widetilde\cB_w)}\Vert\Phi_0\Vert_{\widetilde\cB} \Vert\mathbb E_{\bar\mu}[\Phi_0\cdot]\Vert_{\widetilde\cB_w'}\, .$$
We deduce that this quantity goes to 0
using Remark~\ref{PhiinB}, Lemma~\ref{lem:PuOp}, and Proposition~\ref{lem:close2},  and since $\mathbb E_{\bar\mu}[\Phi_0\cdot]=\mathbb E_{\bar\mu}[\mathcal P(\Phi_0\cdot)]$ (applying Lemma \ref{lem:PuOp} with $E=\{0\}$).

We conclude with the use of the dominated convergence theorem, since
\begin{eqnarray*}
\left|\mathbb E_{\bar\mu}\left[
 \Phi.\Phi\circ \bar T^k\right]\right|
&=& \left|\mathbb E_{\bar\mu}\left[
 P^k\Phi.\Phi\right]\right| \\
&\le&  \Vert \Phi\Vert_{\widetilde\cB}
\widetilde C\widetilde\alpha^{k}\Vert\mathbb E_{\bar\mu}[ \Phi\cdot]\Vert_{\widetilde\cB'}=\Vert \Phi\Vert_{\widetilde\cB}
\widetilde C\widetilde\alpha^{k}\Vert\mathbb E_{\bar\mu}[ P(\Phi\cdot)]\Vert_{\widetilde\cB'}\, ,
\end{eqnarray*}
where we used Proposition \ref{Quasicompact} 
since $\mathbb{E}_{\bar \mu}[\Phi] =\mathbf{0}$, and a similar bound holds for
$\mathbb{E}_{\bar \mu_0}[\Phi_0 . \Phi_0 \circ \bar T_0^k]$.
\end{proof}

\section{Limit theorems under general assumptions}\label{proofs}

We start with the proof of our results which are direct consequences
of Proposition \ref{prop:perturb} and of general results existing in the literature.

\begin{proof}[Proofs of Theorems \ref{CLT}, \ref{TCL} and \ref{mixing}]
Theorem \ref{CLT} is a direct corollary of Theorem~\ref{prop:perturb} by standard arguments 
(see \cite{Nag,GH,HH}) since
$\mathbb E_{\bar\mu}\left[e^{iu\cdot S_n}\right]=\mathbb E_{\bar\mu}\left[P_u^n \mathbf 1_{\bar M}\right]$.

Due to Theorem~\ref{prop:perturb}, our dynamical system
$(\bar M,\bar \mu,\bar T)$ satisfies the general assumptions of
\cite[Theorem 1.4]{DamienSoaz} and \cite[Theorem 3.2, Remark 3.3]{SoazDecorr} and so Theorems \ref{TCL} and \ref{mixing} follow.
\end{proof}
We will prove the other results in a general context.
About these results, let us mention that Theorem \ref{theo:range} and the first part of Theorem \ref{theo:randomscenery} have been proved in \cite{DSV1,soazRange} and in \cite{soazPAPA} for a single billiard map.
We give here the proof in a more general context with a significant 
simplification in the proof of Theorem \ref{theo:randomscenery} due to the better 
estimate of the variance of the auto-intersection and to some simplification in the 
Bolthausen tightness argument. The second part of Theorem~\ref{theo:randomscenery} uses 
a general argument from \cite{NJR}. Theorem \ref{theo:selfintersection}
exists for a single billiard map, but only in an unpublished paper
by the second author \cite{soazSelfInter}.
Let us indicate that the generality of the proof we give in the present paper
is possible due to important modifications of the proof (we do not assume that $\Phi$ is bounded, nor 
that the Banach space we consider is continuously injected in $L^p$ for a suitable $p>1$; both of these conditions were used in \cite{soazSelfInter}).\\

We will prove the limit theorems we are interested in under the following general hypothesis.
\begin{hypo}\label{hypo:perturb}
Let $(M,\mu, T)$ be a $\mathbb Z^2$-extension of a probability preserving dynamical system $(\bar M,\bar\mu,\bar T)$ by a function $\Phi:\bar M\rightarrow\mathbb C$. Let $P$ be the transfer operator associated with
$\bar T$ with respect to $\bar\mu$ and let $(P_u:=P(e^{iu\cdot\Phi}\cdot))_{u\in\mathbb R^2}$.
We assume that these operators act on two Banach spaces  $\widetilde\cB_1$ and $\widetilde\cB_2$ such that $\mathbf 1_{\bar M}\in\widetilde\cB_1\hookrightarrow \widetilde\cB_2$ (continuous inclusion) and that 
$\mathbb E_{\bar\mu}[\cdot]$ is a continuous linear form\footnote{up to extending by continuity the definition of $\mathbb E_{\bar\mu}[\cdot]$} on $\widetilde\cB_2$.\\
Assume that there exist $\beta\in(0,\pi)$, $C>0$ and $\alpha\in(0,1)$, three continuous maps $u\mapsto \lambda_u$ from $[-\beta,\beta]^2$ to $\mathbb C$,
$u\mapsto N_u$ and $u\mapsto \Pi_u$ from $[-\beta,\beta]^2$ to $L(\widetilde\cB_1, \widetilde\cB_2)$ such that
\begin{itemize}
\item[(A1)] for every $u\in[-\beta,\beta]^2$ and every integer $n\ge 1$, $$P_u^n=\lambda_u^n\Pi_u+N_u^n,\,\,\,\,\Pi_uN_u=N_u\Pi_u=0,\,\,\,\,\,\,\Pi_u^2=\Pi_u$$
 and $\Vert N_u^n\Vert_{L(\widetilde \cB_1,\widetilde\cB_1)}\le C\alpha^n$.
\item[(A2)] for every $u\in[-\pi,\pi]^2\setminus[-\beta,\beta]^2$ and every integer
$n\ge 1$, we have $\Vert P_u^n\Vert_{L(\widetilde \cB_1,\widetilde\cB_1)}\le C\alpha^n$.
\item[(A3)] $u\mapsto\Pi_u$, seen as a $L(\widetilde\cB_1,\widetilde \cB_2)$-valued function, is differentiable at 0 
and that $\Pi_0:=\mathbb E_{\bar\mu}[\cdot]\mathbf 1_{\bar M}$,
\item[(A4)] There exists a positive symmetric matrix $\Sigma^2$ such that
$\lambda_u=1-\frac 12 (\Sigma^2 u\cdot u)+
O(|u|^3)$.
\end{itemize}
\end{hypo}
We will also use the following notation and considerations.
We write $S_n$ for the ergodic sum $S_n:=\sum_{k=0}^{n-1}\Phi\circ \bar T^k$. It will be crucial to notice that $P_u^n=P^n(e^{iu\cdot S_n}\cdot)$.\\
We consider a partition of $\bar M$ in $I$ subsets $\bar O_1,...,\bar O_I$ of $\bar\mu$ positive measure (corresponding to $(\partial O_i\times S^1)\times E^{\mathbb N}$ in our example). We consider the function $\mathcal I_0$
which, at every $x\in\bar M$, associates the index $\mathcal I_0(x)$
of the atom $\bar O_{\mathcal I_0(x)}$ of the partition containing $x$.
We also define $\mathcal I_k:=\mathcal I_0\circ\bar T^k$.

We remark that our random map $\bar T$ with $\bar T_\omega \in \bar {\mathcal{F}}_{\vartheta_0}(\bar T_0)$ for all $\omega \in E$
satisfies all the items of Assumption~\ref{hypo:perturb} due to Theorem~\ref{prop:perturb}.

\subsection{Proof of the Local Limit Theorem -- Theorem \ref{theo:LLT}}
For every $n\in\mathbb N^*$, $\ell\in\mathbb Z^2$ and  $h\in\widetilde{\cB}_1$, we set:
\begin{equation}\label{EQOP1}
\mathcal H_{\ell,n}h:=P^n\left(\mathbf 1_{\{S_n=\ell\}}h\right)\, .
\end{equation}
Recall that
\begin{equation}\label{EQOP2}
\mathbf 1_{\{k=\ell\}}= \frac{1}{(2\pi)^2} \int_{[-\pi,\pi]^2}e^{i (k-\ell)\cdot u}\, du \,
\end{equation}
where $du$ is understood as $du_1du_2$ for $u = (u_1,u_2) \in \mathbb{R}^2$ (integral with respect to the Lebesgue measure),
which leads us to the following formula
\begin{equation}\label{EQOP3}
\mathcal H_{\ell,n}h=\frac 1{(2\pi)^2}\int_{[-\pi,\pi]^2}e^{-i \ell\cdot u}P_u^nh\, du \, .
\end{equation}

\begin{theo}\label{LLT0}
Assume Assumptions \ref{hypo:perturb}. Then
\[
\sup_{\ell\in\mathbb Z^2}\left\Vert  \mathcal H_{\ell,n}-\frac {e^{-\frac 1{2n}\Sigma^{-2}\ell\cdot \ell}}{2\pi n\sqrt{\det\Sigma^2}}\Pi_0
\right\Vert_{L(\widetilde\cB_1, \widetilde\cB_2)}=O(n^{-\frac 32}) \, .
\]
Moreover, there exists $K_0\ge 1$
such that for every integer $n\ge 0$ and every $\ell\in\mathbb Z^2$,
\begin{equation}\label{intPun}
 \Vert\mathcal H_{\ell, n} \Vert_{L(\widetilde\cB_1,\widetilde\cB_1)}\le \frac 1{(2\pi)^2}\int_{[-\pi,\pi]^2}
    \Vert P_u^n\Vert_{L(\widetilde\cB_1,\widetilde\cB_1)}\, du\le\frac{K_0}{n+1}\, ,
\end{equation}
\end{theo}
\begin{proof}
Up to a change of $\beta$, there exists $a>0$ such that, for every $u\in[-\beta,\beta]^2 $,
$|\lambda_u|\le \exp(-a|u|^2)$.
 Hence, using Assumption \ref{hypo:perturb}, we have the following equalities in $L(\widetilde\cB_1,\widetilde \cB_2)$:
\begin{eqnarray*}
&\ &\frac 1{(2\pi)^2}\int_{[-\pi,\pi]^2}e^{-i \ell\cdot u}P_u^n\, du=\frac 1{(2\pi)^2}\int_{[-\beta,\beta]^2}e^{-i \ell\cdot u}
P_u^n\, du+O(\alpha^n)\nonumber\\
   &=& \frac 1{(2\pi)^2}\int_{[-\beta,\beta]^2}e^{-i \ell\cdot u}
\lambda_u^n\Pi_u\, du+O(\alpha^n)\label{interm}\\
   &=& \frac 1{(2\pi)^2}\int_{[-\beta,\beta]^2}e^{-i \ell\cdot u}
\lambda_u^n(\Pi_0+O(u))\, du+O(\alpha^n)\label{interm}\\
   &=& \frac 1{(2\pi)^2}\int_{[-\beta,\beta]^2}e^{-i \ell\cdot u}
\left(e^{-\frac 12 (\Sigma^2 u.u)}+O(|u|^3)\right)^n\Pi_0+O(e^{-an|u|^2}|u|)\, du+O(\alpha^n)\nonumber\\
&=& \frac 1{(2\pi)^2n}\int_{[-\beta\sqrt{n},\beta\sqrt{n}]^2}
                       e^{-i \ell\cdot \frac v{\sqrt{n}}}
  e^{-\frac 1{2} (\Sigma^2 v.v)}\Pi_0+O\left(ne^{-a(n-1)\frac{|v|^2}n}\frac{|v|^3}{n^{\frac 32}}+e^{-a{|v|^2}}\frac{v}{\sqrt{n}}\right)\, dv+O(\alpha^n)\nonumber\\
&=& \frac 1{(2\pi)^2n}\int_{[-\beta\sqrt{n},\beta\sqrt{n}]^2}e^{-i \ell\cdot \frac v{\sqrt{n}}}
  e^{-\frac 1{2} (\Sigma^2 v.v)}\Pi_0+O\left(e^{-\frac a 2{|v|^2}} 
\frac{|v|^3}{\sqrt{n}}+e^{-a{|v|^2}}\frac{v}{\sqrt{n}}\right)\, dv+O(\alpha^n)\\
&=& \frac 1{(2\pi)^2n}\int_{[-\beta\sqrt{n},\beta\sqrt{n}]^2}e^{-i \ell\cdot 
        \frac v{\sqrt{n}}}
  e^{-\frac 1{2} (\Sigma^2 v.v)}\Pi_0\, dv+O(n^{-\frac 32})\nonumber\\
&=& \frac 1{(2\pi)^2n}\int_{\mathbb R^2}e^{-i \frac{\ell}{\sqrt {n}}\cdot  v}
  e^{-\frac 1{2} (\Sigma^2 v.v)}\Pi_0\, dv+O(n^{-\frac 32})\nonumber\\
&=&\frac {e^{-\frac 1{2n}\Sigma^{-2}\ell\cdot \ell}}{2\pi n\sqrt{\det\Sigma^2}}\Pi_0
+O(n^{-\frac 32}) \, ,
\end{eqnarray*}
where we have changed variables, $v = u \sqrt{n}$, and 
the $O$ are in $L(\widetilde\cB_1,\widetilde\cB_2)$
with uniform bound. This bound is in $ L(\widetilde\cB_1,\widetilde\cB_2)$ and 
not in $ L(\widetilde\cB_1,\widetilde\cB_1)$ because according to Assumption (A3), the map $u\mapsto \Pi_u$ is differentiable from $[-\beta,\beta]$ to $ L(\widetilde\cB_1,\widetilde\cB_2)$ and a priori not from $[-\beta,\beta]$ to $ L(\widetilde\cB_1,\widetilde\cB_1)$.

For the second estimate, we write
\begin{eqnarray*}
\frac 1{(2\pi)^2}\int_{[-\pi,\pi]^2}\Vert P_u^n\Vert_{\mathcal L(\widetilde\cB_1,\widetilde\cB_1)}\, du
   &=& \frac 1{(2\pi)^2}\int_{[-\beta,\beta]^2}
|\lambda_u|^n\Vert\Pi_u\Vert_{\mathcal L(\widetilde\cB_1,\widetilde\cB_1)}\, du+O(\alpha^n)\\
&\le &\frac 1{(2\pi)^2}\int_{[-\beta,\beta]^2} e^{-a|u|^2 n}
\sup_{u\in[-\beta,\beta]}\Vert\Pi_u\Vert_{\mathcal L(\widetilde\cB_1,\widetilde\cB_1)}\, du+O(\alpha^n)\\
&\le &O(n^{-1})\, ,
\end{eqnarray*}
using again the change of variable $v=u\sqrt{n}$.
\end{proof}
Due to Theorem \ref{prop:perturb}, Theorem \ref{theo:LLT} is a direct consequence of the following result.
\begin{cor}\label{TLL00}
Assume Assumption \ref{hypo:perturb}.
Let $f,g:\bar M\rightarrow \mathbb R$ such that $H_g(\cdot):=\mathbb E_{\bar\mu}[g\cdot]\in\widetilde\cB'_2$ and such that $f\in\widetilde\cB_1$. Then
\begin{equation}
\mathbb E_{\bar\mu}\left[f.\mathbf 1_{\{{S}_n=\ell\}}.g\circ \bar T^n\right]=\frac {\exp\left(-\frac{\Sigma^{-2}\ell\cdot \ell}{2n}\right)}{2\pi n\sqrt{\det \Sigma^2}}\mathbb E_{\bar\mu}[f]\mathbb E_{\bar\mu}[g]+ O\left(n^{-\frac 32}\Vert f\Vert_{\widetilde\cB_1}\, \Vert H_g\Vert_{\widetilde\cB'_2}\right)\, .
\end{equation}
\end{cor}
\begin{proof}
Observe that we have
\[
\mathbb E_{\bar\mu}\left[f.\mathbf 1_{\{{S}_n=\ell\}}.g\circ \bar T^n\right]=\mathbb E_{\bar\mu}\left[P^n(f.\mathbf 1_{\{{S}_n=\ell\}}).g\right]=H_g\left(P^n(\mathbf 1_{\{{S}_n=\ell\}}f)\right)
\, ,
\]
recalling \eqref{eq:H_g}. We conclude due to Theorem \ref{LLT0}.
\end{proof}

\subsection{Return time to the original obstacle:  Proof of Theorem \ref{theo:range}}
\label{proofreturn}
Recall that $\mathcal I_k( x)$ corresponds to the index of the
atom $\bar O_{\mathcal I_k(x)}$ containing $\bar T^k x$ and that $S_n(x)$ corresponds to the label of the copy of $\bar M$ in $M$ containing $T^k(x,\mathbf{0})$. 
We also define $\mathcal I_k$ on $\bar M$
by canonical projection.
We consider the set $B_n$
of $x\in\bar M$ such that the orbit $(T^n(x,\mathbf{0}))_{n\ge 0}$
won't return to the initial atom $\bar O_{\mathcal I_0(x)}\times \{\mathbf{0}\}$
until time $n$:
$$
B_n:=\{\forall k=1,...,n\, :\, 
       ({\mathcal I}_k,S_k)\ne (\mathcal I_0,(0,0))\}  \subset\bar M  \, . 
$$
Analogously we define $B'_n$ the set of points
$x\in \bar M$
for which
the atom visited at time $n$ has not been visited before:
\begin{equation}
\label{eq:Bn'}
B'_n:=\{\forall k=0,...,n-1\, :\, ({\mathcal I}_k,S_k)\ne (\mathcal I_{n},S_n)\} \subset\bar M \, .
\end{equation}
We set $B_n(a):=\bar O_a\cap B_n$ and $B'_n(a):=\bar T^{-n}(\bar O_a)\cap B_n$. We prove the following result on the probability of these sets.

\begin{prop}\label{prop:muB}
Assume Assumption \ref{hypo:perturb}.\\

If $\mathbf 1_{\bar O_a}\in \widetilde{\cB}_1$ and if $f\mapsto\mathbb E_{\bar\mu}[f\mathbf 1_{B_k(a)}]$ are uniformly bounded (uniformly in $k$) in $\widetilde{\cB}'_2$, then
\begin{equation}\label{eq:muB}
\bar\mu(B_n(a))= \frac{2\pi\sqrt{\det\Sigma^2}}{\log n}+O\left((\log n)^{-\frac 43}\right)\, .
\end{equation}

If $f\mapsto\mathbb E_{\bar\mu}[f\mathbf 1_{\bar O_a}]$ is in $\widetilde{\cB}_2'$ and if $P^{k}\mathbf 1_{B'_k(a)}$ are uniformly bounded (uniformly in $k$) in $\widetilde{\cB}_1$, then
\begin{equation}\label{eq:muB'}
\bar\mu( B'_n(a))= \frac{2\pi\sqrt{\det\Sigma^2}}{\log n}+O\left((\log n)^{-\frac 43}\right)\, .
\end{equation}
\end{prop}
\begin{proof}
As in \cite{soazRange}, we follow the idea of the proof of Dvoretzy and Erd\"os and adapt it to our context.
Considering the last visit time
to $\bar O_a\times\{\mathbf{0}\}$ of $(T^k(x,\mathbf{0}))$ until time $n$, we write
\begin{equation}\label{EQUA01}
\bar\mu(\bar O_a)=\sum_{k=0}^{n}\bar\mu\left(\bar O_a\cap\{S_k=\mathbf{0}\}\cap \bar T^{-k}(B_{n-k}(a))\right)\, 
\end{equation}
and, analogously,
\begin{equation}\label{EQUA02} 
\bar\mu(\bar O_a)=\bar\mu\left(\bar T^{-n}\bar O_a\right)=\sum_{k=0}^{n}\bar\mu\left((\bar T^{-n}\bar O_a)\cap\{S_n-S_{n-k}=\mathbf{0}\}\cap B'_{n-k}(a)\right)
\end{equation}
considering the first visit time to $\bar O_a\times\{S_n\}$ before time $n$.
Moreover, due to Corollary \ref{TLL00} and to our assumptions on
$\bar O_a$ and on $B_n(a)$, there exists $C">0$ such that
\begin{equation}\label{EQUA1}
\forall k\in\mathbb N^*,\quad\left\vert\bar\mu\left(\bar O_a\cap\{S_k=\mathbf{0}\}\cap \bar T^{-k}(B_{n-k}(a))\right)-\frac{\bar\mu(\bar O_a)\bar\mu(B_{n-k}(a))}{2k\pi\sqrt{\det \Sigma^2}}\right\vert\le \frac {C"}{k^{\frac 32}}\, ,
\end{equation}
and, since $\bar\mu\left((\bar T^{-n}\bar O_a)\cap\{S_n-S_{n-k}=\mathbf{0}\}\cap B'_{n-k}(a)\right)=\mathbb E_{\bar\mu}\left[\mathbf 1_{\bar O_a} P^k\left(\mathbf 1_{\{S_k=\mathbf{0}\}}P^{n-k}\mathbf 1_{B'_{n-k(a)}}\right)    \right]$.
Due to Theorem \ref{LLT0}, we also have
\begin{equation}\label{EQUA2}
\forall k\in\mathbb N^*,\quad\left\vert\bar\mu\left((\bar T^{-n}\bar O_a)\cap\{S_n-S_{n-k}=\mathbf{0}\}\cap B'_{n-k}(a)\right)-\frac{\bar\mu(\bar O_a)\bar\mu(B'_{n-k}(a))}{2k\pi\sqrt{\det \Sigma^2}}\right\vert\le \frac {C"}{k^{\frac 32}}\, .
\end{equation}
We will prove \eqref{eq:muB} using \eqref{EQUA01} and \eqref{EQUA1}.
The proof of \eqref{eq:muB'} using \eqref{EQUA02} and \eqref{EQUA2} follows the same scheme, 
and we omit it.
\begin{eqnarray*}
\bar\mu(\bar O_a)&\ge&\sum_{k=\lceil m_n\rceil}^{n-1} \frac{\bar\mu(\bar O_a)\bar\mu(B_{n-1-k}(a))}{2k\pi\sqrt{\det \Sigma^2}}+\sum_{k=m_n}^n\frac {C"}{k^{\frac 32}}\nonumber\\
&\ge&\bar\mu(B_{n}(a)) \left(\log (n)-\log(m_n)\right)
    \frac{\bar\mu(\bar O_a)}{2\pi\sqrt{\det \Sigma^2}}+\sum_{k=m_n}^n\frac {C"}{k^{\frac 32}}\\
&\ge& \log (n) \, \bar\mu(B_{n}(a))\left(1-\frac{\log(m_n)}{\log n}\right)
    \frac{\bar\mu(\bar O_a)}{2\pi\sqrt{\det \Sigma^2}}+O(m_n^{-\frac 12})\, ,
\end{eqnarray*}
with $m_n=\lceil(\log n)^2\rceil$, which leads to
\begin{equation}
\log (n) \, \bar\mu(B_n(a)) \le 2\pi\sqrt{\det\Sigma^2}+O\left(\frac{\log\log n}{\log n}\right).\label{BBB1}
\end{equation}

Moreover\\
\begin{eqnarray*}
\bar\mu(\bar O_a)&\le&\sum_{k=0}^{m'_n-1}\bar\mu(B_{\lceil n\log n\rceil-k}(a))+\sum_{k=m'_n}^{\lceil n\log n\rceil-n} \frac{\bar\mu(\bar O_a)\bar\mu(B_{\lceil n\log n\rceil-k}(a))}{2k\pi\sqrt{\det \Sigma^2}}\nonumber\\
&\ &\ \ \ \ \ \ \ \ +\sum_{k=\lceil n\log n\rceil-n+1}^{\lceil n\log n\rceil}\frac{\bar\mu(\bar O_a) \bar \mu(B_{\lceil n\log n\rceil}(a))}{2k\pi\sqrt{\det \Sigma^2} }+\sum_{k=m'_n}^{\lceil n\log n\rceil}\frac {C"}{k^{\frac 32}}\nonumber\\
&\le&  \frac{m'_n}{\log n}+ \bar\mu(B_{n}(a))\left(\left(\log ( n\log n-n+1)-\log(m'_n-1)\right)
    \frac{\bar\mu(\bar O_a)}{2\pi\sqrt{\det \Sigma^2}}\right)\nonumber\\
&\ &+\frac {\log( n\log n)-\log( n\log n-n)}{2\pi\sqrt{\det \Sigma^2} }+C" (m'_n)^{-\frac 12}\nonumber\, ,\\
\end{eqnarray*}
where we used the facts that $\bar \mu(B_{\lceil n\log n\rceil-k}(a))\le\bar\mu(B_n(a))=O((\log n)^{-1})$ 
for every $k\le\lceil n\log n\rceil-n$ and that
$\bar\mu(B_m(a))\le 1$ for $k>\lceil n\log n\rceil-n$. This leads us to
\begin{eqnarray*}
\bar\mu(\bar O_a)
&\le&\log n\frac{\bar\mu(\bar O_a)}{2\pi\sqrt{\det \Sigma^2}}\bar\mu(B_{n}(a))\left(1+O\left(\frac{\log\log n+\log m'_n}{\log n}\right)
    \right)
+O\left(\frac {m'_n}{\log n}+(m'_n)^{-\frac 12}\right)\nonumber\, ,\\
&\le&\log n\frac{\bar\mu(\bar O_a)}{2\pi\sqrt{\det \Sigma^2}}\bar\mu(B_{n}(a))\left(1+O\left(\frac{\log\log n}{\log n}\right)
    \right)
+O\left((\log n)^{-\frac 13}\right)\nonumber\, ,\\
\end{eqnarray*}
by taking $m'_n:=\lfloor (\log n)^{\frac 23}\rfloor$
and so
\begin{equation}
\log (n) \, \bar\mu(B_n(a)) \; \ge \; 2\pi\sqrt{\det\Sigma^2}+O\left((\log n)^{-\frac 13}\right).\label{BBB2}
\end{equation}
 
The proposition follows from \eqref{BBB1} and \eqref{BBB2}.
\end{proof}

\begin{lem}
\label{lem:bounded}
There exists $K_1>0$ such that, for every positive integer $\ell$, for every $(\omega_1,...,\omega_\ell) \in E^\ell$, 
for every uniformly bounded function $g:\bar M_0 \rightarrow \mathbb R$ which is uniformly $p$-H\"older continuous on connected components of
$\bar M_0\setminus\left(\cup_{k=1}^\ell \bar T_{\omega_1}^{-1} \circ \cdots \circ \bar T_{\omega_k}^{-1} (\cS_{0,H})\right)$, and
for all $f \in \cB_w$, 
\begin{equation}
\label{eq}
| \mathbb{E}_{\bar\mu_0} [ f \, g ] | \le K_1 |f|_w \left(| g|_\infty+  \sup_{C\in\mathcal C_{\omega_1,...,\omega_\ell}}C_{g_{|C}}^{(p)}\right)\ .
\end{equation}
Moreover, for every $f\in\mathcal B$,
\begin{equation}
\label{eq2}
\Vert    \mathcal L_{u,\omega_\ell}...\mathcal L_{u,\omega_1}( gf)\Vert_{\mathcal B} \le K_1 \Vert f\Vert_{\mathcal B}\left(| g |_\infty+  \sup_{C\in\mathcal C_{\omega_1,...,\omega_\ell}}C_{g_{|C}}^{(p)}\right)\ ,
\end{equation}
where $ \mathcal C_{\omega_1,...,\omega_\ell}$ is the set of connected components of $\bar M_0\setminus\left(\cup_{k=1}^\ell \bar T_{\omega_1}^{-1} \circ \cdots \circ \bar T_{\omega_k}^{-1} (\cS_{0,H})\right)$ and where $C_{g_{|C}}^{(p)}$ is the H\"older constant of $g$ restricted to $C$.
\end{lem}
The proof of Lemma \ref{lem:bounded} can be found in Appendix A.

\begin{rem}
The purpose of Lemma~\ref{lem:bounded} is to show that $K_1$ can be chosen independently of $\ell$.  
If one wishes similar bounds on piecewise H\"older continuous functions on $M_0$ with respect to a fixed partition, then Remark~\ref{Rke0} and
Lemma~\ref{lem:piecewise} provide such estimates under general conditions on the boundaries of partition elements.

Indeed, we will apply the lemma to the function $g = \mathbf{1}_{B_n(a)}$, where $B_n(a)$ is defined in Section~\ref{self} (see also Section~\ref{proofreturn}).
\end{rem}

Next we are  ready to prove the main Theorem \ref{theo:range}.
\begin{proof}[Proof of Theorem \ref{theo:range}]
Assumption \ref{hypo:perturb} follows from Theorem \ref{prop:perturb}.
The other assumptions of Proposition \ref{prop:muB} follow from Lemma \ref{lem:bounded}
since $\mathbf 1_{B_n(a)}$ satisfies the assumptions on $g$ in that lemma (uniformly
in $n$).
\end{proof}

\subsection{Number of self-intersections: Proof of Theorem \ref{theo:selfintersection}}\label{proofselfinter}
We consider the number of self-intersections $\mathcal V_n$ of the process $(\mathcal I_k,S_k)_k$ defined by
\begin{equation}
\label{eq:Vn}
\mathcal V_n:=\sum_{k,\ell=1}^{n}\mathbf 1_{\{S_\ell=S_k,\ \mathcal I_\ell=\mathcal I_k\}}.
\end{equation}
\begin{theo}\label{theo:SelfInter}
Assume Assumptions \ref{hypo:perturb} with $\widetilde\cB_2=\widetilde \cB_1$. Assume moreover:
\begin{itemize}
\item[(A5)] the operator $f\mapsto f\mathbf 1_{\bar O_a}$ is a linear operator on 
$\widetilde\cB_1$ for every $a \in \{ 1, \ldots, I \}$.
\end{itemize}
Then $(\mathcal V_n/(n\log n))_n$ converges $\bar\mu$-almost surely
to $\frac 1{\pi\sqrt{\det \Sigma^2}}\sum_{a=1}^I  
\bar\mu(\mathcal I_0=a)^2$.
\end{theo}
The proof of Theorem \ref{theo:selfintersection} 
will follow from the following lemmas.
Recalling \eqref{eq:Vn}, let us write $E_{k,\ell}:=\{S_k=S_\ell,\mathcal I_k=\mathcal I_\ell\}$
and $E_\ell:=E_{0,\ell}$.
\begin{lem}\label{expVn}
For $\ell > k$, we have
\[
\bar\mu\left(E_{k,\ell}\cap\bar T^{-k}\bar O_a\right)=\frac{\left(\bar \mu(\bar O_a)\right)^2}{2\pi\sqrt{\det \Sigma^2}(\ell-k)}+O((\ell-k)^{-\frac 32})\, ,
\]
and so
\[
\bar\mu(E_{k,\ell})=\frac{c_1}{\ell-k}+O((\ell-k)^{-\frac 32})\quad\mbox{and}\quad
\mathbb E_{\bar\mu}[\mathcal V_n]=2c_1 n\log n + O(n)\, ,
\]
with $c_1:=\frac 1{2\pi\sqrt{\det \Sigma^2}}\sum_{a=1}^I  
\bar\mu(\mathcal I_0=a)^2$.
\end{lem}
\begin{proof}
Since $\bar\mu$ is $\bar T$-invariant,
for $k< \ell$, recalling \eqref{EQOP1} we have
\begin{eqnarray*}
\bar\mu\left(E_{k,\ell}\cap\bar T^{-k}\bar O_a\right)&=&\bar\mu\left(E_{\ell-k}\cap\bar O_a\right)
= \bar\mu
(\mathcal I_0=a,\ S_{\ell-k}=0,\mathcal I_{\ell-k}=a)\\
&=&\mathbb E_{\bar\mu}
   \left[\mathbf 1_{\bar O_a}\mathcal H_{0,\ell-k} (\mathbf 1_{\bar O_a})\right]\\
&=&\frac{\bar\mu(\bar O_a)^2}{2\pi\sqrt{\det \Sigma^2}(\ell-k)}+
      O((\ell-k)^{-\frac 32})\, ,
\end{eqnarray*}
due to Theorem~\ref{LLT0} since $\mathbf 1_{\bar O_a}\in
\widetilde{\mathcal B}_1$ and since $\mathbb E_{\bar\mu}[\mathbf 1_{\bar O_a}\cdot]\in\widetilde{\mathcal B}'_1$.
Hence
$$\bar\mu(E_{k,\ell})=\sum_{a=1}^I\bar\mu\left(E_{k,\ell}\cap\bar T^{-k}\bar O_a\right)=\frac{\sum_{a=1}^I\left(\bar \mu(\bar O_a)\right)^2}{2\pi\sqrt{\det \Sigma^2}(\ell-k)}+O((\ell-k)^{-\frac 32})\, ,
$$
and
\begin{eqnarray*}
\mathbb E_{\bar\mu}[\mathcal V_n]&=&\sum_{k,\ell=1}^n\bar\mu(E_{k,\ell})=n+2\sum_{1\le k<\ell\le n}\bar\mu(E_{\ell-k})\\
&=&n+2\sum_{m=1}^{n-1}
    (n-m)\bar\mu(E_m)=O(n)+2c_1  n\log n\, .
\end{eqnarray*}
\end{proof}
\begin{lem}\label{PRO1}
There exists $C_1>0$ such that for all non-negative
integers $n,m,k$, for all $i,j,i',j'\in\{1,...,I\}$, and for all $N_1,N_2\in\mathbb Z^2$,
we have
$$|\cov_{\bar\mu}(\mathbf 1_{\{\mathcal I_0=i,S_n=N_1,\mathcal I_n=i'\}},
\mathbf 1_{\{\mathcal I_{n+m}=j,S_{n+m+k}-S_{n+m}=N_2,\mathcal I_{n+m+k}=i'\}})|\le \frac{C_1\alpha^m}{(n+1)(k+1)}.$$
In particular
$$|\cov_{\bar\mu}(\mathbf 1_{E_{0,n}},
\mathbf 1_{E_{n+m,n+m+k}})|\le \frac{I^2\, C_1\alpha^m}{(n+1)(k+1)}.$$

\end{lem}
\begin{proof}
The covariance we are interested in can be rewritten
$$\cov_{\bar\mu}\left(\mathbf 1_{\bar O_i}\mathbf 1_{\{S_n=N_1\}}\mathbf 1_{\bar O_{i'}}\circ\bar T^n, (\mathbf 1_{\bar O_j}\mathbf 1_{\{S_{k}=N_2\}}\mathbf 1_{\bar O_{j'}}\circ\bar T^k)\circ \bar T^{n+m}\right)\, $$
\begin{eqnarray*}
&=&\mathbb E_{\bar\mu}\left[P^{n+m+k}\left(\left(\mathbf 1_{\bar O_i}\mathbf 1_{\{S_n=N_1\}}\mathbf 1_{\bar O_{i'}}\circ\bar T^n-\mathbb E_{\bar\mu}[\mathbf 1_{\bar O_i}\mathbf 1_{\{S_n=N_1\}]}\mathbf 1_{\bar O_{i'}}\circ\bar T^n]\right) (\mathbf 1_{\bar O_j}\mathbf 1_{\{S_{k}=N_2\}}\mathbf 1_{\bar O_{j'}}\circ\bar T^k)\circ \bar T^{n+m}\right)\right]\, .
\end{eqnarray*}
Moreover, using several times $ P^m(f\, g\circ\bar T^m)=g\,  P^m(f)$
and the definition  of $\mathcal H_{\ell,n}$, we obtain that this quantity is equal to

\[
\mathbb E_{\bar\mu}\left[\mathbf 1_{\bar O_{j'}}\,  \mathcal H_{N_2,k}\Big(\mathbf 1_{\bar O_j} (P^{m}-\mathbb E_{\bar\mu})\left(\mathbf 1_{\bar O_{i'}}\mathcal H_{N_1,n}  \left(\mathbf 1_{\bar O_i}\right) \right)\Big) \right]\, 
\]
and so is bounded by
\[
\begin{split}
&a_{j'} \cdot \Vert \mathcal H_{N_2,k}\Vert_{L(\widetilde{\mathcal B}_1,\widetilde{\mathcal B}_1)} 
   \cdot  a_j\cdot
       \Vert P^{m}-\mathbb E_{\bar\mu}\Vert_{L(\widetilde{\mathcal B}_1,\widetilde{\mathcal B}_1)}\cdot
    a_{i'}\cdot
        \Vert \mathcal H_{N_1,n}\Vert_{L(\widetilde{\mathcal B}_1,\widetilde{\mathcal B}_1)}
               \Vert\mathbf 1_{\bar O_i}\Vert_{\widetilde{\mathcal B}_1}\\
& \le \frac{K_0^2}{(n+1)(k+1)}C\alpha^m
 a_{j'}a_{j}a_i               \Vert\mathbf 1_{\bar O_i}\Vert_{\widetilde{\mathcal B}_1}\, ,
\end{split}
\]
due to \eqref{intPun} and assumption (A5) of Theorem \ref{theo:SelfInter}, together with (A1) of Assumptions \ref{hypo:perturb} applied to $u=0$. Here $a_i:= \Vert \mathbf 1_{\bar O_{i}}\times\cdot \Vert_{L(\widetilde{\mathcal B}_1,\widetilde{\mathcal B}_1)} $.

This gives the first estimate of the lemma. To get the second one from the first one, we just observe that
$$\mathbf 1_{E_{k,\ell}}:=\sum_{i=1}^I\mathbf 1_{\bar O_i\cap\{S_{\ell-k}=0\}\cap\bar T^{-(\ell-k)}\bar O_i}\circ \bar T^k\, .$$
\end{proof}

We will use the notation $A_n \sim B_n$ for two positive quantities whenever
$\lim_{n \to \infty} \frac{A_n}{B_n} = 1$.

\begin{lem}
\label{lem:varVn}
We have $Var_{\bar\mu}(\mathcal V_n)\sim cn^2\, ,$
with 
$$c:=\frac{\left(\sum_{a=1}^I(\bar\mu(\bar O_a))^2\right)^{2}}{\det\Sigma^2}\left(\frac{1+2J}{\pi^2}-\frac 16\right),$$ $$J:=\int_{y_1,y_2,y_3>0:y_1+y_2+y_3<1} \frac {1-y_1-y_2-y_3}{y_1y_2+
y_2y_3+y_1y_3} \, dy_1dy_2dy_3.$$
\end{lem}
The proof of Lemma \ref{lem:varVn} is rather technical and involved, so we move it to the appendix B.
 
\begin{proof}[Proof of Theorem \ref{theo:SelfInter}]
Set $n_k:=\exp(\sqrt{k}\log k)$. For every $\varepsilon>0$,
due to the Bienaym\'e-Chebychev inequality and using Lemmas~\ref{expVn} and \ref{lem:varVn},
\[
\begin{split}
\sum_{k\ge 1}\bar\mu\left(|\cV_{n_k}-\mathbb E_{\bar\mu}[\cV_{n_k}]|>\varepsilon \mathbb E_{\bar\mu}[\cV_{n_k}]\right)
& \le \sum_{k\ge 1}\frac {Var_{\bar\mu}(\cV_{n_k})}
 {\varepsilon^2(\mathbb E_{\bar\mu}[\cV_{n_k}])^2} \\
 & = \sum_{k\ge 1} O((\log n_k)^{-2})=
 \sum_{k \ge 1} O(k^{-1}(\log k)^{-2}) \, < \infty \, .
\end{split} 
\]
Hence $(\cV_{n_k}/\mathbb E_{\bar\mu}[\cV_{n_k}])_k$
converges $\bar\mu$-almost surely to 1. Due to
Lemma \ref{expVn}, 
$(\cV_{n_k}/(n_k\log n_k))_k$ converges almost surely to $2c_1$.
Since $n_k\log n_k\sim n_{k+1}\log n_{k+1}$ and since $(\mathcal V_n)_n$ is increasing,  if $n\in\{n_k,...,n_{k+1}\}$, then 
$$\cV_{n_k}/(n_{k+1}\log n_{k+1})\le \mathcal \cV_n/(n\log n)\le \cV_{n_{k+1}}/(n_{k}\log n_{k}),$$ and so $(\mathcal \cV_n/(n\log n))_n$
converges $\bar\mu$-almost surely to $2c_1$.
\end{proof}

\begin{proof}[Proof of Theorem \ref{theo:selfintersection}]
Due to Remark~\ref{Rke1}, Theorem \ref{prop:perturb} and to Lemma \ref{lem:bounded},
the assumptions of Theorem \ref{theo:SelfInter} are satisfied.
Therefore $(\mathcal V_n/(n\log n))_n$ converges $\bar\mu$-almost surely
to 
\begin{equation}
\label{eq:square sum}
\frac 1{\pi\sqrt{\det \Sigma^2}}\sum_{a=1}^I  
\bar\mu(\mathcal I_0=a)^2=\frac 1{\pi\sqrt{\det \Sigma^2}}\sum_{a=1}^I  \left(\frac{2|\partial O_a|}{2\sum_{b=1}^I|\partial O_b|}\right)^2=
\frac 1{\pi\sqrt{\det \Sigma^2}} \frac{\sum_{a=1}^I |\partial O_a|^2}{\left(\sum_{b=1}^I|\partial O_b|\right)^2}\, .
\end{equation}
\end{proof}
\subsection{Random scenery: Proof of Theorem \ref{theo:randomscenery}}\label{proofRS}
Assume that to each atom $\bar O_i\times\{\ell\}$ is associated a random variable $\xi_{i,\ell}$,
independent and identically distributed across $i \in [1, \ldots I]$ and
$\ell \in \mathbb{Z}^2$, centered with variance $\sigma^2_\xi$ and
defined on a common probability space $(\Omega, \mathcal F, \mathbb{P})$.
We define the random variable (defined on $\bar M\times\Omega$):
$$\mathcal Z_n:=\sum_{k=0}^{n-1}\xi_{\mathcal I_k,S_k}\, .$$
We also define a linearly interpolated version of $\mathcal Z_n$ by,
\[
\widetilde{\mathcal Z}_n(t):=\mathcal Z_{\lfloor nt\rfloor}+(nt-\lfloor nt\rfloor)\xi_{(\mathcal I_{\lfloor nt\rfloor+1},S_{\lfloor nt\rfloor+1})}.
\]
\begin{theo}[Annealed and $\xi$-quenched CLT for $\mathcal Z$]\label{theo:pinball}
Assume Assumption \ref{hypo:perturb} and that, 
\begin{itemize}
  \item[i)] for every $a\in \{1,...,I\}$, $f\mapsto\mathbf 1_{\bar O_a}f$ is a continuous linear operator on $\widetilde{\mathcal B}_1$; 
  \item[ii)] and $\sup_{k\in\mathbb N}\Vert P^k\mathbf 1_{B'_k(a)}\Vert_{\widetilde {\mathcal B}_1}<\infty$
  (recalling \eqref{eq:Bn'}).
\end{itemize}
Then, $(\widetilde {\mathcal Z}_n:=(\widetilde{\mathcal Z}_{n}(t)/\sqrt{n\log n})_{t>0})_n$
converges in distribution, with respect to $\bar\mu\otimes\mathbb P$
(and to the uniform norm on $[0,T]$ for every $T>0$), to a Brownian motion 
$B = (B_t)_{t \ge 0}$ such that 
$\mathbb E[B_1^2]=\frac {\sigma^2_\xi}{\pi\sqrt{\det \Sigma^2}}\sum_{a=1}^I  \bar\mu(\mathcal I_0=a)^2$.

If, moreover, there exists $\chi>0$ such that $\mathbb E[|\xi_{(1,0)}|^2(\log^+|\xi_{(1,0)}|)^\chi|)]<\infty$, then, for $\mathbb P$-a.e.
realization of $(\xi_{i,\ell})_{i,\ell}$, 
$(\widetilde {\mathcal Z}_n)_n$ converges also in distribution, with respect to $\bar\mu$, to the same Brownian motion $B$.
\end{theo}

\begin{proof}[Proof of Theorem~\ref{theo:randomscenery}]
Using Theorem~\ref{theo:pinball}, we prove Theorem~\ref{theo:randomscenery}.  Assumptions 5.1
hold in the setting of Theorem~\ref{theo:randomscenery} due to Theorem~\ref{prop:perturb}.
Moreover, assumption (i) of Theorem~\ref{theo:pinball} follows from Remark~\ref{Rke1}, while
assumption (ii) follows from Lemma~\ref{lem:bounded}.  With the hypotheses of Theorem~\ref{theo:pinball} verified, Theorem~\ref{theo:randomscenery} follows using the same
calculation as in \eqref{eq:square sum}.
\end{proof}
We proceed to prove Theorem~\ref{theo:pinball}.

For the annealed central limit theorem, we mostly follow the proof by Bolthausen for random walks
in random scenery in dimension 2 \cite{Bolthausen}.
In comparison with \cite{soazPAPA}, the fact that the almost sure convergence of $\mathcal V_n$ has been proved greatly 
simplifies the proof.
\begin{lem}\label{Nnas}
Fix $\vartheta > 0$.
For $\bar\mu$-almost every $x\in\bar M$,
$\sup_{\ell}\sum_{k=1}^{n} \mathbf 1_{\{S_k=\ell\}}
=o(n^\vartheta)$.
\end{lem}
\begin{proof}
For every $\ell\in\mathbb Z^2$ and every $N\in\mathbb N^*$,
\begin{multline}\label{momentRange}
\mathbb E_{\bar\mu}\left[\left(\sum_{k=1}^{ n} \mathbf 1_{\{S_k=\ell\}}\right)^N \right]
\le N! \sum_{1 \le k_1\le \cdots\le k_N \le n} {\bar\mu}\left(S_{k_1}=S_{k_2}=\cdots=S_{k_N}=\ell\right)\\
=N! \sum_{1 \le  k_1\le \cdots\le k_N \le n} \mathbb E_{\bar\mu}\left[\mathcal H_{0,k_M-k_{N-1}}\cdots \mathcal H_{0,k_2-k_{1}}
\mathcal H_{\ell,k_{1}}(\mathbf 1)\right]\\
\le N! \sum_{1 \le  k_1\le \cdots\le k_N \le n}\frac{(K_0)^N}{(k_1+1)(k_2-k_1+1)\cdots(k_N-k_{N-1}+1)}=O(K_0^N N! \, (\log n)^N)\, ,
\end{multline}
due to Lemma \ref{LLT0}.
Moreover, due to Theorem~\ref{CLT}, there exists $c'>0$ such that $\mathbb E_{\bar\mu}[|S_n|^2]\sim c'n$ and so due to a result by Billingsley (see \cite{BillConvProbab} and \cite{Serfling70})
\[
\mathbb E_{\bar\mu}\left[\max_{k=1,...,n}|S_k|^2\right]=O(n(\log n)^2)\, 
\]
and so due to the Markov inequality, for every $s>0$,
$\bar\mu\left(\max_{k=1,...,n}|S_k|>n^{1+s}\right)\le\frac{\mathbb E_{\bar\mu}\left[\max_{k=1,...,n}|S_k|^2\right]}{n^{2+2s}}
=O(n^{-1-s}) \, .$
Now fix $\vartheta >0$. Then,
\begin{align*}
\bar\mu\left(\sup_\ell\sum_{k=1}^{ n}
 \mathbf 1_{\{S_k=\ell\}}>
n^{\vartheta}\right)&\le\bar\mu\left(\max_{k=1,...,n}|S_k|>n^{1+\vartheta}\right)+\bar\mu\left(\sup_{|\ell|\le n^{1+\vartheta}}\sum_{k=1}^{ n}
 \mathbf 1_{\{S_k=\ell\}}>
n^{\vartheta}\right)\\
&\le O(n^{-1-\vartheta})+ (2n^{1+\vartheta}+1)^2\sup_{|\ell| \le n^{1+\vartheta}}
\bar\mu\left( \sum_{k=1}^{ n}
 \mathbf 1_{\{S_k=\ell\}}>
n^{\vartheta}\right)\\
&\le O(n^{-1-\vartheta}+(\log n)^N n^{2+2\vartheta-\vartheta N})\, ,
\end{align*}
where we used the inequality $\mathbb{E}[X > n^\vartheta] \le \mathbb{E}[X^N] n^{-\vartheta N}$ for any $N \in \mathbb{N}^*$
combined with \eqref{momentRange}.  Now choosing $N > (3+3\vartheta)/\vartheta$,
we conclude the proof of the lemma by the Borel-Cantelli lemma.
\end{proof}
Recall that, for $x\in\bar M$, the random variable
$\mathcal Z_n(x)$ can be rewritten:
$
\mathcal Z_n( x)=\sum_{k=1}^n\xi_{\mathcal I_k,S_k}
 =\sum_{i=1}^I\sum_{\ell\in \mathbb Z^2}
   \xi_{i,\ell}\mathcal N_n(i,\ell)( x),
$
where $\mathcal N_n(i,\ell)( x):=
\sum_{k=1}^{n}\mathbf 1_{\{S_k=\ell,\mathcal I_k=i\}}(x)$
is the number of visits to the obstacle of index $(i,\ell)$
up to time $n$ and where $(\xi_{i, \ell})_{i, \ell}$ is a sequence of i.i.d.
centered square integrable random variables defined on some probability space 
$(\Omega,\mathcal F,\mathbb{P})$.

Note that the variance of $\mathcal Z_n(x)$
(with respect to $\mathbb{P}$) is $\sigma^2_\xi\mathcal V_n(x)$,
where $\sigma^2_\xi:=\mathbb E[\xi_{(1,\mathbf 0)}^2]$.

\begin{lem}[Convergence of finite-dimensional distributions]
\label{lem:fdd}
For every $m\ge 1$, every $0<t_1<t_2<...<t_m$,
$\bar\mu$-almost surely, 
$\left(\sum_{j=1}^ma_j\left(\mathcal Z_{\lfloor nt_j\rfloor}-\mathcal Z_{\lfloor nt_{j-1}\rfloor}\right)(x)/\sqrt{n\log n}\right)_n$ converges
in distribution (with respect to $\bar\mu\otimes \mathbb{P}$) to a centered Gaussian random variable
with variance $2c_1\sigma^2_\xi\sum_{j=1}^ma_j^2(t_j-t_{j-1})$.
\end{lem}

\begin{proof}
We fix $x\in\bar M$.
The variance of $\sum_{j=1}^ma_j\left(\mathcal Z_{\lfloor nt_j\rfloor}-\mathcal Z_{\lfloor nt_{j-1}\rfloor}\right)(x)$ 
(with respect to $\mathbb{P}$) is equal to,
recalling \eqref{eq:Vn},
\begin{equation}
\begin{split}
\label{equivvarcond}
\sigma^2_\xi   \sum_{i=1}^I\sum_{\ell\in\mathbb Z^2}
& \left(
   \sum_{j=1}^m a_j \left( \mathcal N_{\lfloor nt_j\rfloor}(i,\ell)(x)-
     \mathcal N_{\lfloor nt_{j-1}\rfloor}(i,\ell)( x)\right)\right)^2 \\
& =\sigma^2_\xi
   \sum_{i=1}^I\sum_{\ell\in\mathbb Z^2}
   \sum_{j,j'=1}^ma_ja_{j'}
\sum_{k=\lfloor nt_{j-1}\rfloor+1}^{\lfloor nt_{j}\rfloor}
\sum_{k'=\lfloor nt_{j'-1}\rfloor+1}^{\lfloor nt_{j'}\rfloor}
\mathbf 1_{\{S_k=\ell,\mathcal I_k=i,S_{k'}=\ell,\mathcal I_{k'}=i\}}( x)\\
& =\sigma^2_\xi   \sum_{j,j'=1}^ma_ja_{j'}
\sum_{k=\lfloor nt_{j-1}\rfloor+1}^{\lfloor nt_{j}\rfloor}
\sum_{k'=\lfloor nt_{j'-1}\rfloor+1}^{\lfloor nt_{j'}\rfloor}
\mathbf 1_{\{S_k=S_{k'},\mathcal I_k=\mathcal I_{k'}\}}( x)\\
& =\sigma^2_\xi
\left(\sum_{j=1}^m a_j^2 \cV_{\lfloor nt_{j}\rfloor-\lfloor nt_{j-1}\rfloor}\circ   
\bar T^{\lfloor nt_{j-1}\rfloor}\right.\\
& \hspace{-.75 in} \left.   +   \sum_{1\le j<j'\le m}   a_ja_{j'}
\left(\left(\cV_{\lfloor nt_{j'}\rfloor-\lfloor nt_{j-1}\rfloor}-
\cV_{\lfloor nt_{j'-1}\rfloor-\lfloor nt_{j-1}\rfloor}\right)
       \circ   \bar T^{\lfloor nt_{j-1}\rfloor}+\left(
\cV_{\lfloor nt_{j'-1}\rfloor-\lfloor nt_{j}\rfloor}-
    \cV_{\lfloor nt_{j'}\rfloor-\lfloor nt_{j}\rfloor}\right)
\circ   \bar T^{\lfloor nt_{j}\rfloor}
\right)\right)\\
& \sim 2c_1 \sigma^2_\xi\sum_{j=1}^ma_j^2
(t_j-t_{j-1})n\log n\, ,
\end{split}
\end{equation}
for $\bar\mu$-a.e. $x\in\bar M$, 
due to the proof of Theorem \ref{theo:selfintersection}
(since $(\mathcal V_n/(n\log n))_n$ $\bar\mu$-converges almost surely to
$2c_1$, as well as any sequence of random variables
with the same marginal distributions).

Note that, with respect to $\mathbb{P}$, $\sum_{j=1}^ma_j\left(\mathcal Z_{\lfloor nt_j\rfloor}-\mathcal Z_{\lfloor nt_{j-1}\rfloor}\right)(x)$ is a sum of independent 
centered random variables with variances
$$\left(\sigma^2_{n,i,\ell}(x):=\sigma^2_\xi
\left(\sum_{j=1}^ma_j(\mathcal N_{\lfloor nt_j\rfloor}(i,\ell)(x)-
     \mathcal N_{\lfloor nt_{j-1}\rfloor}(i,\ell)( x))\right)^2\right)_{i,\ell}\, . $$
Hence, due to Lemma \ref{Nnas} and to the Lindeberg
Theorem, for $\bar\mu$-almost every
$x\in\bar M$, the sequence of random variables
\[
\left(\frac{\sum_{j=1}^ma_j\left(\mathcal Z_{\lfloor nt_j\rfloor}-\mathcal Z_{\lfloor nt_{j-1}\rfloor}\right)(x)}
  {Var(\sum_{j=1}^ma_j(\mathcal Z_{\lfloor nt_j\rfloor}-\mathcal Z_{\lfloor nt_{j-1}\rfloor})(x))}
\right)_n
\]
converges in distribution (with respect to $\mathbb{P}$) to a standard Gaussian random variable.
The conclusion then follows from \eqref{equivvarcond}.
\end{proof}
\begin{lem}\label{lem:tight}
The sequence of random variables $\left(\widetilde{\mathcal Z}_n(t)
/\sqrt{n\log n}\right)_n$ is tight (with respect to $\bar\mu\otimes \mathbb{P}$) in $C([0,T])$ for every $T>0$.
\end{lem}
\begin{proof}
Due to Theorem \ref{theo:selfintersection}, 
it is enough to prove the tightness of 
$\left(\widetilde{\mathcal Z}_n(t)/\sqrt{\sigma^2_\xi \mathcal V_n}\right)_n$.
Due to \cite[Lemma p. 88]{BillConvProbab}, 
it is enough to prove that
\begin{equation}\label{tight}
\lim_{\lambda\rightarrow +\infty}\limsup_{n\rightarrow +\infty}
   \lambda^2(\bar\mu\otimes \mathbb{P})
       \left(\max_{k=1,...,n}|\mathcal Z_k|\ge \lambda \sigma_\xi\sqrt{ \mathcal V_n}\right)=0.
\end{equation}
We modify the proof of tightness of Bolthausen in \cite{Bolthausen}. 
For completeness, we explain the adaptations to make.
Following \cite{Bolthausen} (see also \cite[bottom of page 824]{soazPAPA},
using the fact that $(\mathcal Z_n)_n$ has positively associated increments knowing $(S_n)_n$, 
we obtain that, for any $\lambda>\sqrt{2}$, 
$$(\bar\mu\otimes \mathbb{P})\left(\max_{j\le n}|\mathcal Z_j|\ge \lambda\sigma_\xi \sqrt{\mathcal V_n}\right) \le 
2(\bar\mu\otimes \mathbb{P})\left(|\mathcal Z_n|>(\lambda-\sqrt{2})\sigma_\xi\sqrt{\mathcal V_n}\right)\, .$$
Now we simplify the conclusion of \cite{Bolthausen}. Since we know that $(\mathcal Z_n/\sqrt{\mathcal V_n})_n$ converges in distribution to a Gaussian random variable $Y$, so 
\begin{eqnarray*}
\limsup_{n\rightarrow +\infty}(\bar\mu\otimes \mathbb{P})\left(\max_{j\le n}|\mathcal Z_j|\ge \lambda \sigma_\xi\sqrt{\mathcal V_n}\right)
      \le  2\mathbb P\left(|Y|>(\lambda-\sqrt{2})\sigma_\xi\right).
\end{eqnarray*}
and $\mathbb P\left(|Y|>x\right)=O(e^{-c_Yx^2})$ for some $c_Y>0$,
which proves \eqref{tight} and so the tightness.
\end{proof}

\begin{proof}[Proof of Theorem~\ref{theo:pinball}]
The first result of Theorem \ref{theo:pinball} is a direct consequence of
Lemmas \ref{lem:fdd} and \ref{lem:tight}.

Now let us prove the last point. For this, we use the general argument
developed by Guillotin-Plantard, Dos Santos and Poisat in \cite{NJR}.
Indeed the proof of \cite{NJR} only uses the following assumptions:
\begin{itemize}
\item $\Gamma$ is a denumerable set,
\item $\widetilde S:=(\widetilde S_n)_{n\ge 0}$ is a sequence of $\Gamma$-valued random variables,
\item $\xi:=(\xi)_{y\in \Gamma}$ is a sequence of independent identically distributed real valued
random variables, which are centered and such that $\mathbb E[|\xi_{y}|^2(\log^+|\xi_{y}|)^\chi|)]<\infty$
for some $\chi >0 $,
\item the sequences of random variables $\xi$ and $\widetilde S$ are independent,
\item $\left(\frac 1{\sqrt{n\log n}}(\sum_{k=0}^{\lfloor nt\rfloor-1}\xi_{\widetilde S_k}+(nt-\lfloor nt\rfloor)\xi_{\widetilde S_{\lfloor nt\rfloor}})\right)_{t\in[0,1]}$ converges in distribution
in $C(0,T)$ to the Brownian motion $B$,
\item $\sup_{y\in \Gamma}\mathbb E[\widetilde N_n(y)]=O(\log n)$ with
$\widetilde N_n(y):=\#\{k=0,...,n-1\, :\, \widetilde S_k=y\}=\sum_{k=0}^{n-1}\mathbf 1_{\{\widetilde S_k=y\}}$ being the local time of $\widetilde S$.
\item $\sum_{y\in \Gamma}(\mathbb E[(\widetilde N_n(y))])^2=O(n)$, with the same notation.
\item $\mathbb P(\widetilde S_n\not\in\{\widetilde S_0,...,\widetilde S_{n-1}\})=O((\log n)^{-1})$.
\end{itemize}
We apply this to $\Gamma=\{1,...,I\}\times\mathbb Z^2$ and $\widetilde S_n=(\mathcal I_n,S_n)$.
For the antepenultimate condition, observe that, due to Corollary \ref{TLL00},
$$ \mathbb E[\widetilde N_n(a,\ell)]=\sum_{k=0}^{n-1}\mathbb E_{\bar\mu}\left[\mathbf 1_{\{S_k=\ell\}}.\mathbf 1_{\bar O_a}\circ\bar T^{n}\right]=\sum_{k=0}^{n-1}\mathbb E_{\bar\mu}\left[\mathbf 1_{\bar O_a}\mathcal H_{\ell,k}(\mathbf 1)\right]=O(\log n)\, .$$
For the penultimate condition, 
$$\sum_{y\in \Gamma}(\mathbb E[(\widetilde N_n(y))])^2=\sum_{y\in \Gamma}\sum_{k,j=0}^{n-1}
\mathbb E_{\bar\mu\times\bar\mu}[\mathbf 1_{\widetilde S_k=\widetilde S'_j=y}]=\sum_{i,j=0}^{n-1}\mathbb E_{\bar\mu\times\bar\mu}[\mathbf 1_{\widetilde S_k=\widetilde S'_j}],$$
considering an independent copy $\widetilde S'=(\widetilde S'_n=(\mathcal I'_n,S'_n))_n$ of $\widetilde S$. 
Now, using again \eqref{EQOP2} combined with Assumption \ref{hypo:perturb} with $\beta$ and $a>0$ as in the proof of Theorem \ref{LLT0},
we obtain
\begin{eqnarray*}
\mathbb E_{\bar\mu\times\bar\mu}\left[\mathbf 1_{\widetilde S_k=\widetilde S'_j}\right]
&\le&\mathbb E_{\bar\mu\times\bar\mu}\left[\mathbf 1_{ S_k=S'_j}\right]\\
&=&\int_{[-\pi,\pi]^2}\mathbb E_{\bar\mu\times\bar\mu}\left[e^{i u\cdot  S_k}e^{-iu\cdot  S'_j}\right]\, du\\
&=&\int_{[-\pi,\pi]^2}\mathbb E_{\bar\mu}\left[e^{i u\cdot  S_i}\right]\mathbb E_{\bar\mu}\left[e^{-iu\cdot  S'_j}\right]\, du\\
&=&\int_{[-\pi,\pi]^2}\mathbb E_{\bar\mu}\left[P_u^k\mathbf 1\right]\mathbb E_{\bar\mu}\left[P_{-u}^j\mathbf 1\right]\, du\\
&\le&\int_{[-\beta,\beta]^2}e^{-ak|u|^2}\left|\mathbb E_{\bar\mu}\left[\Pi_u\mathbf 1\right]\right|\, e^{-aj|u|^2}\left|\mathbb E_{\bar\mu}\left[\Pi_u\mathbf 1\right]\right|\, du+O(\alpha^{k+j})\\
&\le&\int_{\mathbb R^2}e^{-ak|u|^2}\left|\mathbb E_{\bar\mu}\left[\Pi_u\mathbf 1\right]\right|\, e^{-aj|u|^2}\left|\mathbb E_{\bar\mu}\left[\Pi_u\mathbf 1\right]\right|\, du+O(\alpha^{k+j})\\
&=&O(|1+k+j|^{-1})\, .
\end{eqnarray*}
Therefore
$$\sum_{y\in \Gamma}(\mathbb E[(\widetilde N_n(y))])^2=O\left(\sum_{0\le j,k\le n-1}\frac 1{1+k+j}\right)\, =O(n)\, .$$

The last condition comes from the second part of Proposition~\ref{prop:muB}.  
Note that in order to invoke Proposition~\ref{prop:muB}, we need that
the operator $f \mapsto \mathbb{E}_{\bar \mu}[f \, \mathbf{1}_{\bar O_a}]$ is continuous
on $\widetilde B_2'$.  This follows from the fact that we have assumed (i) in the statement
of the theorem, that
$f \mapsto f \mathbf{1}_{\bar O_a}$ is a continuous operator on $\widetilde \cB_1$, and
that by Assumption~\ref{hypo:perturb}, $\mathbb{E}_{\bar \mu}[ \cdot ]$ acts continuously on 
$\widetilde \cB_2$.  The second condition needed to conclude \eqref{eq:muB'} from 
Proposition~\ref{prop:muB} is precisely assumption (ii) in the statement of the theorem.
\end{proof}
\begin{appendix}

\section{Proof of Lemma \ref{lem:bounded}}
Here we prove the Lemma \ref{lem:bounded}, which was used in Subsection \ref{proofreturn}, especially used in the proof of Theorem \ref{theo:range}].

Let us prove that \eqref{eq} holds true. 
By density, it suffices to perform the estimate for $f \in \cC^1(\bar M_0)$.
In the proof below, we use the fact that the invariant measure $\bar\mu_0$ is absolutely continuous with respect to the Lebesgure measure.

Choose $\ell \ge 1$ and fix $\underline\omega_\ell := (\omega_1, \ldots, \omega_\ell)$.  Let $g$ be as in the statement of the lemma.
For brevity, denote by
$\bar T_{\underline{\omega}_\ell}^\ell = \bar T_{\omega_\ell} \circ \cdots \circ \bar T_{\omega_1}$ the composition of
random maps and  by $\cL_{\underline{\omega}_\ell}^\ell$ its associated transfer operator.
Also, set $H_\ell^p(g) = |g|_\infty + \sup_{C\in\mathcal C_{\omega_1,...,\omega_\ell}}C_{g_{|C}}^{(p)}$.
We must estimate
$$
\mathbb{E}_{\bar\mu_0} [f \, g ] = \int_{\bar M_0} f \, g \, d\bar\mu_0 = 
\int_{\bar M_0} \cL_{\underline\omega_\ell}^\ell f \cdot g \circ (\bar T_{\underline\omega_\ell}^{\ell})^{-1} \, d\bar\mu_0 .
$$
To do this, we decompose $\bar M_0$ into a countable collection of local rectangles, each
foliated by a smooth collection of stable curves on which we may apply our norms.  This
technique follows closely the decomposition used in \cite[Lemma~3.4]{MarkHongKun2013}.

We partition each connected component of $\bar M_0 \setminus (\cup_{|k| \ge k_0} \mathbb{H}_k)$, 
into finitely many boxes $B_j$ whose boundary curves are elements of $\cW^s$ and 
$\cW^u$, as well as the horizontal boundaries of $\mathbb{H}_{\pm k_0}$.  We construct the
boxes $B_j$ so that each has diameter in $(\delta/2, \delta)$, for some $\delta>0$,
and is foliated by a smooth foliation of
stable curves
$\{ W_\xi \}_{\xi \in \Xi_j}$, such that each curve $W_{\xi}$ is stretched completely between the two unstable boundaries of $B_j$. 
Indeed, due to the continuity of the cones $C^s(x)$ from {\bf (H1)}, we can choose $\delta$ sufficiently
small that the family $\{ W_\xi \}_{\xi \in \Xi_j}$ is a family of parallel line segments.

We disintegrate the  measure $\bar\mu_0$  on $B_j$ into
a family of  conditional probability measures $d\mu_{\xi} = c_\xi \cos \vf \, dm_{W_\xi}$, $\xi \in \Xi_j$, 
where $c_\xi$ is a normalizing constant, and a factor measure 
$ \lambda_j(\xi)$ on the index set $\Xi_j$.  Since $\bar\mu_0$ is absolutely continuous with respect to Lebesgue measure
on $\bar M_0$, 
we have $ \lambda_j(\Xi_j) = \bar \mu_0(B_j) = \mathcal{O}(\delta^2)$.  

Similarly, on each homogeneity strip $\mathbb{H}_t$, $t \ge k_0$, we choose a smooth foliation of parallel line segments 
$\{ W_\xi \}_{\xi \in \Xi_t} \subset \mathbb{H}_t$ which completely cross $\mathbb{H}_t$.  Due to the
uniform transversality of the stable cone with $\partial \mathbb{H}_t$, we may choose a single index set $\Xi_t$ for each homogeneity strip.  We again disintegrate $\bar\mu_0$
into a family of 
conditional probability measures $d\mu_\xi = c_\xi \cos \vf \, dm_{W_\xi}$, $\xi \in \Xi_t$, and a transverse
measure $\lambda_t(\xi)$ on the index set $\Xi_t$.  
This implies that $\lambda_t(\Xi_t) = \bar\mu_0(\mathbb{H}_t) = \mathcal{O}(|t|^{-5})$ for each $|t| \ge k_0$.

Notice that on each homogeneity strip $\mathbb{H}_k$, the function $\cos \vf$ satisfies,
\begin{equation}
\label{eq:jac dist}
|\log \cos \vf(x)- \log \cos \vf(y)| \leq C d(x,y)^{1/3}
\end{equation}
for some uniform constant $C>0$ (uniform in $k$).

We are ready to estimate the required integral.  Let
$\cG_\ell(W_\xi)$ denote the components of $(\bar T_{\underline{\omega}_\ell}^{\ell})^{-1} W_{\xi}$ , with long
pieces subdivided
to have length between $\delta_0/2$ and $\delta_0$, as in the proof of Lemma~\ref{lem:uniform ly}.
\[
\begin{split}
\int & \cL_{ \underline{\omega}_\ell}^\ell f \cdot g \circ (\bar T_{\underline{\omega}_\ell}^{\ell})^{-1} \, d\bar\mu_0
 =\sum_j \int_{B_j} \cL_{\underline{\omega}_\ell}^\ell f \cdot g \circ (\bar T_{\underline{\omega}_\ell}^{\ell})^{-1} d\bar\mu_0  
 +  \sum_{|t| \ge k_0} \int_{\mathbb{H}_t} \cL_{\underline{\omega}_\ell}^\ell f \cdot g \circ (\bar T_{\underline{\omega}_\ell}^{\ell})^{-1} d\bar\mu_0\\
 &=\sum_j \!\! \int_{\Xi_j} \! \int_{W_\xi} \! \! \cL_{\underline{\omega}_\ell}^\ell f \cdot g \circ (\bar T_{\underline{\omega}_\ell}^{\ell})^{-1} \,
d\mu_\xi d\lambda_j(\xi)  
 +  \sum_{|t| \ge k_0} \! \! \int_{\Xi_t} \! \int_{W_\xi} \! \! \cL_{\underline{\omega}_\ell}^\ell f \cdot g \circ (\bar T_{\underline{\omega}_\ell}^{\ell})^{-1} \, d\mu_\xi d\lambda_t(\xi) \\ 
 & =\sum_j \int_{\Xi_j}  \sum_{W_{\xi,i} \in \cG_\ell(W_\xi)} \int_{W_{\xi,i}}  f \, g \,  c_\xi \cos \vf \circ  \bar T_{\underline\omega_\ell}^\ell  \, 
 J_{W_{\xi,i}}\bar T_{\underline\omega_\ell}^\ell \, dm_{W_{\xi,i}} d\lambda_j(\xi) \\
 & \qquad +
 \sum_{|t| \ge k_0} \int_{\Xi_t}  \sum_{W_{\xi,i} \in \cG_\ell(W_\xi)} \int_{W_{\xi,i}}  f \, g \, c_\xi \cos \vf \circ  \bar T_{\underline\omega_\ell}^\ell  \, J_{W_{\xi,i}}\bar T_{\underline\omega_\ell}^\ell \, dm_{W_{\xi,i}} d\lambda_t(\xi) \, .
\end{split}
\]

Next we use the assumption that $g$ is H\"older continuous on connected componts of
$\bar M_0 \setminus (\cup_{k=1}^\ell \bar T_{\omega_1}^{-1} \circ \cdots \circ \bar T_{\omega_k}^{-1} (\cS_{0,H}))$.  Since
elements of $\cG_\ell(W_\xi)$ are also subdivided according to these singularity sets, we have
that $g$ is H\"older continuous on each $W_{\xi, i} \in \cG_\ell(W_\xi)$.
Thus,
\[
\begin{split}
   \int_{W_{\xi, i}}  f  \, g \, c_\xi \cos \vf \circ \bar T_{\underline\omega_\ell}^\ell  \, J_{W_{\xi,i}}\bar T_{\underline\omega_\ell}^\ell \, dm_{W_{\xi,i}}
 &\le |f|_w  |g|_{\cC^p(W_{\xi,i})} c_\xi |\cos \vf \circ \bar T_{\underline\omega_\ell}^\ell | _{\cC^p(W_{\xi,i})}   |J_{W_{\xi,i}}\bar T_{\underline\omega_\ell}^\ell |_{\cC^p(W_{\xi,i})} \\
 & \le  |f|_w H_\ell^p(g) |J_{W_{\xi,i}}\bar T_{\underline\omega_\ell}^\ell |_{\cC^0(W_{\xi,i})}  \frac{C}{|W_{\xi}|},
 \end{split}
\]
where we used (\ref{eq:jac dist}) in the last estimate, as well as the fact that the normalizing constant $c_\xi$ is
proportional to $|W_\xi|^{-1}$.
This implies that
\begin{align*}
\mathbb{E}_{\bar\mu_0} [ f \, g ] & 
\le C |f|_w H_\ell^p(g) \Big( \sum_j \int_{\Xi_j}  \sum_{W_{\xi,i} \in \cG_\ell(W_\xi)}   |J_{W_{\xi,i}}\bar T_{\underline\omega_\ell}^\ell |_{\cC^0(W_{\xi,i})}   |W_\xi|^{-1} \, d\lambda_j(\xi) \\
 & \qquad +
 \sum_{|t| \ge k_0} \int_{\Xi_t}  \sum_{W_{\xi,i} \in \cG_\ell(W_\xi)}  |J_{W_{\xi,i}}\bar T_{\underline\omega_\ell}^\ell |_{\cC^0(W_{\xi,i})}   |W_\xi|^{-1} \, d\lambda_t(\xi) \Big)
\end{align*}
Now $\sum_{W_{\xi,i} \in \cG_\ell(W_\xi)}  |J_{W_{\xi,i}}\bar T_{\underline\omega_\ell}^\ell |_{\cC^0(W_{\xi,i})}$ is bounded by a uniform
constant independent of $\xi$ and $\underline\omega_\ell$ by \cite[Lemma~5.5(b)]{MarkHongKun2013}.
Moreover, $\int_{\Xi_j} |W_{\xi}|^{-1} d \lambda_j(\xi) \leq C\delta_0$ for some constant $C>0$ since we chose our foliation to be comprised of long cone-stable curves.
We conclude that the first term to the right hand side of  the last inequality  is uniformly bounded by $C_1|f|_w H_\ell^p(g)$ since
the sum over $j$ is finite.

For the second term on the right hand side of the last equality, we again use \cite[Lemma~5.5(b)]{MarkHongKun2013}
as well as the fact that $|W_{\xi}|^{-1} = \mathcal{O}(t^3)$ for
$\xi \in \Xi_t$, while $\lambda_t(\Xi_t) = \mathcal{O}(t^{-5})$. Thus
\[ 
\sum_{|t| \ge k_0} \int_{\Xi_t}|W_{\xi}|^{-1}d \lambda_t(\xi)\le 
\sum_{|t| \ge k_0} C t^{-2} \leq C k_0^{-1} .
\]
We conclude that
\[
\left| \mathbb{E}_{\bar \mu_0} [f \, g] \right| \le K_1 |f|_w H^p_\ell(g) ,
\] 
for some uniform constant $K_1$ depending on $\bar {\mathcal{F}}_{\vartheta_0}$, but not on 
$f$, $\ell$ or $\underline{\omega}_\ell$. This completes the proof of \eqref{eq}.

To prove \eqref{eq2}, we follow the proof of 
Lemma~\ref{lem:uniform ly}.  Note that for 
$f \in \cC^1(\bar M_0)$, $W \in \cW^s$, and a test function
$\psi$, we have
\[
\int_W \cL_{u, \omega_\ell} \ldots \cL_{u, \omega_1}(fg) \, \psi \, dm_W
= \sum_{W_i} \int_{W_i} fg\, e^{iu \cdot S_\ell} \, \psi \circ \bar T^\ell_{\underline \omega_\ell} \,
J_{W_i}\bar T^\ell_{\underline \omega_\ell} \, dm_{W_i} \, ,
\] 
where the sum is taken over $W_i \in \cG_\ell(W)$, the components of $(\bar T^\ell_{\underline \omega_\ell})^{-1}W$, subdivided as before.  This is the same type of expression as in 
\cite[eq. (5.24)]{MarkHongKun2013} or \cite[eq. (4.4)]{MarkHongKun2013}, but now the
test function is
\[
g\, e^{iu \cdot S_\ell} \, \psi \circ \bar T^\ell_{\underline \omega_\ell} \,
J_{W_i}\bar T^\ell_{\underline \omega_\ell}
\]
rather than simply 
$\psi \circ \bar T^\ell_{\underline \omega_\ell} \, J_{W_i}\bar T^\ell_{\underline \omega_\ell}$.
Since $S_\ell$ is constant on each $W_i \in \cG_\ell(W)$, and we have assumed that 
$g$ is (uniformly in $\ell$) H\"older continuous on each $W_i \in \cG_\ell(W)$,
the proof of the Lasota-Yorke inequalities follows as in the proof of 
\cite[Proposition~5.6]{MarkHongKun2013}.  The bound \eqref{eq2} then follows as in
the proof of Lemma~\ref{lem:uniform ly}.

\begin{rem}
\label{rem:dual}
As a consequence of this lemma, if 
$g:\bar M \rightarrow\mathbb R$ is a bounded measurable function such that, for every $\uomega=(\omega_k)_{k\ge 0}\in E^{\mathbb N}$, there exists positive integer 
$\ell_{\uomega}$ such that 
$g(\cdot,\uomega)$ is $p$-H\"older on every connected component (uniformly on $\uomega$) of
$\bar M_0\setminus\left(\cup_{k=0}^{\ell_{\uomega}-1} \bar T_{\omega_0}^{-1} \circ \cdots \circ \bar T_{\omega_{\ell(\uomega)-1}}^{-1} (\cS_{0,H})\right)$. Then, for every 
$f\in\widetilde{\mathcal B}_w$, we have
\begin{eqnarray*}
\left|\mathbb E_{\bar\mu}[gf]\right|&=&
   \left|\int_{E}\mathbb E_{\bar\mu_0}[g(\cdot,\uomega) f(x,\uomega)]\, d\eta(\uomega)\right|\\
&=&K_1\Vert f\Vert_{\widetilde\cB_w} \left(\Vert g\Vert_\infty+\sup_{\uomega\in E^{\mathbb N}}\sup_{C\in\mathcal C_{\omega_1,...,\omega_\ell(\omega)}}C_{(g(\cdot,\uomega))_{|C}}^{(p)}\right)\, ,
\end{eqnarray*}
with the same notations as in the previous lemma.
Therefore, $\mathbb{E}_{\bar\mu} [ g \cdot] $ is in $\widetilde{\mathcal \cB}_w'$.
\end{rem}

\section{Proof of Lemma \ref{lem:varVn}.}

Note that $\mathcal V_n=n+2\sum_{1\le k<\ell\le n}\mathbf 1_{\{S_\ell=S_k,\mathcal I_\ell=\mathcal I_k\}}$. Hence
\[
Var_{\bar\mu}(\mathcal V_n)=4\sum_{1\le k_1<\ell_1\le n}\sum_{1\le k_2<\ell_2\le n}
    D_{k_1,\ell_1,k_2,\ell_2},
\]
with $D_{k_1,\ell_1,k_2,\ell_2}:=\bar\mu(E_{k_1,\ell_1}\cap
    E_{k_2,\ell_2})-\bar\mu(E_{k_1,\ell_1})\bar\mu(E_{k_2,\ell_2})$.
It follows that 
\begin{equation}
\label{eq:Var}
\left|Var_{\bar\mu}(\mathcal V_n)-8(A_2+A_3)\right|\le 8(A_1+A_4),
\end{equation}
with
\[
A_1:=\sum_{1\le k_1<\ell_1\le k_2<\ell_2\le n}
    \left|D_{k_1,\ell_1,k_2,\ell_2}\right|,\quad
A_2:=\sum_{1\le k_1< k_2< \ell_1<\ell_2\le n}
    D_{k_1,\ell_1,k_2,\ell_2}\, ,
\]
\[
A_3:=\sum_{1< k_1< k_2<\ell_2<\ell_1\le n}
    D_{k_1,\ell_1,k_2,\ell_2},\quad
A_4:=\sum_{(k_1,k_2,\ell_1,\ell_2)\in E_n\cup F_n}
    \left| D_{k_1,\ell_1,k_2,\ell_2}\right|\, ,
\]
with $$E_n:=\{(k_1,k_2,\ell_1,\ell_2)\in\{1,...,n\}\ :\
    k_1=k_2<\min(\ell_1,\ell_2)\},$$$$
 F_n:=\{(k_1,k_2,\ell_1,\ell_2)\in\{1,...,n\}\ :\
    \max(k_1, k_2)<\ell_1=\ell_2\}.$$\
We will start with the two easiest estimates: the estimates of the error terms $A_1$ and $A_4$.
The method we will use to estimate the main terms $A_2$ and $A_3$ differs from \cite{soazSelfInter}.

 Due to Lemma \ref{PRO1},
\[
A_1\le\, I^2\sum_{1\le k_1<\ell_1\le k_2<\ell_2\le n}\frac{C_1\alpha^{k_2-\ell_1}}{(\ell_1-k_1)(\ell_2-k_2)}
=O(n(\log n)^2)=o(n^2).
\]
 Let us now prove that $A_4=o(n^2)$ by writing
\begin{eqnarray*}
\sum_{(k_1,k_2,\ell_1,\ell_2)\in E_n}
 \left|D_{k_1,\ell_1,k_2,\ell_2}\right|
&\le&2\sum_{1\le k<\ell_1\le\ell_2\le n}
\left(\bar\mu(E_{k,\ell_1}\cap
    E_{k,\ell_2})+\bar\mu(E_{k,\ell_1})\bar\mu(E_{k,\ell_2})\right)\\
&\le&2\sum_{1\le k<\ell_1\le\ell_2\le n}
\left(\bar\mu(S_{\ell_1}=S_{\ell_2}=S_k)
     +\bar\mu(S_{\ell_1}=S_k)\bar\mu(S_{\ell_2}=S_k)\right)\\
&\le&2\sum_{1\le k<\ell_1\le\ell_2\le n}
\left(\mathbb E_{\bar\mu}\left[\mathcal H_{0,\ell_2-\ell_1}\mathcal H_{0, \ell_1-k}(\mathbf 1)\right]+ \mathbb{E}_{\bar\mu} \left[ \mathcal H_{0, \ell_1-k}(\mathbf 1) \right]
\mathbb{E}_{\bar\mu} \left[ \mathcal H_{0, \ell_2-k}(\mathbf 1) \right]
\right)\\
&\le&K'_0\sum_{1\le k<\ell_1\le\ell_2\le n}
\left(\frac{1}{(\ell_1-k)(\ell_2-\ell_1+1)}
+\frac{1}{(\ell_1-k)(\ell_2-k)}\right)
\end{eqnarray*}
for some $K'_0>0$ due to Theorem \ref{LLT0}, since 
$\mathbb E_{\bar\mu}[\cdot]$ is a continuous linear operator on $\widetilde \cB_1$
and since $\mathbf 1\in \widetilde \cB_1$.
This leads to
$\sum_{(k_1,k_2,\ell_1,\ell_2)\in E_n}
 \left|D_{k_1,\ell_1,k_2,\ell_2}\right|
=O(n(\log n)^2)$.
Analogously, we obtain
$\sum_{(k_1,k_2,\ell_1,\ell_2)\in F_n}
 \left|D_{k_1,\ell_1,k_2,\ell_2}\right|=O(n(\log n)^2)$.
Hence $A_4=o(n^2)$.

 For $A_2$, we study separately the terms 
$\bar\mu(E_{k_1,\ell_1}\cap
    E_{k_2,\ell_2})$ and the terms $\bar\mu(E_{k_1,\ell_1})\bar\mu(E_{k_2,\ell_2})$.
First by Lemma~\ref{expVn},
\begin{multline}
\sum_{1\le k_1< k_2< \ell_1<\ell_2\le n}
\bar\mu(E_{k_1,\ell_1})\bar\mu( E_{k_2,\ell_2})\\
=c_1^2\sum_{1\le k_1< k_2< \ell_1<\ell_2\le n}\left((\ell_1-k_1)^{-1}+O((\ell_1-k_1)^{-3/2})\right)\left((\ell_2-k_2)^{-1}+O((\ell_2-k_2)^{-3/2})\right)\\
=o(n^2)+c_1^2\sum_{1\le k_1< k_2< \ell_1<\ell_2\le n}\frac 1{(\ell_1-k_1)(\ell_2-k_2)}\, ,
\end{multline}
where we used the fact that
\begin{eqnarray*}
\sum_{1\le k_1<k_2<\ell_1<\ell_2\le n}\frac 1{\ell_1-k_1}
\frac 1{(\ell_2-k_2)^{\frac 32}}&\le&\sum_{m_1,m_2,m_3,m_4=1}^{n}\frac 1{m_2+m_3}\frac 1{(m_3+m_4)^{\frac 32}}\\
&\le&n\sum_{m_3=1}^n\sum_{m_2=1}^n\frac 1{m_2+m_3}\sum_{m_4=1}^{n}\frac 1{(m_3+m_4)^{\frac 32}}\\
&=&O\left( n\sum_{m_3=1}^n\log n \, m_3^{-\frac 12}\right)\\
&=& O(n^{\frac 32}\log n)=o(n^2)\, .
\end{eqnarray*}
Therefore, due to the Lebesgue dominated convergence theorem, we obtain
\begin{eqnarray}
\sum_{1\le k_1< k_2< \ell_1<\ell_2\le n}
\bar\mu(E_{k_1,\ell_1})\bar\mu( E_{k_2,\ell_2})
&=&o(n^2)+c_1^2 n^2\int_{\frac 1n\le\frac{\lfloor nx\rfloor}n<\frac{\lfloor ny\rfloor}n<\frac{\lceil nz\rceil}n<\frac{\lceil nt\rceil}n\le 1}\frac{dxdydzdt}{\left(\frac{\lceil nz\rceil}n-\frac{\lfloor nx\rfloor}n\right)
\left(\frac{\lceil nt\rceil}n-\frac{\lfloor ny\rfloor}n\right)} \nonumber \\
&\sim& c_1^2 n^2\int_{0<x<y<z<t<1}\frac{dxdydzdt}{(z-x)(t-y)} \nonumber \\
&=&  c_1^2\frac{\pi^2}{12}n^2
= \frac {n^2}{48\det \Sigma^2}\left(\sum_{a=1}^I  
\bar\mu(\mathcal I_0=a)^2\right)^2.   \label{eq:first A2}
\end{eqnarray}
The rest of the estimate of $A_2$ is new (it is different from \cite{soazSelfInter}). 
Fix for the moment $1\le k_1< k_2< \ell_1<\ell_2\le n$. Note that
\begin{multline*}
\bar\mu(E_{k_1,\ell_1}\cap E_{k_2,\ell_2})\\
=\sum_{a,b=1}^I\bar\mu\left(\bar T^{-k_1}\bar O_a\cap\bar T^{-k_2}\bar O_b\cap\bar T^{-\ell_1}(\bar O_a)\cap\bar T^{-\ell_2}\bar O_b\cap\{S_{k_2}-S_{k_1}=-(S_{\ell_1}-S_{k_2})=S_{\ell_2}-S_{\ell_1}\}\right)\, .
\end{multline*}
Using now \eqref{EQOP2} as for \eqref{EQOP3}, we observe that
$\mathbf 1_{\{S_{k_2}-S_{k_1}=-(S_{\ell_1}-S_{k_2})=S_{\ell_2}-S_{\ell_1}\}}$ is equal to the following quantity
\[
\frac 1{(2\pi)^4}\int_{([-\pi,\pi]^2)^2}e^{i u\cdot((S_{k_2}-S_{k_1})+(S_{\ell_1}-S_{k_2}))} 
      e^{i v\cdot((S_{\ell_2}-S_{\ell_1})+(S_{\ell_1}-S_{k_2}))}   \, du\, dv \,  , 
\]
which is also equal to
\[
\begin{split}
& \frac 1{(2\pi)^4}\int_{([-\pi,\pi]^2)^2}e^{i u\cdot(S_{k_2}-S_{k_1})}e^{i(u+v)\cdot(S_{\ell_1}-S_{k_2})} 
      e^{i v\cdot(S_{\ell_2}-S_{\ell_1})}   \, du\, dv \\
&  = \frac 1{(2\pi)^4}\int_{([-\pi,\pi]^2)^2} e^{i u\cdot S_{k_2 - k_1} \circ \bar T^{k_1}}
e^{i(u+v)\cdot S_{\ell_1 - k_2} \circ \bar T^{k_2}} 
      e^{i v\cdot S_{\ell_2 - \ell_1} \circ \bar T^{\ell_1}}   \, du\, dv \, . 
\end{split}
\]
Now using the $P$-invariance and $\bar T$-invariance of $\bar\mu$ and several times the formula
$P^m(f.g\circ\bar T^m)=gP^m(f)$, we obtain
$$
\bar\mu(E_{k_1,\ell_1}\cap E_{k_2,\ell_2})=
\sum_{a,b=1}^I\frac 1{(2\pi)^4}\int_{([-\pi,\pi]^2)^2}\mathbb E_{\bar\mu}
\left[ \mathbf 1_{\bar O_b} P_v^{\ell_2-\ell_1}\left(\mathbf 1_{\bar O_a}   P_{u+v}^{\ell_1-k_2}\left(\mathbf 1_{\bar O_b}P_u^{k_2-k_1}(\mathbf 1_{\bar O_a})\right)\right) \right]\, du\, dv\, .
$$
Due to our spectral assumptions, we observe that
$$P_u^n= \lambda_u^n\Pi_u + O(\alpha^n)\, ,$$
up to defining $\lambda_u=e^{-\frac 12 \Sigma^2 u\cdot u}$ for $u$ outside $[-\beta,\beta]$ and so, proceding as in the proof of Theorem \ref{LLT0}, we obtain that, for every $n\ge 2$ and every $u,v\in[-\pi,\pi]^2$,
\begin{eqnarray*}
P_u^n&=& e^{-\frac n2 \Sigma^2u\cdot u}\mathbb E_{\bar\mu}[\cdot]\mathbf 1 + O(\alpha^n)+O(e^{-2na|u|^2}(|u|+n|u|^3))\\
&=& e^{-\frac n2 \Sigma^2u\cdot u}\mathbb E_{\bar\mu}[\cdot]\mathbf 1 + O(e^{-n a|u|^2}|u|)\, ,
\end{eqnarray*}
and $|\lambda_u^n|\le e^{- 2a|u|^2}$
for some $a>0$ (such that $e^{-2a|\pi|^2}>\alpha^n$, $\max(\lambda_u^{n-1},e^{-\frac{n-1}2\Sigma^2u\cdot u})\le e^{- 2an|u|^2}$) since $n|u|^2e^{-2n a|u|^2}= O(e^{-n  a|u|^2})$.
Therefore, we obtain
\begin{multline}\label{DDDDD1}
\mathbb E_{\bar\mu}
\left[ \mathbf 1_{\bar O_b} P_v^{\ell_2-\ell_1}\left(\mathbf 1_{\bar O_a}   P_{u+v}^{\ell_1-k_2}\left(\mathbf 1_{\bar O_b}P_u^{k_2-k_1}(\mathbf 1_{\bar O_a})\right)\right) \right]\\
=(\bar\mu(\bar O_a)\bar\mu(\bar O_b))^2 e^{-\frac 12Q(\Sigma u,\Sigma v)}
+O\left((|u|+|v|)e^{-naQ(u,v)}\right)\, ,
\end{multline}
where we have set 
\begin{eqnarray*}
Q(u,v)&:=&(\ell_2-\ell_1)|v|^2+(\ell_1-k_2)|u+v|^2+(k_2-k_1) |u|^2\\
&=&(\ell_2-k_2)|v|^2+2(\ell_1-k_2)u\cdot v+(\ell_1-k_1) |u|^2\\
&=& (A_Q (u,v))\cdot (A_Q(u,v))=|A_Q(u,v)|^2\, ,
\end{eqnarray*}
with $A^2_Q:=\left(\begin{array}{cccc}
        \ell_1-k_1&0&\ell_1-k_2&0\\
         0&\ell_1-k_1&0&\ell_1-k_2\\
         \ell_1-k_2&0&\ell_2-k_2&0\\
          0&\ell_1-k_2&0&\ell_2-k_2\end{array}\right)$
which is symmetric with determinant
\begin{equation}
\label{eq:AQ det}
\begin{split}
\det A^2_Q &= (\ell_1-k_1)^2(\ell_2-k_2)^2+(\ell_1-k_2)^4-2(\ell_1-k_2)^2(\ell_1-k_1)(\ell_2-k_2)\\
&= ((k_2-k_1)(\ell_1-k_2)+(k_2-k_1)(\ell_2-\ell_1)+(\ell_1-k_2)(\ell_2-\ell_1))^2\, .
\end{split}
\end{equation}
Due to the form of $A^2_Q$, we observe that $A^2_Q$ has eigenvectors of the forms $(*,0,*,0)$ and $(0,*,0,*)$, that it has two double eigenvalues of sum (without multiplicity) $\ell_1-k_1+\ell_2-k_2$
and of product (without multiplicity) $\sqrt{\det A_Q^2}$.
Therefore its dominating eigenvalue is smaller than the sum and so is less than $4\max(k_2-k_1,\ell_1-k_2,\ell_2-\ell_1)$ and so (using the fact that the product of the two eigenvalues is larger than the maximum times the median of these three values) the smallest eigenvalue of $A^2_Q$ cannot be smaller than a quarter of the median of 
$k_2-k_1,\ell_1-k_2,\ell_2-\ell_1$, that we denote by $med (k_2-k_1,\ell_1-k_2,\ell_2-\ell_1)$.
So 
\begin{multline*}
\int_{([-\pi,\pi]^2)^2} e^{-nQ(\Sigma u,\Sigma v)}\, dudv\, 
=(\det \Sigma)^{-2}\int_{(\Sigma [-\pi,\pi]^2)^2}  e^{-nQ( u,v)}\, dudv\\
=(\det A_Q)^{-1}(\det \Sigma)^{-2}\int_{A_Q(\Sigma ([-\pi,\pi]^2)^2)} e^{-|(x,y)|^2}\, dxdy\\
=(\det A_Q)^{-1}(\det \Sigma)^{-2}\left(\int_{(\mathbb R^2)^2} e^{-|(x,y)|^2}\, dxdy 
        + O(e^{-a_1 {med (k_2-k_1,\ell_1-k_2,\ell_2-\ell_1)^2}})\right)\\
=(2\pi)^2(\det A_Q)^{-1}(\det \Sigma)^{-2}\left(1
        + O(e^{-a_1 {med (k_2-k_1,\ell_1-k_2,\ell_2-\ell_1)^2}})\right)\, ,
\end{multline*}
for some $a_1>0$. 
Moreover
\begin{multline*}
\int_{(\mathbb R^2)^2} |(u,v)|e^{-naQ(u,v)}\, dudv
=(\det A_Q)^{-1}\int_{(\mathbb R^2)^2} |A_Q^{-1}(u,v)| e^{-a|(x,y)|^2}\, dxdy\\
=O\left((\det A_Q)^{-1}\, med(k_2-k_1,\ell_1-k_2,\ell_2-\ell_1)^{-\frac 12} \right)\, .
\end{multline*}
Therefore
\begin{equation}\label{EEEEE1}
\bar\mu(E_{k_1,\ell_1}\cap E_{k_2,\ell_2})
=\frac{\left(\sum_{a=1}^I\bar\mu(\bar O_a)^2\right)^2}{(2\pi)^{2}\, \det A_Q\, \det \Sigma^{2}}\left(1+
O\left( med(k_2-k_1,\ell_1-k_2,\ell_2-\ell_1))^{-\frac 12} \right)\right)\, .
\end{equation}
But using \eqref{eq:AQ det},
\begin{multline*}
\sum_{1\le k_1< k_2< \ell_1<\ell_2\le n}
(\det A_Q)^{-1}=\sum_{1\le k_1< k_2< \ell_1<\ell_2\le n}\frac 1
{(k_2-k_1)(\ell_1-k_2)+(k_2-k_1)(\ell_2-\ell_1)+(\ell_1-k_2)(\ell_2-\ell_1)}\\
=\sum_{m_1,m_2,m_3,m_4\ge 1\ :\ m_1+m_2+m_3+m_4\le  n}\frac {1}
{m_2m_3+m_2m_4+m_3m_4}\\
=n^2\int_{(0,+\infty)^4}
      \frac {\mathbf 1_{\left\{\frac{\lceil ny_1\rceil}n+\frac{\lceil ny_2\rceil}n+\frac{\lceil ny_3\rceil}n+\frac{\lceil ny_4\rceil}n\le 1\right\}}}{\frac{\lceil ny_2\rceil}n \frac{\lceil ny_3\rceil}n +\frac{\lceil ny_2\rceil}n \frac{\lceil ny_4\rceil}n+\frac{\lceil ny_3\rceil}n\frac{\lceil ny_4\rceil}n}\, dy_1\, dy_2\, dy_3\, dy_4\\
\sim n^2\int_{(0,+\infty)^4}
      \frac {\mathbf 1_{\left\{y_1+y_2+y_3+y_4\le 1\right\}}}{y_2y_3 +y_2y_4+y_3y_4}\, dy_1\, dy_2\, dy_3\, dy_4\, ,
\end{multline*}
due to the dominated convergence theorem.
Therefore
\begin{equation}\label{EEEEE2}
\sum_{1\le k_1< k_2< \ell_1<\ell_2\le n}
(\det A_Q)^{-1}\sim n^2\int_{(0,+\infty)^3}
      \frac {(1-y_2-y_3-y_4)\mathbf 1_{\left\{y_2+y_3+y_4\le 1\right\}}}{y_2y_3 +y_2y_4+y_3y_4}\, dy_2\, dy_3\, dy_4=n^2J\, .
\end{equation}
Analogously
\begin{equation}
\begin{split}
\label{EEEEE3}
\sum_{1\le k_1< k_2< \ell_1<\ell_2\le n} &
(\det A_Q)^{-1}\, (med(k_2-k_1,\ell_1-k_2,\ell_2-\ell_1))^{-\frac 12} \\
& =\sum_{m_1,m_2,m_3,m_4\ge 1\ :\ m_1+m_2+m_3+m_4\le  n}\frac {1}
{(m_2m_3+m_2m_4+m_3m_4)\, med(m_2,m_3,m_4)^{\frac 12}}\\
& \le n\sum_{1\le m_2\le m_3\le m_4\le  n}\frac {1}
{(m_2m_3+m_2m_4+m_3m_4)\, m_3^{\frac 12}}\\
& \le  n\sum_{1\le m_2\le m_3\le m_4\le  n}\frac {1}
{ m_3^{\frac 32}m_4}
\; \le \; n \log n \sum_{m_2=1}^n\sum_{m_3=m_2}^n m_3^{-\frac 32}\\
& \le n \log n \sum_{m_2=1}^n O( m_3^{-\frac 12})=O(n^{\frac 32}\log n)=o(n^2)\, .
\end{split}
\end{equation}
Equations \eqref{EEEEE1}, \eqref{EEEEE2} and \eqref{EEEEE3} lead to
$$
\sum_{1\le k_1< k_2< \ell_1<\ell_2\le n}\bar\mu(E_{k_1,\ell_1}\cap E_{k_2,\ell_2})=\frac{\left(\sum_{a=1}^I\bar\mu(\bar O_a)^2\right)^2}{((2\pi)^2\det\Sigma^2)}J+o(n^2)\, .
$$
Combining this with \eqref{eq:first A2}, we conclude that
\begin{equation}
\label{eq:A2}
A_2\sim \frac{n^2}{\det\Sigma^2}\left(\sum_{a=1}^I\bar\mu(\mathcal I_0=a)^2\right)^2\left(\frac{-1}{48}+\frac J{4\pi^2}
\right)\, .
\end{equation}

 The study of 
$A_3$ is the most delicate. We can observe that both
sums
$\sum_{1\le k_1< k_2<\ell_2<\ell_1\le n}
    \bar\mu(E_{k_1,\ell_1}\cap
    E_{k_2,\ell_2})$ and $\sum_{1\le k_1< k_2<\ell_2<\ell_1\le n}\bar\mu(E_{k_1,\ell_1})\bar\mu(E_{k_2,\ell_2})$
are in $O(n^2\log n)$. However, we will see that their difference
is in $n^2$. Once again our proof differs from the one in \cite{soazSelfInter} 
and is based on the same idea as the one used to prove $A_2$.
We set $E_{k,\ell}(b):=E_{k,\ell}\cap\{\mathcal I_k=b\}$. 
Due to the first part of Lemma \ref{expVn},
\begin{align}
A_3&=\sum_{1\le k_1< k_2< \ell_2 < \ell_1\le n}
\bar\mu(E_{k_1,\ell_1}\cap E_{k_2,\ell_2})
-\bar\mu(E_{k_1,\ell_1})\bar\mu(E_{k_2,\ell_2})\nonumber\\
&=o(n^{2})+\sum_{1\le k_1< k_2< \ell_2 <    \ell_1\le n}
\sum_{a,b=1}^I\left(-I_{k_1,k_1,l_1,l_2}+
\bar\mu(O_{k_1,k_2,l_1,l_2}\cap S_{k_1,k_2,l_1,l_2})\right)\nonumber\\
&=o(n^{2})+\sum_{1\le k_1< k_2< \ell_2 <  \ell_1\le n}
\sum_{a,b=1}^I\left(-I_{k_1,k_1,l_1,l_2}\right)\\
&+\sum_{1\le k_1< k_2< \ell_2 <  \ell_1\le n}
\sum_{a,b=1}^I\left(\frac 1{(2\pi)^2}\int_{[-\pi,\pi]^2}\mathbb E_{\bar\mu}\left[
 \mathbf 1_{\bar O_a}P_u^{\ell_1-\ell_2}\left(\mathbf 1_{\bar O_b}\mathcal H_{0,\ell_2-k_2}\left(\mathbf 1_{\bar O_b}P_u^{k_2-k_1}\left(\mathbf 1_{\bar O_a}\right)\right)\right)\right]\, du\right)\, ,\label{AAAAA1}
\end{align}
where $$I_1(k_1,k_1,l_1,l_2)=\frac{(\bar\mu(\bar O_a))^2\bar\mu(E_{k_2,\ell_2}(b))} {2\pi\sqrt{\det \Sigma^2}(\ell_1-k_1)},$$
$$O_{k_1,k_2,l_1,l_2}=\bar O_a\cap \bar T^{-(k_2-k_1)}\bar O_b\cap\bar T^{-(\ell_2-k_1)}\bar O_b \cap\bar T^{-(\ell_1-k_1)}\bar O_a,$$$$S_{k_1,k_2,l_1,l_2}= \{S_{\ell_2-k_2}\circ\bar T^{k_2-k_1}=0\}
\cap\{S_{\ell_1-\ell_2}\circ \bar T^{\ell_2-k_1}=-S_{k_2-k_1}\}$$
Now, as we did for \eqref{DDDDD1} (and using Theorem \ref{LLT0}), we get that
\begin{multline*}
\mathbb E_{\bar\mu}\left[
 \mathbf 1_{\bar O_a}P_u^{\ell_1-\ell_2}\left(\mathbf 1_{\bar O_b}\mathcal H_{0,\ell_2-k_2}\left(\mathbf 1_{\bar O_b}P_u^{k_2-k_1}\left(\mathbf 1_{\bar O_a}\right)\right)\right)\right]\\
=(\bar\mu(\bar O_a))^2 e^{-\frac {(\ell_1-\ell_2)+(k_2-k_1)}2|\Sigma u|^2}\mathbb E_{\bar\mu}\left[\mathbf 1_{\bar O_b}\mathcal H_{0,\ell_2-k_2}\mathbf 1_{\bar O_b}\right]
+O\left(\frac{|u|}{\ell_2-k_2}e^{-na|u|^2}\right)\, .
\end{multline*}
Therefore
\begin{multline}\label{AAAAA2}
\frac 1{(2\pi)^2}\int_{(-\pi,\pi)^2}\mathbb E_{\bar\mu}\left[
 \mathbf 1_{\bar O_a}P_u^{\ell_1-\ell_2}\left(\mathbf 1_{\bar O_b}\mathcal H_{0,\ell_2-k_2}\left(\mathbf 1_{\bar O_b}P_u^{k_2-k_1}\left(\mathbf 1_{\bar O_a}\right)\right)\right)\right]\, du\\
=  \frac{(\bar\mu(\bar O_a))^2{\bar\mu}\left(
E_{k_2,\ell_2}(b)\right)}
      {2\pi(\ell_1-\ell_2+k_2-k_1)\sqrt{\det \Sigma ^2}}
+O\left(\frac 1{(\ell_2-k_2)(\ell_1-\ell_2+k_2-k_1)^{\frac 32}}\right)\, .
\end{multline}
We will now prove that the term in $O$ in this last formula is negligable. Indeed its sum over $\{1\le k_1\le k_2\le \ell_2\le\ell_1\le n\}$ is in $O$ of the following quantity:
\begin{eqnarray*}
\sum_{m_1+m_2+m_3+m_4\le n}\left(\frac 1{m_3(m_4+m_2)^{\frac 32}}\right)
&\le&n\log n\sum_{m_2=1}^n\sum_{m_4=1}^n(m_4+m_2)^{-\frac 32}\\
&\le&O\left(n\log n\sum_{m_2=1}^nm_4^{-\frac 12}\right)=O(n^{\frac 32}\log n)=o(n^2)\, .
\end{eqnarray*}
This combined with \eqref{AAAAA1} and \eqref{AAAAA2} leads to
$$ A_3=o(n^{2})+\sum_{1\le k_1< k_2< \ell_2\le\ell_1\le n}
\sum_{a,b=1}^I\frac{(\bar\mu(\bar O_a))^2\bar\mu\left(
E_{k_2,\ell_2}(b)\right)}
      {2\pi\sqrt{\det \Sigma ^2}}\left(\frac 1{\ell_1-\ell_2+k_2-k_1}-\frac{1}
  {\ell_1-k_1}\right)\, ,$$
i.e.
\begin{multline*}
A_3
=o(n^{2})+\frac{\sum_{a}^I(\bar\mu(\mathcal I_0=a))^2}{2\pi\sqrt{\det \Sigma^2}}\sum_{m_1+m_2+m_3+m_4\le n}\left(\frac{c_1}{m_3}+O(m_3^{-\frac 32})\right)\frac{m_3} {(m_2+m_4)(m_2+m_3+m_4)}\\
=o(n^{2})+c_1^2\sum_{m_1+m_2+m_3+m_4\le n}\frac{1} {(m_2+m_4)(m_2+m_3+m_4)}\, ,
\end{multline*}
since 
$$ \sum_{m_1+m_2+m_3+m_4\le n}\frac{1} {m_3^{\frac 12}(m_2+m_4)(m_2+m_3+m_4)}=O\left(n\sum_{m_2,m_3,m_4=1}^nm_3^{-\frac 12}(m_2m_4)^{-1}\right)=o(n^2)\, .$$
Therefore, due to the Lebesgue dominated convergence theorem,
\begin{multline*}
A_3
\sim n^{2}c_1^2\int_{y_1,y_2,y_3,y_4>0:y_1+y_2+y_3+y_4<1}
\frac{1} {(y_2+y_4)(y_2+y_3+y_4)}\, dy_1dy_2dy_3dy_4
\sim \frac {c_1^2}2 n^{2} .\\
\end{multline*}

To conclude the proof of the lemma, we use the estimate for 
$A_3$
together with
\eqref{eq:Var} and \eqref{eq:A2} to obtain,
\begin{align*}
8A_2 + 8A_3 &= 4c_1^2 n^2 + \frac{8n^2}{\det \Sigma^2} \left( \sum_{a=1}^I \bar \mu(\bar O_a)^2 \right)^2
\left( \frac{-1}{48} + \frac{J}{4\pi^2} \right) \\
&= \frac{n^2}{\det \Sigma^2} \left( \sum_{a=1}^I \bar \mu(\bar O_a)^2 \right)^2
\left[ \frac{2J+1}{\pi^2} - \frac 16 \right].
\end{align*}
This finished the proof.

\section{Spectrum of $\cP_u$}\label{eigenvalue}
In this appendix, we are interested in the spectrum of the family of
operators $\cP_u$.
We start by stating a result for the unperturbed operators $\mathcal L_{u,0}$.
\begin{lem}\label{lem:nonarithm}
Let $u\in \mathbb R^2$, $h\in \cB$ and $\lambda\in \mathbb C$ 
be such that 
$\mathcal L_{u,0}h=\lambda h$ in ${\mathcal B}$ and $|\lambda| \ge 1$.
Then either $h\equiv 0$ or $u\in 2\pi\mathbb Z^2$, $\lambda =1$ and $h$ is $\bar \mu_0$-almost surely constant.
\end{lem}
\begin{proof}
Recall that for $\psi \in \cC^p(\bar M_0)$, we have $\psi \circ \bar T_0^n \in \cC^p(\bar T^{-n}\cW^s)$.   Note that
$$\mathcal L_{u,0} h(\psi)=h(e^{iu \cdot \Phi_0}\psi\circ \bar T_0).$$
 Thus for $n\geq 1$,
 $$\mathcal L_{u,0}^n h(\psi)= h(e^{iu \cdot S_n \Phi_0 }\psi\circ \bar T_0^n),$$
 where $S_n \Phi_0 =\Phi_0+\Phi_0\circ \bar T_0+\cdots+\Phi_0\circ \bar T_0^{n-1}$ denotes the partial sum.
By \cite[Lemma~3.4]{MarkHongKun2013}, using the invariance of $h$,
\begin{equation}
\label{eq:measure}
|h(\psi)| = |\lambda|^{-n} |h(e^{iu \cdot S_n\Phi_0}\psi\circ \bar T_0^n) |
\leq C |\lambda|^{-n} | h |_w \big(|e^{i u \cdot S_n\Phi_0} \psi \circ \bar T_0^n|_\infty + C^{(p)}_{\bar T_0^{-n} \cW^s}(e^{iu\cdot S_n\Phi_0}\cdot\psi \circ \bar T_0^n) \big) ,
\end{equation}
where $C^{(p)}_{\bar T_0^{-n}\cW^s}(\cdot)$ denotes the H\"older constant of exponent 
$p$ 
measured along elements of $\bar T_0^{-n} \cW^s$.  Since $|e^{i u \cdot S_n  \Phi_0}| = 1$
and $S_n \Phi_0$ is constant on each element of $\bar T_0^{-n} \cW^s$, we have
\[
\begin{split}
C^{(p)}_{\bar T_0^{-n} \cW^s}(e^{iu\cdot S_n\Phi_0}\cdot\psi \circ 
\bar T_0^n) 
& \le |e^{i u \cdot S_n \Phi_0}|_\infty C^{(p)}_{\bar T_0^{-n} \cW^s}(\psi \circ \bar T_0^n) 
+ |\psi \circ \bar T_0^n|_\infty C^{(p)}_{\bar T_0^{-n}\cW^s}(e^{i u \cdot S_n \Phi_0}) \\
& \le C \Lambda^{-pn} C^{(p)}_{\cW^s}(\psi) .
\end{split}
\]
Using this estimate in \eqref{eq:measure} and taking the limit as $n \to \infty$ yields 
$|h(\psi)| = 0$ if $|\lambda|>1$ and $|h(\psi)| \leq C| h |_w |\psi|_\infty$ for all $\psi \in \cC^p(\cW^s)$
if $|\lambda|=1$.  From this we conclude that the spectrum of $\cL_{u,0}$ is always
contained in the unit disk.  Furthermore, when $|\lambda|=1$, then
$h$ is a signed measure.  For the remainder of the proof, we assume $|\lambda|=1$.

 Let $\mathbb{V}_{u,0}$ be the eigenspace of $\mathcal L_{u,0}$ corresponding to eigenvalue $\lambda_{u,0}$, and $\Pi_{u,0}$ the eigenprojection operator. 
Since we are assuming $\mathbb{V}_{u,0}$ is non-empty, Lemma~\ref{lem:uniform ly} implies
that $\cL_{u,0}$ is quasi-compact with essential spectral radius bounded by 
$\tau < 1$.
Moreover, Lemma~\ref{lem:uniform ly} implies that $\| \cL_{u,0}^n \|_{L(\cB, \cB)}$ remains bounded
for all $n \ge 0$, so using \cite[Lemma~5.1]{MarkHongKun2011}, we conclude that
$\cL_{u,0}$ has no Jordan blocks corresponding to its peripheral spectrum. 
 
Using these facts,  $\Pi_{u,0}$ has the representation
$$\lim_{n\to\infty}\frac{1}{n} \sum_{j=1}^n \lambda^{-j} \mathcal L_{u,0}^j =\Pi_{u,0} . $$
In addition, for $f\in \cC^1(\bar M_0)$, $\psi\in \cC^p(\cW^s)$,
$$
\left| \Pi_{u,0} f(\psi) \right| =\left| \lim_{n\to\infty}\frac{1}{n} \sum_{j=1}^n \lambda^{-j} 
f((e^{iu \cdot S_j \Phi_0}\psi\circ \bar T_0^j) \right| \leq |f|_{\infty}|\psi|_{\infty} .
$$

Since $\Pi_{u,0} \cC^1(\bar M_0)$ is dense in the finite dimensional space $\Pi_{u,0}\cB$, therefore $\Pi_{u,0} \cC^1(\bar M_0)=\Pi_{u,0}\cB=\mathbb{V}_{u,0}$.  So for  $h \in \mathbb{V}_{u,0}$, there
exists $f \in \cC^1(\bar M_0)$ such that $\Pi_{u,0} f = h$.  Now for each $\psi \in \cC^p(\bar M_0)$,
\[
|h(\psi)| = |\Pi_{u,0} f(\psi)| \leq |f|_\infty \Pi_0 1(|\psi|) = |f|_\infty \bar \mu_0(|\psi|) .
\]
Thus $h$ is absolutely continuous with respect to $\bar \mu_0$. For simplicity, we identify $h$ and its density with respect to $\bar\mu_0$; then $h \in L^\infty(\bar M_0, \bar\mu_0)$.
Now for any $\psi\in \cC^p(\cW^s)$, we have
\begin{align*}
\lambda\int_{\bar M_0} h \psi\, d\mu_0&=\int_{\bar M_0} \cL_0(e^{iu \cdot \Phi_0} h)\cdot \psi\, d\bar\mu_0\\
&=\int_{\bar M_0} (e^{iu\cdot \Phi_0} h)\circ \bar T_0^{-1}\cdot \psi\,d\bar \mu_0
\end{align*}
Accordingly, $\lambda\,  h=(e^{iu \cdot \Phi_0} h)\circ \bar T_0^{-1}$, $\bar\mu_0$-a.e. Or equivalently, we have $\lambda\, h\circ \bar T_0=e^{iu \cdot \Phi_0}h$.
Hence $\lambda^n\, h\circ \bar T_0^n=e^{iu\cdot S_n \Phi_0}h$.

Let $G_\lambda$ be the closed multiplicative group generated by $\lambda$ and let $m_{\lambda}$ be the normalized Haar measure on $G_\lambda$.
($G_\lambda$ is finite if $\lambda$ is a root of unity; 
it is $\{z\in\mathbb C\, :\, |z|=1\}$ otherwise.) 
The dynamical system $(G_\lambda,m_{\lambda},T_\lambda)$
is ergodic, where $T_\lambda$ denotes multiplication by $\lambda$ in $G_\lambda$.
Due to \cite{FPCRAS},
the dynamical system $(M_0 \times G_\lambda,\mu_0\otimes m_{\lambda},
T_0 \times T_\lambda)$
in infinite measure is conservative and ergodic. But
the function $H: M_0 \times G_\lambda \rightarrow\mathbb C$ defined as follows is
$(T_0 \times T_\lambda)$-invariant:
\[
\forall (\bar x,\ell,y)\in\bar M_0\times\mathbb Z^2\times G_\lambda ,\quad 
H(\bar x+\ell,y):=y h(\bar x)e^{-i u\cdot \ell}.
\]
Indeed, for $\mu_0\otimes m_{\lambda}$-a.e. $(\bar x+\ell,y)\in M_0 \times G_\lambda$,
\begin{eqnarray*}
H((T_0 \times T_\lambda)(\bar x+\ell,y))
&=&H(\bar T_0(\bar x)+\ell+\Phi_0(\bar x),\lambda y)
=\lambda y h(\bar T_0(\bar x))e^{-i u\cdot (\ell+\Phi_0(\bar x))}\\
&=& ye^{-i u\cdot \ell}(\lambda h(\bar T_0(\bar x))
   e^{-i u\cdot \Phi_0(\bar x)})\\
&=& ye^{-i u\cdot \ell} h(\bar x)\, ,
\end{eqnarray*}
due to our assumption on $h$.
We conclude that $H$ is a.e. equal to a constant, which implies that
$u\in 2\pi\mathbb Z^2$, $\lambda=1$, and $h$ is $\bar\mu_0$-a.s. constant.
\end{proof}
\begin{prop}\label{lem:nonarithm2}
Given $\beta > 0$, there exists $C>1$ and $\alpha\in(0,1)$ such that 
$$\forall n\in\mathbb N^*,\quad \sup_{\beta \le |u|\le\pi }\Vert \mathcal P_u^n\Vert_{L(\widetilde\cB,\widetilde\cB)}\le C\alpha^n\, .$$
\end{prop}
\begin{proof}
Fix $\beta > 0$.
Due to  \cite[Lemma 4.3]{AaronsonDenker},  Lemma~\ref{lem:nonarithm}, 
and the
continuity in $u$ provided by \cite[Lemma~5.4]{MarkHongKun2014} (see also Lemma~\ref{analytic}
applied to $\cL_{u,0}$ rather than $P_u$), we know that there exists $C>1$ and $\alpha\in(0,1)$ such that 
$$\forall n\in\mathbb N^*,\quad \sup_{\beta \le |u|\le \pi}\Vert \mathcal L_{u,0}^n\Vert_{L(\cB,\cB)}\le C\alpha^n\, .$$
Therefore, for every $f\in\widetilde\cB$, we have
\begin{eqnarray*}
\sup_{\uomega\in E^{\mathbb N}}\left\Vert   \mathcal P_u^nf(x,\uomega)\right\Vert_{\cB}&=&\sup_{\uomega\in E^{\mathbb N}}\left\Vert  \int_{E^n}\mathcal L_{u,0}^nf(\cdot,(\tilde\omega,\uomega))\,
    d\eta^{\otimes n}(\tilde\omega)   \right\Vert_{\cB}\\
&\le&\sup_{\uomega\in E^{\mathbb N}} \int_{E^n}\left\Vert \mathcal L_{u,0}^nf(\cdot,(\tilde\omega,\uomega))\right\Vert_{\cB}\,
    d\eta^{\otimes n}(\tilde\omega)\\
&\le&  \sup_{\uomega\in E^{\mathbb N}} C\alpha^n\sup_{\underline{\omega'}}\left\Vert f(\cdot,\underline{\omega'})\right\Vert_{\cB}\, .
\end{eqnarray*}
where we used Lemma \ref{tildeBnorm} to obtain the second line.
Analogously,
\begin{eqnarray*}
\sup_{\uomega\ne \underline{\omega'}}\frac{\left\Vert   \mathcal P_u^nf(x,\uomega)-\mathcal P_u^nf(x,\underline{\omega'})\right\Vert_{\cB}}{d(\uomega,\underline{\omega'})}&=&\sup_{\uomega\ne \underline{\omega'}}\frac{\left\Vert  \int_{E^n}\mathcal L_{u,0}^n\left(f(\cdot,(\tilde\omega,\uomega))-f(\cdot,(\tilde\omega,\underline{\omega'}))\right)\,    d\eta^{\otimes n}(\tilde\omega)   \right\Vert_{\cB}}{d(\uomega,\underline{\omega'})}\\
&\le&\sup_{\uomega\ne \underline{\omega'}} \int_{E^n}\frac{\left\Vert \mathcal L_{u,0}^n\left(f(\cdot,(\tilde\omega,\uomega))-f(\cdot,(\tilde\omega,\underline{\omega'}))\right)\right\Vert_{\cB}}{d(\uomega,\underline{\omega'})}\,
    d\eta^{\otimes n}(\tilde\omega)\\
&\le&  C\alpha^n\varkappa^n\sup_{\uomega\ne \underline{\omega'}} \frac{\left\Vert f(\cdot,\underline{\omega'})-f(\cdot,\underline{\omega'})\right\Vert_{\cB}}{d(\uomega,\underline{\omega'})}\, .
\end{eqnarray*}
We conclude by putting these two estimates together.
\end{proof}
\end{appendix}
\paragraph{\bf Acknowledgements.}
This work was begun at the AIM workshop {\em Stochstic Methods for Non-Equilibrium Dynamical Systems}, in June 2015.  Part of this work was carried out during visits by the authors 
to ESI, Vienna in 2016, to CIRM, Luminy in 2017 and 2018, and to BIRS, Canada in 2018, 
and by a visit of FP to the University of Massachusetts at Amherst in 2018.
MD was supported in part by NSF Grant DMS 1800321.

\end{document}